\documentclass[final]{siamart}
\usepackage[margin=1in]{geometry}
\usepackage{amsmath,amssymb,amsfonts,stmaryrd}
\usepackage{soul}
\usepackage{array}
\usepackage{empheq}
\usepackage{xfrac}
\usepackage{minibox}
\usepackage{blkarray}
\usepackage{enumitem}
	\setlist{nosep} \usepackage{color}
\usepackage[numbers,sort]{natbib}
\usepackage{hyperref}
\usepackage{arydshln}
\newsiamremark{remark}{Remark}
\newsiamremark{example}{Example}
\newsiamremark{conjecture}{Conjecture}

\DeclareMathOperator{\Rea}{\textnormal{Re}}
\DeclareMathOperator{\Ima}{\textnormal{Im}}

\allowdisplaybreaks

\newcommand{\TheTitle}{Necessary Conditions and Tight Two-level Convergence\\Bounds for Parareal and Multigrid Reduction in Time}
\newcommand{\TheAuthors}{Ben S. Southworth}
\headers{Necessary and Sufficient Conditions for Convergence of Parareal and MGRiT}{\TheAuthors}
\title{{\TheTitle}\thanks{This work was  performed under the auspices of the U.S. Department of Energy under 
grant number (NNSA) DE-NA0002376 and Lawrence Livermore National Laboratory under 
contracts B614452 and B627942.}}

\author{  Ben~S.~Southworth
  \thanks{Department of Applied Mathematics,
          University of Colorado at Boulder
          (\email{ben.s.southworth@gmail.com}).}
}

\ifpdf\hypersetup{  pdftitle={\TheTitle},
  pdfauthor={\TheAuthors}
}
\fi

\begin{document}
\maketitle

\begin{abstract}
Parareal and multigrid reduction in time (MGRiT) are two of the most popular parallel-in-time methods. The basic idea is to
treat time integration in a parallel context by using a multigrid method in time. If $\Phi$ is the (fine-grid) time-stepping scheme
of interest, such as RK4, then let $\Psi$ denote a ``coarse-grid'' time-stepping scheme chosen to approximate $k$ steps of
$\Phi$, where $k\geq 1$. In particular, $\Psi$ defines the coarse-grid correction, and evaluating $\Psi$ should be (significantly)
cheaper than evaluating $\Phi^k$. Parareal is a two-level method with a fixed relaxation scheme, and MGRiT is a
generalization to the multilevel setting, with the additional option of a modified, stronger relaxation scheme. 

A number of papers have studied the convergence of Parareal and MGRiT. However, there have yet to be general conditions
developed on the convergence of Parareal or MGRiT that answer simple questions such as, (i) for a given $\Phi$ and $k$,
what is the best $\Psi$, or (ii) can Parareal/MGRiT converge for my problem? This work derives necessary and sufficient
conditions for the convergence of Parareal and MGRiT applied to linear problems, along with tight two-level convergence bounds,
under minimal additional assumptions on $\Phi$ and $\Psi$. Results all rest on the introduction of a \textit{temporal approximation property}
(TAP) that indicates how $\Phi^k$ must approximate the action of $\Psi$ on different vectors. Loosely, for unitarily diagonalizable
operators, the TAP indicates that the fine-grid and coarse-grid time integration schemes must integrate geometrically
smooth spatial components similarly, and less so for geometrically high frequency. In the (non-unitarily) diagonalizable setting,
the conditioning of each eigenvector, $\mathbf{v}_i$, must also be reflected in how well $\Psi\mathbf{v}_i \sim\Phi^k\mathbf{v}_i$. 
In general, worst-case convergence bounds are exactly given by $\min \varphi < 1$ such that an inequality along the lines of
$\|(\Psi-\Phi^k)\mathbf{v}\| \leq\varphi \|(I - \Psi)\mathbf{v}\|$ holds for all $\mathbf{v}$. Such inequalites are formalized as 
different realizations of the TAP in Section 2, and form the basis for convergence of MGRiT and Parareal applied
to linear problems.
\end{abstract}

\section{Introduction} \label{sec:intro}

Efficiently distributing computational work over many processors, or parallelizing, is fundamental to running large-scale
numerical simulations. In the case of partial differential equations (PDEs) in space and time, problems are at least 3-4
dimensional, and as many as seven or more for problems such as radiative transport. Additionally accounting for
multiple variables that may have to be solved for, even a moderate number of points in each dimension requires a massive
number of unknowns, as well as a high level of parallelism, to obtain an accurate solution. Furthermore, computational
power is largely increasing in the number of processors available and less in the power of individual processors, making
increased parallelism an important area of research.

Steady state PDEs (non-time-dependent) are typically posed as boundary value problems (BVPs), which provide
a natural mechanism to parallelize in space. When time derivates are introduced, adding parallelism in the time
dimension is more complicated. In particular, for time-dependent
PDEs, it is often the case that only an initial value in time is given. To that end, propagating information through
the temporal domain appears to be an inherently sequential process because the initial information can only
be propagated in one direction, namely forward in time. This is how most time-dependent PDEs are solved -- 
given some initial value in time, a BVP is formulated and discretized in the spatial domain. The IVP is then
propagated forward one time step by solving the BVP and applying some time integration routine, and
the process repeats based on a new ``initial value'' in time. However, as the number of processors available
to run numerical simulations has increased, so has the interest in so-called parallel-in-time methods,
which are designed to parallelize the process of integrating forward in time.

Because time integration typically involves solving for a solution at a set of discrete ``time points,'' it can
be represented in block matrix form, where 
\begin{align}\label{eq:system}
A\mathbf{u} = \begin{bmatrix} I \\ -\Phi & I \\ & -\Phi & I \\ & & \ddots & \ddots\end{bmatrix}
	\begin{bmatrix} \mathbf{u}_0 \\ \mathbf{u}_1\\ \mathbf{u}_2 \\ \vdots \end{bmatrix} = \mathbf{f}.
\end{align}
Here, $\mathbf{u}_i$ is the solution at the $i$th time point, $\mathbf{f}$ the right-hand side, and $\Phi$
some invertible operator that progresses the solution from time $t_i$ to $t_{i+1}$. In this setting, classical
(sequential) time integration can be seen as a direct (forward) solve of \eqref{eq:system}. Parallel-in-time methods
can typically be posed as some form of preconditioner or iterative method to solve \eqref{eq:system}. This
introduces new questions on the convergence of such iterations, which do not arise in the sequential setting.
Interestingly, although a lower bi-diagonal matrix is one of the easiest linear systems to solve in serial
though a forward solve, solving lower triangular matrices remains difficult in the parallel setting. 

Many parallel-in-time methods have been proposed, with varying levels of success. Some of the prominent
methods include full space-time multigrid \cite{FalgoutFriedhoffKolevMaclachlanSchroderVandewalle2017,HortonVandewalle1995},
parallel full approximation scheme in space and time (PFASST) \cite{Emmett:2012du},
Parareal \cite{Parareal}, and multigrid reduction in time (MGRiT) \cite{FalgoutFriedhoffKolevMaclachlanSchroder2014}.
Parareal is perhaps the most well-known and one of the original ideas for parallel-in-time integration. Parareal is
effectively a two-level multigrid method. Time points are partitioned into C-points and F-points, and relaxation
consists of integrating each C-point $k-1$ time steps forward, based on its current solution (that is, to the final F-point
preceding the next C-point); here $k$ denotes the coarsening factor. This is followed by a coarse-grid correction,
which approximately inverts the Schur complement of \eqref{eq:system}. In particular, $k$ steps on the fine grid,
$\Phi^k$, are approximated by some operator $\Psi$ that is cheaper to evaluate.
The simplest example is letting $\Psi$ be the same time-integration scheme as $\Phi$, using time steps that are
$k$ times larger. The multigrid reduction in time (MGRiT) algorithm generalizes this to the multilevel setting, by
recursively coarsening the temporal grid until it is sufficiently small to solve directly at minimal cost. 

Analysis of Parareal dates back to \cite{Bal:2005cw}, where Parareal is analyzed from a time-integration
perspective, looking at the stability and error of Parareal as a time-integration routine with respect to the continuous problem.
In \cite{Gander:2007bdb}, the connection between Parareal and a two-level multigrid algorithm with F-relaxation
is discussed, and initial bounds developed for Parareal that are, in some cases, sharp. An analysis of the nonlinear
case is developed in \cite{Gander:2008bt}, largely demonstrating that Parareal is applicable to nonlinear problems, and
the convergence of Parareal applied to elasticity and plasma simulations is discussed, respectively, in
\cite{HessenthalerNordslettenRoehrleSchroderFalgout2018,ReynoldsBarredo:2012bl}. 
More recently, \cite{Wu:2015ht,Wu:2015jv} analyze several specific time integration schemes applied to
problems of the form $\mathbf{u}_t = \mathcal{L}\mathbf{u} + \mathbf{g}$, where $\mathcal{L}$ is symmetric negative
definite. Some of the results are tight, but are indeed limited to specific time-integration schemes
and symmetric negative definite spatial discretizations. 
A detailed analysis of error propagation of two-level MGRiT and Parareal is developed in
\cite{DobrevKolevPeterssonSchroder2017}, under the assumption that fine- and coarse-grid time-stepping operators
commute and are diagonalizable. Results in \cite{DobrevKolevPeterssonSchroder2017} are, to some
extent, a generalization of \cite{Gander:2007bdb}, and also introduce FCF-relaxation to the analysis,
a variation in MGRiT that is not used in Parareal, but that can be important for convergence on difficult
problems. Numerical results in \cite{DobrevKolevPeterssonSchroder2017} demonstrate the derived
bounds appear to be tight when applied to a number of model problems. In fact, the framework developed
in this paper is a substantial generalization of that in \cite{DobrevKolevPeterssonSchroder2017}. 
One corollary proven here is that, for a certain class of problems, the bounds developed in
\cite{DobrevKolevPeterssonSchroder2017} are indeed exact to $O(1/N_c)$, where $N_c$ is the number
of time points on the coarse grid. Relaxation is generalized further in the recent paper \cite{Gander:2018gj}
using an algebraic perspective, similar to the framework used in this paper. {The framework developed in
\cite{DobrevKolevPeterssonSchroder2017} is also extended to the multilevel setting in \cite{MGRiT19}, under
the assumption of diagonalizable operators}.\footnote{{Results and the framework developed in this paper are
also extended in \cite{MGRiT19} to discuss the difficulties that a multilevel method presents over
the two-level setting.}} Finally, in \cite{Friedhoff:2015gt}, local mode
analysis techniques are generalized and applied to parabolic parallel-in-time problems, providing more accurate
estimates of convergence over traditional techniques. 

Despite a number of works analyzing Parareal and MGRiT, there remains a gap in the literature in answering the
fundamental question, for a general problem, what must a coarse-grid time-stepping scheme satisfy to see convergent
behavior? The main contribution of this paper is the development of, under minimal assumptions, necessary and sufficient
conditions for the convergence in norm of linear Parareal and two-level MGRiT. A simple \textit{temporal approximation property} (TAP)
is introduced that measures how accurately the fine-grid time-integration scheme approximates the coarse-grid
integration scheme. This leads to necessary and sufficient conditions for convergence of error propagation in
the $\ell^2$- and $A^*A$-norms, including tight bounds on convergence. Additional results are established under
further assumptions on the time-integration operators being diagonalizable and unitarily diagonalizable. Indeed,
if we assume that the spatial discretization is symmetric and definite as in \cite{Wu:2015jv,Wu:2015ht}, results
here provide exact bounds on convergence for arbitrary time-integration schemes. For the most general results, 
the only assumption is that the problem is linear. Most results also require that the same operator $\Phi$ integrates the
solution at all time-points (for example, there can be no time-dependent differential spatial coefficients). In all
cases, some variation on the TAP provides a simple and relatively intuitive explanation of exactly how the
coarse-grid operator must approximate the fine-grid operator for convergence. Given a time-integration scheme
of interest, it can easily be plugged into results here for a more problem-specific description of necessary and sufficient
conditions for convergence. Theory is based on building error-propagation operators and appealing to block-Toeplitz
matrix theory; it is interesting to point out that a similar approach as used here likely provides a general convergence
framework for the recent idea of using circulant preconditioners in time to solve the space-time system in parallel
\cite{Goddard:2018vy,McDonald:2018cj}.

The rest of this paper proceeds as follows. Section \ref{sec:res} presents the main theoretical contributions
in a concise and accessible manner. Proofs of these results are then established in the following sections. Section
\ref{sec:conv} discusses convergence of reduction-based multigrid-type methods and derives analytic formulas
for two-level error- and residual-propagation operators of MGRiT. The most general theorems are then derived in Section 
\ref{sec:gen}, and further analysis based on additional assumptions is given in Section \ref{sec:diag}. Some
of the analysis can be extended to the time-dependent case, and a discussion on that topic is given in Section
\ref{sec:time}. There are many applications for the new theorems, such as finding the ``best'' coarse-grid time-stepping
scheme for a given problem, better understanding why hyperbolic problems tend to be difficult for parallel-in-time solvers
\cite{Ruprecht:2018hk},
and understanding the effect of spatial coarsening on the convergence of Parareal/MGRiT \cite{Ruprecht:2014gd,HowseDesterckMachlachlanFalgoutSchroder2017}, among others. A brief discussion on implications
of results derived here is given in Section \ref{sec:conc}, and a detailed study is the topic of a forthcoming paper. 

\section{Statement of results} \label{sec:res}

This section presents the theoretical contributions of this paper; proofs
are derived in the sections that follow. The underlying idea is that Parareal and MGRiT are iterative methods
to solve a discrete linear system \eqref{eq:system}, of which the exact solution is simply the space-time vector
achieved through sequential time-stepping. Here, convergence theory is derived to provide, under certain 
assumptions, necessary conditions and sufficient conditions for two-level MGRiT/Parareal to converge to the
sequential solution in norm.

\subsection{The framework}\label{sec:res:frame}

Let $\Phi$ be an $N_x\times N_x$ invertible fine-grid time-stepping operator and $\Psi$ an $N_x\times N_x$ invertible
coarse-grid time-stepping operator, and suppose we coarsen in time by a factor of $k$. The primary results rest on
three assumptions:
\begin{enumerate}
\item $\Phi$ and $\Psi$ are linear.
\item {$\|\Phi\|, \|\Psi\|$ are stable; that is, $\|\Phi^p\|,\|\Psi^p\| < 1$ for some $p$.}
\item $\Phi$ and $\Psi$ are independent of time; that is, the same operator propagates the solution from time
$t_i$ to $t_{i+1}$ and from $t_{j}$ to $t_{j+1}$, for all $i,j$. 
\end{enumerate}
Assumption 1 restricts our attention to the linear case, which, as in many problems, allows for a more detailed
analysis. {The second assumption is an algebraic requirement for a stable time-stepping scheme, which is a
also a natural and reasonable thing to assume. Note that a stronger assumption is $\|\Phi\|,\|\Psi\| < 1$. LeVeque
refers to this as ``strong stability'' \cite[Chapter 9.5]{leveque2007finite}, but results here hold for the more general case as well.}
The third assumption is the strongest one, yet it still
encompasses all problems for which variables do not have time-dependent coefficients, which consists of a large
class of space-time PDEs, among other problems. Some of the theory developed here applies to the time-dependent
case as well. In particular, sufficient conditions can be derived (see Theorem \ref{th:nontoeplitz}) for convergence if
$\Phi$ and $\Psi$ are time-dependent, but simultaneously diagonalizable for all time steps. This occurs, for example,
in the case of time-dependent reaction terms, or for adaptive time-stepping. Some other results hold in the time-dependent
case as well, which are discussed in Section \ref{sec:time}.

Further assumptions that come up (yielding stronger convergence results) are:
\begin{enumerate}\setcounter{enumi}{3}
\item Assume that $(\Psi-\Phi^k)$ is invertible.
\item Assume that $\Phi$ and $\Psi$ commute.
\begin{enumerate}
\item Assume that $\Phi$ and $\Psi$ are diagonalizable. 
\item Assume that $\Phi$ and $\Psi$ are normal (unitarily diagonalizable).
\end{enumerate}
\end{enumerate}
A discussion on these assumptions is provided in Section \ref{sec:res:ass}; however, it is worth noting that
all of them are plausible assumptions for many problems of interest. {Also, it is believed that Assumption 4 is
not actually necessary for any of the presented theoretical results. However, without Assumption 4, some of the
analysis becomes significantly more complicated, and it is not pursued in this work.}

 {\color{black}To consider convergence of Parareal and MGRiT, let $\hat{\mathbf{u}}$ be the exact solution
to \eqref{eq:system}. Then the error and residual for an approximate solution at the $i$th iteration, 
$\mathbf{u}_i$, are given, respectively, by
\begin{align*}
\mathbf{e}_i & =\hat{\mathbf{u}} - \mathbf{u}_i , \\
\mathbf{r}_i & = \mathbf{b} - A \mathbf{u}_i 
= A ( \hat{\mathbf{u}} - \mathbf{u}_i) 
= A\mathbf{e}_i.
\end{align*}
Here, we seek bounds on how Parareal and MGRiT propagate the error and residual corresponding to the linear
system in \eqref{eq:system}. To measure this propagation, we use the discrete $\ell^2$-norm ($\|\cdot\|$) and
$A^*A$-norm ($\|\cdot\|_{A^*A}$),\footnote{The $\ell^2$-norm and $A^*A$-norm are generally the most common
norms used for nonsymmetric problems, where the latter corresponds to a normal-equation formulation.} defined
by
\begin{align*}
\|\mathbf{e}_i\|^2 = \langle \mathbf{e}_i,\mathbf{e}\rangle, \hspace{3ex}
	\|\mathbf{e}_i\|_{A^*A}^2 = \langle A^*A\mathbf{e}_i,\mathbf{e}_i\rangle = \|\mathbf{r}_i\|^2.
\end{align*}
Note that assuming $A$ is nonsingular, $\mathbf{e}_i = \mathbf{0}$ if and only if $\mathbf{r}_i = \mathbf{0}$, and as
$\mathbf{e},\mathbf{r}\to \mathbf{0}$, we converge to the discrete solution obtained through sequential time
stepping. Although in practice users typically want the error to be small, the error cannot be easily 
measured in practice, while the residual can, making error and residual propagation both of interest.

Moving forward, let $\Phi$ denote the fine-grid time-stepping operator and $\Psi$ denote the coarse-grid time-stepping operator,
for all time points, with coarsening factor $k$, and $N_c$ time points on the coarse grid. Let $\mathbf{e}_i^{(F)}$ denote error
associated with Parareal or MGRiT with F-relaxation after $i$ iterations, $\mathbf{e}_i^{(FCF)}$ denote error associated with
MGRiT with FCF-relaxation after $i$ iterations, and similarly for residual vectors.}
Most results here are asymptotic in the sense that certain approximation properties and bounds are given up to order $O(1/N_c)$.
However, the leading constants in the $O(1/N_c)$ terms are also generally quite small, {and positive in all cases.} Furthermore,
because parallel-in-time is most often used when the number of time steps is relatively large, in practice these terms can
often be considered negligible. {Note that the operators $\Phi$ and $\Psi$ correspond to one time step, and are
independent of $N$ and $N_c$.}

{The most general results and remarks are presented in Section \ref{sec:res:th}. Several extensions for specific cases are
given in Sections \ref{sec:res:err} and \ref{sec:res:comm}, and proofs for all results are provided in the sections that follow.}

\subsection{Necessary and sufficient conditions}\label{sec:res:th}

{This section introduces necessary and sufficient conditions for convergence of residual in the $\ell^2$-norm
and, equivalently, error in the $A^*A$-norm, including tight bounds in norm. In the case
that $\Phi$ and $\Psi$ commute, these results hold for error in the $\ell^2$-norm as well.}
To start, a new temporal approximation property (TAP) is introduced, which is the fundamental assumption
leading to convergence. 
\begin{definition}[Temporal approximation property]
Let $\Phi$ denote a fine-grid time-stepping operator and $\Psi$ denote a coarse-grid time-stepping operator, for all time points,
with coarsening factor $k$. Then, $\Phi$ satisfies an F-relaxation temporal approximation property with power $p$
(F-TAP$_p$), with respect to $\Psi$, with constant $\varphi_{F,p}$, if, for all vectors $\mathbf{v}$,
\begin{align}\label{eq:tap_f}
\|(\Psi - \Phi^k)^p\mathbf{v}\| \leq \varphi_{F,p} \left[\min_{x\in[0,2\pi]} \left\| (I - e^{\mathrm{i}x}\Psi )^p\mathbf{v}\right\| \right].
\end{align}
Similarly, $\Phi$ satisfies an FCF-relaxation temporal approximation property with power $p$ (FCF-TAP$_p$), with
respect to $\Psi$, with constant $\varphi_{FCF,p}$, if, for all vectors $\mathbf{v}$,
\begin{align}\label{eq:tap_fcf}
\|(\Psi - \Phi^k)^p\mathbf{v}\| \leq \varphi_{FCF,p}\left[\min_{x\in[0,2\pi]} \left\| (\Phi^{-k}(I - e^{\mathrm{i}x}\Psi) )^p\mathbf{v}\right\| \right].
\end{align}
\end{definition}

Necessary and sufficient conditions for convergence of MGRiT and Parareal under various further assumptions are all
based on satisfying one of the above approximation properties with a nicely bounded constant, typically less than one.
The two variations on
a TAP are conceptually simple and can be presented in a more intuitive manner as follows. An F-TAP requires that
$\Phi^k$ approximates the action of $\Psi$ very accurately for vectors $\mathbf{v}\approx \Psi\mathbf{v}$, and less
accurately for $\mathbf{v}$ that differs significantly from $\Psi\mathbf{v}$. If $\Phi$ and $\Psi$ commute and have an orthogonal eigenvector
basis, then $\Phi^k$ must approximate $\Psi$ very accurately for eigenvectors of $\Psi$ with associated
eigenvalue close to one in magnitude, and less accurately for smaller eigenvalues. In the context of PDEs, order-one
eigenmodes of $\Psi$ typically correspond to the smallest eigenvalues of the spatial discretization. To that end, the
fine-grid and coarse-grid time-stepping operators must propagate ``smooth'' modes in the spatial domain
(corresponding to small eigenvalues) very similarly. In the case of an FCF-TAP, the additional term $\Phi^{-k}$ often makes
the TAP easier to satisfy. If $\|\Phi^k\| < 1$, that is, $\Phi$ is strongly stable \cite[Chapter 9.5]{leveque2007finite}, then
$\|\Phi^{-k}\mathbf{v}\| > \|\mathbf{v}\|$ for all $\mathbf{v}$. Think of this as an extra fudge
factor to help convergence (at the added expense of FCF-relaxation). When $\Psi$ is not diagonalizable, the eigenvectors do
not form a basis, and when $\|(I - e^{\mathrm{i}x}\Psi )\mathbf{v}\| \approx 0$ is a more complicated question. Further analysis
of that case, particularly for hyperbolic problems, is ongoing work. 

{\color{black}
Necessary and sufficient conditions for convergence of MGRiT and Parareal are now presented with respect to the TAP. 

\begin{theorem}[Necessary and sufficient conditions -- Error in the $A^*A$-norm]\label{th:genF}
Suppose that Assumptions 1, 2, and 3 hold, and that $\Phi$ satisfies an F-TAP$_1$ with respect to $\Psi$, with constant
$\varphi_F$, and $\Phi$ satisfies an FCF-TAP$_1$ with respect to $\Psi$, with constant $\varphi_{FCF}$. Then, 
\begin{align}
\frac{\|\mathbf{r}_{i+1}^{(F)}\|}{\|\mathbf{r}_i^{(F)}\|} = \frac{\|\mathbf{e}_{i+1}^{(F)}\|_{A^*A}}{\|\mathbf{e}_i^{(F)}\|_{A^*A}} & < 
	 \varphi_F\left(1+\|\Psi^{N_c}\|\right), \\
\frac{\|\mathbf{r}_{i+1}^{(FCF)}\|}{\|\mathbf{r}_i^{(FCF)}\|} = \frac{\|\mathbf{e}_{i+1}^{(FCF)}\|_{A^*A}}{\|\mathbf{e}_i^{(FCF)}\|_{A^*A}} & < 
	  \varphi_{FCF}\left(1+\|\Phi^{-k}\Psi^{N_c}\Phi^k\|\right),
\end{align}
for iterations $i > 1$. Thus, satisfying $\varphi_F\left(1+\|\Psi^{N_c}\|\right)< 1$ and
$ \varphi_{FCF}\left(1+\|\Phi^{-k}\Psi^{N_c}\Phi^k\|\right)<1$ are sufficient conditions for convergence of MGRiT with
F-relaxation and FCF-relaxation, respectively, on every iteration but one, with respect to error in the $A^*A$-norm. 

Additionally, assume that $(\Psi-\Phi^k)$ is invertible (Assumption 4). Then, 
\begin{align}
\frac{\varphi_F}{1+O(1/\sqrt{N_c})} & \leq \frac{\|\mathbf{r}_{i+1}^{(F)}\|}{\|\mathbf{r}_i^{(F)}\|} =
	\frac{\|\mathbf{e}_{i+1}^{(F)}\|_{A^*A}}{\|\mathbf{e}_i^{(F)}\|_{A^*A}} , \\
 \frac{\varphi_{FCF}}{1+O(1/\sqrt{N_c})} & \leq \frac{\|\mathbf{r}_{i+1}^{(FCF)}\|}{\|\mathbf{r}_i^{(FCF)}\|} =
 	\frac{\|\mathbf{e}_{i+1}^{(FCF)}\|_{A^*A}}{\|\mathbf{e}_i^{(FCF)}\|_{A^*A}},
\end{align}
for iterations $i > 1$. Thus, satisfying $\varphi_F < 1+O(1/\sqrt{N_c})$ and $\varphi_{FCF}< 1+O(1/\sqrt{N_c})$ are necessary
conditions for convergence of MGRiT with F-relaxation and FCF-relaxation, respectively, on every iteration but one,
with respect to error in the $A^*A$-norm. 

Finally, assume that $\Phi$ and $\Psi$ commute and either, (i) $(\Psi-\Phi^k)$ is invertible (Assumptions 4 and 5), or (ii)
$\Phi$ and $\Psi$ are diagonalizable (Assumptions 5 and 5a/5b). Then, $\Phi$ satisfying an F-TAP$_p$, for power $p\geq1$,
with respect to $\Psi$, with $\varphi_{F,p}<(1+O(1/\sqrt{N_c}))$, is a necessary condition to see convergent behavior
of Parareal and two-level MGRiT with F-relaxation, after $p$ iterations, with respect to error in the the $A^*A$-norm.
Similarly,  $\Phi$ satisfying an FCF-TAP$_p$, for power $p\geq1$, with respect to $\Psi$, with
$\varphi_{FCF,p}<(1+O(1/\sqrt{N_c}))$, is a necessary condition to see convergent behavior
of two-level MGRiT with FCF-relaxation, after $p$ iterations, with respect to error in the the $A^*A$-norm.
\end{theorem}

Theorem \ref{th:genF} presents necessary and sufficient conditions for convergence of Parareal and MGRiT
with minimal assumptions. The first statements provide necessary and sufficient conditions that every
iteration is convergent in the $\ell^2$-norm for residual and $A^*A$-norm for error. Note that the bounds are
tight. That is, as $N_c$ increases, $\|\Psi^{N_c}\| \approx 0$ and the worst-case ratio of successive error vectors
in the $A^*A$-norm converges exactly to $\varphi_F$ or $\varphi_{FCF}$. 

However, this only considers worst-case convergence for a single iteration. It is possible that convergence after
$p$ iterations is $\ll \varphi^p$. In theory, it is possible to see divergent behavior on initial iterations, but eventual
convergence (because for some non-Hermitian operators, $\|M^p\|\ll \|M\|^p$).  Under the additional assumptions that
$\Phi$ and $\Psi$ commute (which holds, for example, in using arbitrary single-step multi-stage integration schemes for
$\Phi$ and $\Psi$ (Section \ref{sec:res:ass})) and either Assumption 4 or 5a/5b, the final statement in each theorem
provides necessary conditions to see convergence of residual in the $\ell^2$-norm and error in the $A^*A$-norm
after an arbitrary number of iterations. 

\begin{corollary}[Extension to error in the $\ell^2$-norm]\label{cor:err_res}
If $\Phi$ and $\Psi$ commute, then identical conditions and bounds as in Theorem \ref{th:genF} hold for convergence
of error in the $\ell^2$-norm, on all iterations except the last (as opposed to the first).
\end{corollary}

If $\Phi$ and $\Psi$ do not commute, similar results as in Theorem \ref{th:genF} hold, but require a modified version
of the TAP, which is introduced in Section \ref{sec:res:err}. If they do commute \textit{and} are diagonalizable, Section
\ref{sec:res:comm} introduces additional results in a modified norm.

\begin{remark}[Real-valued operators]\label{rem:real}
Suppose $\Psi$ is real-valued and we only consider real-valued $\mathbf{v}$. Expanding
$\|(I - e^{\mathrm{i}x}\Psi)\mathbf{v}\|$ as an inner product yields
\begin{align*}
\|(I - e^{\mathrm{i}x}\Psi)\mathbf{v}\|^2 = \|\mathbf{v}\|^2 + \|\Psi\mathbf{v}\|^2 - 2\cos(\theta_x)\langle \Psi\mathbf{v},\mathbf{v}\rangle,
\end{align*}
where $e^{\mathrm{i}x} = \cos(\theta_x) + i \sin(\theta_x)$. Then,
\begin{align*}
\min_{x\in[0,2\pi]} \|(I - e^{\mathrm{i}x}\Psi)\mathbf{v}\| & = \begin{cases} \|(I + \Psi)\mathbf{v}\|  &
	\text{if } \langle \Psi\mathbf{v},\mathbf{v}\rangle \leq 0 \\
	\|(I - \Psi)\mathbf{v}\| & \text{if } \langle \Psi\mathbf{v},\mathbf{v}\rangle > 0 \end{cases}.
\end{align*}
Similarly, if $\Phi$ is also real-valued, then 
\begin{align*}
\min_{x\in[0,2\pi]} \|\Phi^{-k}(I - e^{\mathrm{i}x}\Psi)\mathbf{v}\| & = \begin{cases} \|\Phi^{-k}(I + \Psi)\mathbf{v}\| &
	\text{if } \langle \Phi^{-k}\Psi\mathbf{v},\Phi^{-k}\mathbf{v}\rangle \leq 0 \\
	\|\Phi^{-k}(I - \Psi)\mathbf{v}\| & \text{if } \langle \Phi^{-k}\Psi\mathbf{v},\Phi^{-k}\mathbf{v}\rangle > 0 \end{cases}.
\end{align*}
\end{remark}

\begin{remark}[F(CF)$^\rho$ relaxation]
Here, we only consider the cases of F- and FCF-relaxation. However, results generalize naturally to arbitrary F(CF)$^\rho$
relaxation (as considered in \cite{Gander:2018gj}), where the CF-steps are repeated $\rho$ times, by simply adding the
term $\Phi^{-k\rho}$ to the right-hand side of the F-TAP, analogous to the $\Phi^{-k}$ in the FCF-TAP. 
\end{remark}

\begin{remark}[Error tolerance, $\delta t$, and superlinear convergence]
It is important to see Parareal and MGRiT as iterative solvers to a discrete linear system \eqref{eq:system}, rather than an
integration scheme to solve the continuous problem. In general, the solution obtained through Parareal or MGRiT should be
no more accurate than that obtained through sequential time stepping, which is exactly defined by the choice of $\Phi$.
An important question in discrete linear systems is how accurately to solve them. Suppose $\Phi$ is a  time-integration
scheme with global accuracy $\mathcal{O}(\delta t^s)$. Then it is generally only necessary to solve the discrete linear
system \eqref{eq:system} (for example, using Parareal or MGRiT) to accuracy $\mathcal{O}(\delta t^s)$. 

When superlinear convergence of Parareal to the continuous solution is observed (for example, see \cite{Gander:2007jq}),
this corresponds to Parareal iterations converging faster (in the discrete sense) than the integration accuracy of $\Phi$. From
Theorem \ref{th:genF}, we see this is likely a result of satisfying the F-TAP with constant $O(\delta t^\ell)$, where $\ell$ is greater
than the integration accuracy of $\Phi$. For example, if the F-TAP is satisfied with constant $\varphi_F = \delta t^2$,
for given $\Phi$ and $\Psi$, independent of $\delta t$, then Parareal will converge like $\delta t^2$, even as $\delta t\to 0$. 
\end{remark}

\begin{remark}[Self-consistency of $\Psi$]
One of the most surprising results of this theory is how convergence of Parareal depends on the coarse-grid time stepper,
$\Psi$. It is natural to assume that $\Psi$ must approximate $k$ steps on the fine grid, $\Phi^k$, with accuracy that somehow
depends on $\Phi$. However, this is not the case. Indeed, the TAP illustrates that $\Psi$ must approximate $\Phi^k$ with
accuracy based on $I - \Psi$, indicating that there must be some self-consistency in terms of which vectors $\Psi$ approximates
the action of $\Phi^k$ on well. 
\end{remark}

\begin{remark}[Computing TAP constants]
The constants in the TAP are exactly defined as the maximum generalized singular value of the pair
$\{\Psi - \Phi^k, I - e^{\mathrm{i}\hat{x}}\Psi\}$, for some $\hat{x}\in[0,2\pi]$. If we consider real-valued operators, then
(from Remark \ref{rem:real}) we seek the maximum generalized singular value of $\{\Psi - \Phi^k, I - \Psi\}$ and
$\{\Psi - \Phi^k, I + \Psi\}$. For sparse matrices, such as those that arise with explicit time stepping of differential
discretizations, iterative methods have been developed to compute these values and vectors for relatively cheap
and without forming $(I - \Psi)^{-1}$ (for example, see \cite{zha1996computing}). Even if $\Phi$ and $\Psi$ are implicit
and thus contain inverses, iterative methods to compute the largest generalized singular value are typically applicable
if the action of $\Phi$ and $\Psi$ are available.
\end{remark}
}

\subsection{Tight convergence of $\ell^2$-error}\label{sec:res:err}
{\color{black}

Section \ref{sec:res:th} developed necessary and sufficient conditions for convergence of error in the $A^*A$-norm, and
Corollary \ref{cor:err_res} states that if $\Phi$ and $\Psi$ commute, equivalent results as Theorem \ref{th:genF} follow
immediately for error in the $\ell^2$-norm. If $\Phi$ and $\Psi$ do not commute, we need to introduce a modified inverse
temporal approximation property to study convergence of error in the $\ell^2$-norm.

\begin{definition}[Inverse temporal approximation property]
Let $\Phi$ denote a fine-grid time-stepping operator and $\Psi$ denote a coarse-grid time-stepping operator, for all time
points, with coarsening factor $k$, such that $(I - e^{\mathrm{i}x}\Psi)$ is invertible. Then, $\Phi$ satisfies an F-relaxation inverse
temporal approximation property (F-ITAP), with respect to $\Psi$, with constant $\tilde{\varphi}_{F}$, if, for all vectors
$\mathbf{v}$,
\begin{align}\label{eq:itap_f}
\max_{x\in[0,2\pi]} \left\| (I - e^{\mathrm{i}x}\Psi )^{-1}(\Psi - \Phi^k)\mathbf{v}\right\|  \leq \tilde{\varphi}_{F} \|\mathbf{v}\|.
\end{align}
Similarly, $\Phi$ satisfies an FCF-relaxation inverse temporal approximation property (FCF-ITAP), with
respect to $\Psi$, with constant $\tilde{\varphi}_{FCF}$, if, for all vectors $\mathbf{v}$,
\begin{align}\label{eq:itap_fcf}
\max_{x\in[0,2\pi]} \left\| (I - e^{\mathrm{i}x}\Psi )^{-1}(\Psi - \Phi^k)\mathbf{v}\right\| \leq \tilde{\varphi}_{FCF} \|\Phi^{-k}\mathbf{v}\|.
\end{align}
\end{definition}
In the case that $(\Psi - \Phi^k)$ is invertible, $(\Psi - \Phi^k)$ can be moved to the right-hand side. For example, the
F-ITAP can be expressed as
\begin{align*}
\max_{x\in[0,2\pi]} \left\| (I - e^{\mathrm{i}x}\Psi )^{-1}\mathbf{v}\right\| \leq \tilde{\varphi}_{F} \|(\Psi - \Phi^k)^{-1}\mathbf{v}\|,
\end{align*}
for all $\mathbf{v}$. Note the ITAP is not considered with respect to powers $p$. This is because derived results
based on powers also assume that $\Phi$ and $\Psi$ commute (see Theorem \ref{th:genF}),
in which case the results from Section \ref{sec:res:th} hold for error in the $\ell^2$-norm, and the ITAP is not necessary. 
Also, the assumption that $I - e^{\mathrm{i}x}\Psi$ is invertible is equivalent to assuming that $\Psi$ does not have an eigenvalue
of magnitude exactly one.\footnote{This assumption is likely a flaw in our line of proof and not actually necessary.}

\begin{theorem}[Necessary and sufficient conditions -- Error in the $\ell^2$-norm]\label{th:genF2}
Suppose that Assumptions 1, 2, and 3 hold, and that $\Phi$ satisfies an F-ITAP with respect to $\Psi$, with constant
$\tilde{\varphi}_F$, and $\Phi$ satisfies an FCF-ITAP with respect to $\Psi$, with constant $\tilde{\varphi}_{FCF}$. Then, 
with $n$ total iterations,
\begin{align}
\frac{\|\mathbf{e}_{i+1}^{(F)}\|}{\|\mathbf{e}_i^{(F)}\|} & < \tilde{\varphi}_F\left(1+\|\Psi^{N_c}\|\right), \\
\frac{\|\mathbf{e}_{i+1}^{(FCF)}\|}{\|\mathbf{e}_i^{(FCF)}\|} & < \tilde{\varphi}_{FCF}\left(1+\|\Psi^{N_c}\|\right),
\end{align}
for iterations $i=0,..,n-2$. Thus, satisfying $\tilde{\varphi}_F\left(1+\|\Psi^{N_c}\|\right)< 1$ and
$ \tilde{\varphi}_{FCF}\left(1+\|\Phi^{-k}\Psi^{N_c}\Phi^k\|\right)<1$ are sufficient conditions for convergence of MGRiT with
F-relaxation and FCF-relaxation, respectively, on every iteration but one, with respect to error in the $\ell^2$-norm. 

Additionally, assume that $(\Psi-\Phi^k)$ is invertible (Assumption 4). Then, with $n$ total iterations,
\begin{align}
\frac{\tilde{\varphi}_F}{1+O(1/\sqrt{N_c})} & \leq	\frac{\|\mathbf{e}_{i+1}^{(F)}\|}{\|\mathbf{e}_i^{(F)}\|} , \\
 \frac{\tilde{\varphi}_{FCF}}{1+O(1/\sqrt{N_c})} & \leq  \frac{\|\mathbf{e}_{i+1}^{(FCF)}\|}{\|\mathbf{e}_i^{(FCF)}\|},
\end{align}
for iterations $i=0,..,n-2$. Thus, satisfying $\tilde{\varphi}_F < 1+O(1/\sqrt{N_c})$ and $\tilde{\varphi}_{FCF}< 1+O(1/\sqrt{N_c})$
are necessary conditions for convergence of MGRiT with F-relaxation and FCF-relaxation, respectively, on every iteration but one,
with respect to error in the $\ell^2$-norm. 
\end{theorem}

As in Section \ref{sec:res:th}, worst-case convergence in the $\ell^2$-norm is given by constants $\tilde{\varphi}_F$ 
and $\tilde{\varphi}_{FCF}$ to $O(1/N_c)$. 
}

\subsection{Additional results for commuting diagonalizable operators}\label{sec:res:comm}

As it turns out, results from above can be strengthened in some sense under the additional assumption that $\Phi$ and
$\Psi$ commute and are diagonalizable. This leads to exact bounds on convergence in a modified norm, that are fairly tight
for a large number of iterations, $p$, as well. By norm equivalence in finite-dimensional spaces, convergence in the modified norm is also
necessary and sufficient for (asymptotic) convergence in the $\ell^2$- and $A^*A$-norms. The constants in norm
equivalence depend on the conditioning of the eigenvectors. First, we introduce a less general approximation
property based on the assumption that $\Phi$ and $\Psi$  commute and are diagonalizable.

\begin{definition}[Temporal eigenvalue approximation property]
Let $\Phi$ denote a fine-grid time-stepping operator and $\Psi$ denote a coarse-grid time-stepping operator,
of size $N_x\times N_x$, for all time points, with coarsening factor $k$. Suppose that $\Phi$ and $\Psi$ commute and are
diagonalizable, with eigenvalues given by $\{\lambda_\ell\}$ and $\{\mu_\ell\}$, respectively. Then, $\Phi$ satisfies an F-relaxation
temporal eigenvalue approximation property (F-TEAP), with respect to $\Psi$, with constant
$\varphi_{F}$, if, for $\ell=0,...,N_x-1$,
\begin{align}\label{eq:teap_f}
|\mu_\ell - \lambda_\ell^k| \leq \varphi_{F}(1 - |\mu_\ell|).
\end{align}
Similarly, $\Phi$ satisfies an FCF-relaxation temporal eigenvalue approximation property (FCF-TEAP), with
respect to $\Psi$, with constant $\varphi_{FCF}$, if, for $\ell=0,...,N_x-1$,
\begin{align}\label{eq:teap_fcf}
|\mu_\ell - \lambda_\ell^k| \leq \varphi_{FCF}\frac{1 - |\mu_\ell|}{|\lambda_\ell^{k}|}.
\end{align}
\end{definition}
Note that for the TEAP, there is no distinction between powers, because scalars commute. {Furthermore, Assumption
2 implies that all eigenvalues $|\mu_i|, |\lambda_i| < 1$. }

\begin{theorem}[The diagonalizable case -- F-relaxation]\label{th:necF}
Let $\Phi$ denote the fine-grid time-stepping operator and $\Psi$ denote the coarse-grid time-stepping operator, for all time points,
with coarsening factor $k$, and $N_c$ time points on the coarse grid. Assume that $\Phi$ and $\Psi$ commute and
are diagonalizable, with eigenvectors given as columns of $U$, and that $\Phi$ satisfies an F-TEAP with respect
to $\Psi$, with constant $\varphi_F<1$. Let $\mathbf{e}_{p+1}$ denote the error vector of Parareal/MGRiT with F-relaxation
after $p+1$ iterations. Then, $\|\mathbf{e}_1\|_{(UU^*)^{-1}}^2 \leq k\|\mathbf{e}_0\|_{(UU^*)^{-1}}^2$, and
\begin{align}\label{eq:errF}
\|\mathbf{e}_{p+1}\|_{(UU^*)^{-1}}^2 \leq \left(\varphi_F^{2p} - O(1/N_c^2)\right)\|\mathbf{e}_1\|_{(UU^*)^{-1}}^2.
\end{align}
Furthermore, this bound is tight; that is, there exists an initial error $\mathbf{e}_0$ such that \eqref{eq:errF} holds
with equality, to $O(1/N_c^2)$. 

This also provides necessary and sufficient (asymptotic) conditions for convergence in the $\ell^2$- and $A^*A$-norms.
That is, iterations may diverge at first, but will eventually converge in the $\ell^2$- and $A^*A$-norms.
\end{theorem}

\begin{theorem}[The diagonalizable case -- FCF-relaxation]\label{th:necFCF}
Let $\Phi$ denote the fine-grid time-stepping operator and $\Psi$ denote the coarse-grid time-stepping operator, for all time points,
with coarsening factor $k$, and $N_c$ time points on the coarse grid. Assume that $\Phi$ and $\Psi$ commute and
are diagonalizable, with eigenvectors given as columns of $U$, and that $\Phi$ satisfies an FCF-TEAP with respect
to $\Psi$, with constant $\varphi_{FCF}<1$. Let $\mathbf{e}_{p+1}$ denote the error vector of MGRiT with FCF relaxation
after $p+1$ iterations. Then, $\|\mathbf{e}_1\|_{(UU^*)^{-1}}^2 \leq k\|\mathbf{e}_0\|_{(UU^*)^{-1}}^2$, and
\begin{align}\label{eq:errFCF}
\|\mathbf{e}_{p+1}\|_{(UU^*)^{-1}}^2 \leq \left(\varphi_{FCF}^{2p} - O(1/N_c^2)\right)\|\mathbf{e}_1\|_{(UU^*)^{-1}}^2.
\end{align}
Furthermore, this bound is tight; that is, there exists an initial error $\mathbf{e}_0$ such that \eqref{eq:errFCF} holds
with equality, to $O(1/N_c^2)$. 

This also provides necessary and sufficient (asymptotic) conditions for convergence in the $\ell^2$- and $A^*A$-norms.
That is, iterations may diverge at first, but will eventually converge in the $\ell^2$- and $A^*A$-norms.
\end{theorem}

Note that in the case of normal matrices, $U^{-1} = U^*$, and we have that the $(UU^*)^{-1}$-norm is exactly equal to
the $\ell^2$-norm. In that case, the TEAP and TAP are equivalent, and we have an exact bound on Parareal and
two-level MGRiT convergence of residual in the $\ell^2$-norm and error in the $\ell^2$- and $A^*A$-norms. In the
commuting and diagonalizable case, these results are an extension of the upper bounds developed in
\cite{DobrevKolevPeterssonSchroder2017}.

Some of these results can be extended to the time-dependent case as well, such as when there are time-dependent reaction
terms in a PDE or variable time-step size. Such scenarios are discussion in Section \ref{sec:time}.

\begin{remark}[Convergence bounds observed in practice]
It is worth pointing out that the TAP and Theorems \ref{th:genF}, \ref{th:genF2}, \ref{th:necF}, and \ref{th:necFCF} not only
define worst-observable convergence factors of Parareal and two-level MGRiT, but such convergence factors are likely to be
observed in practice. In the theoretical derivations that follow, convergence bounds are derived based on minimum and maximum
eigenvalues or singular values of block-Toeplitz matrices. In many cases, it can be shown that there are clusters of singular modes
or eigenmodes near these upper or lower bounds \cite{Miranda:2000is,Serra:1998cl,Serra:1998fm,Serra:1999bj},
making them likely to be observed in practice. Numerical results confirming this for diagonalizable model problems can be
found in \cite{DobrevKolevPeterssonSchroder2017}, where the proposed upper bounds match the exact bounds
of Theorems \ref{th:necF} and \ref{th:necFCF}.
\end{remark}

\subsection{Discussion on assumptions}\label{sec:res:ass}

To remark on Assumptions 4 and 5 from Section \ref{sec:res:th}, note that almost all time-integration routines (including all
single-step Runge-Kutta-type methods) are rational functions of
some invertible operator $\mathcal{L}$, where, for example, $\mathcal{L}$ is a scalar in the case of a standard ODE,
or a spatial discretization operator in the case of a space-time PDE. Starting with Assumption 5, if $\Phi$ and
$\Psi$ are both functions of $\mathcal{L}$, assuming that they commute is a mild assumption, because any
rational function of $\mathcal{L}$ commutes. This includes most time-integration schemes, including all single-step
multi-stage Runge-Kutta type schemes. Assumptions (5a) and (5b) then follow if, in addition, $\mathcal{L}$
is diagonalizable and normal, respectively. These are stronger assumptions that are satisfied, for example, in
the case of many parabolic PDEs.

{\color{black}

Returning to Assumption 4, note again that it is believed Assumption 4 is not necessary, and is rather a flaw
in our line of proof. Nevertheless, here we use an example to show that assuming $\Psi - \Phi^k$ is invertible is 
reasonable anyways. Of course, because we want $\Psi \approx\Phi^k$, we don't \textit{want} $\Psi - \Phi^k$ to be
invertible. If $\Psi \mathbf{v} = \Phi^k\mathbf{v}$ for any vector $\mathbf{v}$, then indeed it is not invertible.
However, in practice it is unlikely for $\Psi$ to \textit{exactly} preserve a mode of $\Phi^k$ in this manner.

\begin{example}[RK4 and $\Psi - \Phi^k$]\label{ex:rk}
Consider RK4 time integration with coarsening by a factor of two, where $\Phi$ corresponds to RK4 with time step $\delta t$,
and $\Psi$ corresponds to RK4 with time step $2\delta t$ (the standard approach used in Parareal and MGRiT to approximate
$\Phi^2$). Assume that $\Phi$ and $\Psi$ are stable. Then,
\begin{align*}
\Phi & = I + \delta t\mathcal{L} + \tfrac{\delta t^2}{2}\mathcal{L}^2 + \tfrac{\delta t^3}{6}\mathcal{L}^3 +
	\tfrac{\delta t^4}{24}\mathcal{L}^4 , \\
\Psi & =  I + 2 \delta t \mathcal{L} + 2 \delta t^2\mathcal{L}^2 +  \tfrac{4\delta t^3}{3}\mathcal{L}^3 + \tfrac{2\delta t^4}{3}\mathcal{L}^4 \\
& = \Phi^2 - \left[\tfrac{\delta t^5}{4}\mathcal{L}^5 + \tfrac{5\delta t^6}{72}\mathcal{L}^6 +
	\tfrac{\delta t^7}{72}\mathcal{L}^7 + \tfrac{\delta t^8}{576} \mathcal{L}^8\right] \\
& = \Phi^2 - \tfrac{\delta t^5}{4}\mathcal{L}^5\left[I + \tfrac{5\delta t}{18}\mathcal{L} +
	\tfrac{\delta t^2}{18}\mathcal{L}^2 + \tfrac{\delta t^3}{144} \mathcal{L}^3\right], \\
{\Psi - \Phi^2} & {= - \tfrac{\delta t^5}{4}\mathcal{L}^5\left[I + \tfrac{5\delta t}{18}\mathcal{L} +
	\tfrac{\delta t^2}{18}\mathcal{L}^2 + \tfrac{\delta t^3}{144} \mathcal{L}^3\right]}.
\end{align*}
Assuming $\mathcal{L}$ is nonsingular (which it should be for a well-posed problem), $\Psi - \Phi^2$ is only
singular (non-invertible) if an eigenvalue $\lambda_i$ of $\mathcal{L}$ is exactly one of the three roots of
\begin{align*}
p(\lambda) = 1 + \tfrac{5\delta t}{18}\lambda + \tfrac{\delta t^2}{18}\lambda^2 + \tfrac{\delta t^3}{144}\lambda^3.
\end{align*}
If such an eigenvalue does not exist, then $\Psi - \Phi^k$ is invertible. Working through the closed form for cubic
roots, one can show that the roots of $p(\lambda)$ are approximately given by 
\begin{align*}
\lambda_0 \approx -\frac{5.5}{\delta t}, \hspace{2ex}\lambda_1 \approx \frac{4.96\mathrm{i} - 0.2}{\delta t},
	\hspace{2ex}\lambda_2 \approx \frac{-4.96\mathrm{i} - 0.2}{\delta t}.
\end{align*}

Returning to the assumption that $\Phi$ is stable, a necessary condition for this is that all eigenvalues of $\Phi$
are less than one in magnitude. Thus suppose $\hat{\lambda}\in\{\lambda_0,\lambda_1,\lambda_2\}$ is an
eigenvalue of $\mathcal{L}$ that is also a root of $p(\lambda)$. Applying $\Phi$ to the corresponding eigenvector
$\hat{\mathbf{v}}$ shows that $\Phi$ must have an eigenvalue $\gg 1$. By contradiction, $\Psi - \Phi^k$ must
be invertible if $\Phi$ is stable.
\end{example}

Example \ref{ex:rk} proves that for RK4 with coarsening by a factor of two, $\Psi - \Phi^k$ must be invertible
if $\Phi$ is stable. In general, if $\Phi$ and $\Psi$ are rational functions of some operator $\mathcal{L}$
(same operator on each level), then $\Psi-\Phi^k$ is invertible as long as one of the eigenvalues of
$\mathcal{L}$ is not a root of the difference of the two characteristic polynomials. If $\mathcal{L}$ has
non-negative eigenvalues, then this holds for \textit{all} explicit Runge-Kutta schemes. 

The analysis is more complicated for larger $k$ or if $\Phi$ and $\Psi$ are based on different
operators (for example, if spatial coarsening is used \cite{HowseDesterckMachlachlanFalgoutSchroder2017}).
However, the general heuristic stands that it is unlikely for $\Phi^k$ to exactly preserve a mode
of $\Psi$. If, in fact, it does for a specific problem, a small perturbation to $\delta t$ would likely nullify that
property, and the assumption that $\Psi - \Phi^k$ is invertible will stand.
}

\section{Convergence theory framework} \label{sec:conv}

\subsection{Error and residual propagation} \label{sec:conv:conv}

Let $\mathcal{E}$ denote the error-propagation operator and $\mathcal{R}$ the residual-propagation operator of 
Parareal or two-level MGRiT. These operators are derived analtically in this section. Note that for fixed-point 
iterative methods, error propagation takes the form $\mathcal{E} = I - M^{-1}A$, where $M$ is some approximation of
$A$. From Section \ref{sec:res:frame}, observe that
\begin{align*}
\mathbf{e}_i & = \mathcal{E}^i\mathbf{e}_0\hspace{3ex}  
\Longleftrightarrow \hspace{3ex} A^{-1}\mathbf{r}_i = \mathcal{E}^iA^{-1}\mathbf{r}_0\hspace{3ex} 
\Longleftrightarrow \hspace{3ex} \mathbf{r}_i = (A\mathcal{E}A^{-1})^i\mathbf{r}_0.
\end{align*}
Thus, residual propagation is formally similar to error propagation, where $\mathcal{R} = A\mathcal{E}A^{-1} = I - AM^{-1}$.
In this form, error propagation is a measure of $M$ as a left approximate inverse of $A$ and residual propagation a measure
of $M$ as a right approximate inverse of $A$. Then, observe that
\begin{align}
\|\mathcal{R}\|^2 & = \sup_{\mathbf{x}\neq\mathbf{0}} \frac{\langle \mathcal{R}\mathbf{x},\mathcal{R}\mathbf{x}\rangle}{\langle\mathbf{x},\mathbf{x}\rangle}
	= \sup_{\mathbf{x}\neq\mathbf{0}} \frac{\langle A\mathcal{E}A^{-1}\mathbf{x},A\mathcal{E}A^{-1}\mathbf{x}\rangle}{\langle\mathbf{x},\mathbf{x}\rangle}
	= \sup_{\mathbf{y}\neq\mathbf{0}} \frac{\langle A^*A\mathcal{E}\mathbf{y},\mathcal{E}\mathbf{y}\rangle}{\langle A^*A\mathbf{y},\mathbf{y}\rangle} 
	= \|\mathcal{E}\|^2_{A^*A}; \label{eq:AsA}
\end{align}
that is, the norm of the residual-propagation operator in the $\ell^2$-norm is equivalent to that of the error-propagation 
operator in the $A^*A$-norm.\footnote{In general, the $A^*A$-norm is considered a stronger norm, also consistent with
the result that one can have an arbitrarily accurate left approximate inverse that makes for a poor right approximate
inverse \cite{Mendelsohn:1956wk}.} Note, this is consistent with the previous noted relation,
$\|\mathbf{e}\|_{A^*A} = \|\mathbf{r}\|$.

\subsection{Reduction-based multigrid}\label{sec:conv:mg}

MGRiT and Parareal are both reduction-based multigrid algorithms. Multigrid methods are
a class of iterative methods based on two parts: relaxation and coarse-grid correction, which are designed to be complementary
in the sense that they each reduce different, complementary, error modes \cite{BrHeMc2000}. Error
propagation of relaxation typically takes the form $I - M^{-1}A$, where $M$ is some easy to compute approximation of $A$, such as the
diagonal or lower-triangular block. Coarse-grid correction is a subspace projection, for which error propagation takes the form
$I - P(RAP)^{-1}RA$. Here, $R$ is a restriction operator, which restricts residuals on the fine grid to a coarse-grid problem; $RAP$
defines the coarse-grid problem to be solved; and $P$ is an interpolation operator to interpolate a correction back to the fine grid. 
Moving forward, at times we will just refer to MGRiT, but imply Parareal as well.

For most multigrid methods, the purpose of relaxation is to reduce error associated with highly oscillatory modes (large eigenvalues/singular
values in the algebraic case). Coarse-grid correction is then complementary by reducing error associated with geometrically smooth modes, or
small eigenvalues/singular values. In a reduction-based multigrid method, relaxation and coarse-grid correction instead reduce error associated
with different degrees of freedom (DOFs), or, equivalently, blocks in the matrix. To this end, suppose all DOFs are partitioned into a disjoint covering
of C-points and F-points, and matrices $A,P$, and $R$ take the following block forms:
\begin{align} \label{eq:blocks}
A = \begin{bmatrix} A_{ff} & A_{fc} \\ A_{cf} & A_{cc} \end{bmatrix}, \hspace{3ex}
P = \begin{bmatrix} W \\ I \end{bmatrix}, \hspace{3ex}
R = \begin{bmatrix} Z & I \end{bmatrix},
\end{align}
where the identity blocks in $P$ and $R$ are on the $n_c\times n_c$ C-point block.
A simple two-level reduction-based multigrid method is given by letting $Z = -A_{cf}A_{ff}^{-1}$ and $W = \mathbf{0}$. In this case,
coarse-grid correction yields zero error at C-points \cite{air1}. The restriction operator defined through $Z=-A_{cf}A_{ff}^{-1}$
is referred to as ``ideal restriction,'' denoted $R_{\textnormal{ideal}}$,
where it is ideal in being the unique restriction operator that yields an exact coarse-grid correction at C-points. Following this with an
exact solve on F-points as a relaxation scheme then yields an exact solution at F-points, without modifying the solution at C-points
\cite{air2,air1}. Thus, the solution is exact and we have a two-grid reduction, where solving $A\mathbf{x} = \mathbf{b}$ is reduced to
 solving one system based on $A_{ff}$ and one system based on $RAP$.

MGRiT is also a reduction-based multigrid method, instead using the so-called ``ideal interpolation'' operator. Ideal interpolation,
denoted $P_{\textnormal{ideal}}$,
is defined through $W = -A_{ff}^{-1}A_{fc}$. For symmetric positive definite matrices, ideal interpolation is ideal in a certain
theoretical sense \cite{Falgout:2004cs,ideal18}. In the nonsymmetric setting, ideal interpolation is ideal as the unique interpolation operator that
eliminates the contribution of F-point residuals to the coarse-grid right-hand side \cite{air2,FalgoutFriedhoffKolevMaclachlanSchroder2014}.
When coupled with $R = \begin{bmatrix}\mathbf{0} & I\end{bmatrix}$, referred to as restriction by injection, coarse-grid correction then yields zero
residual at C-points. Note that an exact solve on F-points yields zero residual at F-points. Thus, coarse-grid correction with
$P_{\textnormal{ideal}}$ and restriction by injection, preceded by an exact solve on F-points, also yields an exact two-level
reduction \cite{air2}.\footnote{Note that the ordering is important: coarse-grid correction with $P_{\textnormal{ideal}}$ and restriction by
injection, followed by an exact solve on F-points, does not yield a two-grid reduction \cite{air2}.}

In the algebraic setting, $A_{ff}^{-1}$ is often not easily computed, so approximations are made, such as in AMG methods
based on an approximate ideal restriction (AIR) \cite{air1,air2}. MGRiT and the system in \eqref{eq:system} are unique in that the action
of $A_{ff}^{-1}$ can be computed, so ideal interpolation and exact F-relaxation are feasible choices. In this case, assuming a
block form as in \eqref{eq:blocks}, $RAP_{\textnormal{ideal}} = S_A$, where $S_A$ is the Schur complement, independent of $R$ \cite{air2}. 
Although we can express a closed form for $S_A$ (see \eqref{eq:rap}), $S_A$ is not amenable to a recursive multilevel algorithm.
In particular, one time step on the Schur-complement coarse grid simply consists of taking $k$ steps on the fine grid, which is no
cheaper than solving the fine grid problem directly. Because of this, MGRiT is based on a non-Galerkin coarse grid, where we
approximate $\Phi^k$ with some other operator $\Psi$. Usually, this is accomplished by approximating $k$ steps of size
$\delta t$ with one step of size $k\delta t$. 

The following section derives error and residual-propagation operators for MGRiT. Further details on reduction-based
multigrid methods can be found in \cite{air1,air2,MacLachlan:2006gt,Ries:1983}, and further details on the MGRiT algorithm can
be found in, for example, \cite{DobrevKolevPeterssonSchroder2017,FalgoutFriedhoffKolevMaclachlanSchroder2014,
FalgoutLecouvezWoodward2017,GahvariDobrevFalgoutKolevSchroderSchulzYang2016}. The reduction properties
rely on the ideal interpolation and restriction operators,
\begin{align*}
R_{\textnormal{ideal}} & = \begin{bmatrix}-A_{cf}A_{ff}^{-1}& I \end{bmatrix}, \hspace{5ex}
	P_{\textnormal{ideal}} =\begin{bmatrix} -A_{ff}^{-1}A_{fc} \\ I \end{bmatrix}.
\end{align*}

\subsection{Error and residual-propagation operators}\label{sec:conv:prop}

Consider residual and error propagation of two-level MGRiT, with a non-Galerkin coarse-grid operator, $B_\Delta^{-1}$,
to approximate the Schur complement, $A_\Delta:=S_A^{-1}$ ($A_\Delta$ is used to be consistent with previous works
\cite{DobrevKolevPeterssonSchroder2017,FalgoutFriedhoffKolevMaclachlanSchroder2014}). Because MGRiT is
based on ideal interpolation, here we use a pre-relaxation scheme of
F-relaxation or FCF-relaxation \cite{air2}. It is important to note that, in the case of the MGRiT algorithm, F-relaxation and
C-relaxation refer to an exact solve on F- and C-points, respectively. {This is how parallelization in time is achieved --
an exact solve on F-points corresponds to using the current (spatial) solution at each C-point (in time) and integrating
the spatial solution forward in time over the proceeding $k-1$ F-points. This local time integration is coupled with a global
time integration on a coarse grid that is cheaper to evaluate ($B_\Delta^{-1}$), for a two-level parallel-in-time iterative method.}

Recall that error propagation of relaxation and coarse-grid correction each take the form of a classic fixed-point method,
$I - M^{-1}A$. The approximate inverses for an exact solve on F-points, an exact solve on C-points, and coarse-grid correction,
are given, respectively, by
\begin{align*}
M_F^{-1}& = \begin{bmatrix} A_{ff}^{-1} & \mathbf{0} \\ \mathbf{0} & \mathbf{0} \end{bmatrix} , \hspace{3ex}
M_C^{-1} = \begin{bmatrix} \mathbf{0} & \mathbf{0} \\ \mathbf{0} & A_{cc}^{-1}\end{bmatrix} ,\\
M_{cgc}^{-1} & = \begin{bmatrix}-A_{ff}^{-1}A_{fc} \\ I \end{bmatrix} B_\Delta^{-1} \begin{bmatrix}\mathbf{0} & I \end{bmatrix} 
	 = \begin{bmatrix} \mathbf{0} & -A_{ff}^{-1}A_{fc}B_\Delta^{-1} \\ \mathbf{0} & B_\Delta^{-1} \end{bmatrix}.
\end{align*}
Then, error and residual propagation of two-level MGRiT with pre F-relaxation are given by
$\mathcal{E}_F = I - (M_F^{-1} + M_{cgc}^{-1} -M^{-1}_{cgc}AM^{-1}_F)A$ and
$\mathcal{R}_F = I - A(M_F^{-1} + M_{cgc}^{-1} -M^{-1}_{cgc}AM^{-1}_F)$, respectively. Note that 
\begin{align*}
M^{-1}_{cgc}AM^{-1}_F & =  \begin{bmatrix} \mathbf{0} & -A_{ff}^{-1}A_{fc}B_\Delta^{-1} \\ \mathbf{0} & B_\Delta^{-1} \end{bmatrix}
	\begin{bmatrix} A_{ff} & A_{fc} \\ A_{cf} & A_{cc} \end{bmatrix} \begin{bmatrix} A_{ff}^{-1} & \mathbf{0} \\ \mathbf{0} & \mathbf{0} \end{bmatrix} \\
& = \begin{bmatrix} -A_{ff}^{-1}A_{fc}B_\Delta^{-1}A_{cf}A_{ff}^{-1} & \mathbf{0} \\B_\Delta^{-1}A_{cf}A_{ff}^{-1} & \mathbf{0} \end{bmatrix}.
\end{align*}
Combining,
\begin{align} 
\mathcal{R}_F & = \begin{bmatrix} I & \mathbf{0} \\ \mathbf{0} & I \end{bmatrix} - \begin{bmatrix} A_{ff} & A_{fc} \\ A_{cf} & A_{cc} \end{bmatrix} 
	\begin{bmatrix} A_{ff}^{-1} + A_{ff}^{-1}A_{fc}B_\Delta^{-1}A_{cf}A_{ff}^{-1} & -A_{ff}^{-1}A_{fc} B_\Delta^{-1} \nonumber\\ 
	-B_\Delta^{-1}A_{cf}A_{ff}^{-1} & B_\Delta^{-1}\end{bmatrix} \\
& = \begin{bmatrix} \mathbf{0} & \mathbf{0} \\ -(I - A_\Delta B_\Delta^{-1})A_{cf}A_{ff}^{-1} & I - A_\Delta B_{\Delta}^{-1}\end{bmatrix}  \label{eq:r_pref} \\
& = \begin{bmatrix} \mathbf{0} \\ I - A_\Delta B_{\Delta}^{-1}\end{bmatrix} R_{\textnormal{ideal}}, \label{eq:rf_pow}\\
\mathcal{E}_F & = \begin{bmatrix} I & \mathbf{0} \\ \mathbf{0} & I \end{bmatrix} - 
	\begin{bmatrix} A_{ff}^{-1} + A_{ff}^{-1}A_{fc}B_\Delta^{-1}A_{cf}A_{ff}^{-1} & -A_{ff}^{-1}A_{fc} B_\Delta^{-1} \nonumber\\ 
	-B_\Delta^{-1}A_{cf}A_{ff}^{-1} & B_\Delta^{-1}\end{bmatrix} \begin{bmatrix} A_{ff} & A_{fc} \\ A_{cf} & A_{cc} \end{bmatrix} \\
& = \begin{bmatrix} \mathbf{0} & -A_{ff}^{-1}A_{fc}({I - B_\Delta^{-1} A_\Delta}) \\ \mathbf{0} & {I - B_\Delta^{-1} A_\Delta}\end{bmatrix}  \label{eq:e_pref} \\
& = P_{\textnormal{ideal}}\begin{bmatrix}\mathbf{0} & {I - B_\Delta^{-1} A_\Delta}\end{bmatrix}\label{eq:ef_pow}.
\end{align}

To consider FCF-relaxation, note that MGRiT residual and error propagation for FCF-relaxation is equivalent to multiplying
$\mathcal{R}_F$ and $\mathcal{E}_F$ by residual and error propagation for FC-relaxation, which are respectively given by
\begin{align*}
I - A(M_F^{-1} + M_C^{-1} - M_C^{-1}AM_F^{-1}) & = I - \begin{bmatrix} A_{ff} & A_{fc} \\ A_{cf} & A_{cc} \end{bmatrix}
	\begin{bmatrix} A_{ff}^{-1} & \mathbf{0} \\ -A_{cc}^{-1}A_{cf}A_{ff}^{-1} & A_{cc}^{-1} \end{bmatrix} \\
& = \begin{bmatrix} A_{fc}A_{cc}^{-1}A_{cf}A_{ff}^{-1} & -A_{fc}A_{cc}^{-1} \\ \mathbf{0} & \mathbf{0} \end{bmatrix}, \\
I - (M_F^{-1} + M_C^{-1} - M_C^{-1}AM_F^{-1})A & = I - \begin{bmatrix} A_{ff}^{-1} & \mathbf{0} \\ 
	-A_{cc}^{-1}A_{cf}A_{ff}^{-1} & A_{cc}^{-1} \end{bmatrix} \begin{bmatrix} A_{ff} & A_{fc} \\ A_{cf} & A_{cc} \end{bmatrix} \\
& = \begin{bmatrix} \mathbf{0} & -A_{ff}^{-1}A_{fc} \\ \mathbf{0} & A_{cc}^{-1}A_{cf}A_{ff}^{-1}A_{fc} \end{bmatrix}.
\end{align*}
It follows that residual and error propagation of two-level MGRiT with pre FCF-relaxation are given by
\begin{align}
\mathcal{R}_{FCF} & = \begin{bmatrix} \mathbf{0} & \mathbf{0} \\ -(I - A_\Delta B_\Delta^{-1})A_{cf}A_{ff}^{-1} & I - A_\Delta B_{\Delta}^{-1}\end{bmatrix}
	 \begin{bmatrix} A_{fc}A_{cc}^{-1}A_{cf}A_{ff}^{-1} & -A_{fc}A_{cc}^{-1} \\ \mathbf{0} & \mathbf{0} \end{bmatrix} \nonumber\\
& = \begin{bmatrix} \mathbf{0} & \mathbf{0} \\ -(I - A_\Delta B_\Delta^{-1})A_{cf}A_{ff}^{-1}A_{fc}A_{cc}^{-1}A_{cf}A_{ff}^{-1} & 
	(I - A_\Delta B_\Delta^{-1})A_{cf}A_{ff}^{-1}A_{fc}A_{cc}^{-1}\end{bmatrix} \label{eq:r_fcf} \\
& = \begin{bmatrix} \mathbf{0} \\ (I - A_\Delta B_\Delta^{-1})A_{cf}A_{ff}^{-1}A_{fc}A_{cc}^{-1}\end{bmatrix} R_{\textnormal{ideal}}, \label{eq:rfcf_pow}\\
\mathcal{E}_{FCF} & =\begin{bmatrix} \mathbf{0} & -A_{ff}^{-1}A_{fc}({I - B_\Delta^{-1} A_\Delta}) \\ \mathbf{0} & {I - B_\Delta^{-1} A_\Delta}\end{bmatrix}
	\begin{bmatrix} \mathbf{0} & -A_{ff}^{-1}A_{fc} \\ \mathbf{0} & A_{cc}^{-1}A_{cf}A_{ff}^{-1}A_{fc} \end{bmatrix} \nonumber \\
& = \begin{bmatrix} \mathbf{0} & -A_{ff}^{-1}A_{fc}({I - B_\Delta^{-1} A_\Delta})A_{cc}^{-1}A_{cf}A_{ff}^{-1}A_{fc} \\
	\mathbf{0} & ({I - B_\Delta^{-1} A_\Delta})A_{cc}^{-1}A_{cf}A_{ff}^{-1}A_{fc}\end{bmatrix} \label{eq:e_fcf} \\
& = P_{\textnormal{ideal}}\begin{bmatrix}\mathbf{0} & ({I - B_\Delta^{-1} A_\Delta})A_{cc}^{-1}A_{cf}A_{ff}^{-1}A_{fc}\end{bmatrix}\label{eq:efcf_pow}.
\end{align}

Note from \eqref{eq:AsA} that $\|\mathcal{E}_F^p\|_{A^*A} = \|\mathcal{R}_F^p\|$ and
$\|\mathcal{E}_{FCF}^p\|_{A^*A} = \|\mathcal{R}_{FCF}^p\|$, for $p\geq 1$. 

\subsection{MGRiT matrices}\label{sec:conv:mat}

So far, derivations have been purely algebraic and assumed no structure to the linear system. Focusing on the MGRiT framework,
consider the MGRiT system \eqref{eq:system} and suppose we coarsen in time by a factor of $k$. This corresponds to partitioning
time points into C-points and F-points, such that for every $k$ points, $k-1$ are F-points. For convenience, assume that the first
and last time points are C-points, in which case the total number of C-points is given by $N_c = 1+\frac{N-1}{k}$, where $N$ is the
total number of time points.\footnote{This is slightly different notation than used in \cite{DobrevKolevPeterssonSchroder2017}. There,
it is assumed $N_c = N/k$; however, the coarse grid then has $N_c+1$ points. Here, $N_c$ exactly denotes the number of coarse-grid
time points, at the expense of a slightly more complicated relation between $N$ and $N_c$.}
Then, blocks in an FC-partitioning of the matrix $A$ \eqref{eq:blocks} take the following form:
\begin{align*}
A_{ff} & =
\left[\begin{array}{@{}cccc : c : cccc@{}}
	I &&&&&&&&\text{} \\
		-\Phi & I &&&&&&& \\
		 & \ddots  & \ddots &&&&& \\
		&  & -\Phi & I &&&& \\ \hdashline
		& & & & \ddots \\ \hdashline 
	&&&&& I &&& \\
		&&&&& -\Phi & I & \\
		&&&&&  & \ddots & \ddots  \\
		&&&&& & & -\Phi & I
\end{array}\right],
	\hspace{3ex}
A_{fc} =
\left[\begin{array}{@{}c : c : c : c c@{}}
	-\Phi & & &  \mathbf{0} \\
		~~\mathbf{0} & & & \vdots \\ 
		~~\vdots & & &\vdots\\
		~~\mathbf{0} & & &\mathbf{0} \\ \hdashline 
	& \ddots & & \vdots \\ \hdashline
	&& -\Phi & \mathbf{0} \\ 
		&& ~~\mathbf{0} & \vdots \\
		&& ~~\vdots &\vdots  \\ 
		&& ~~\mathbf{0} & \mathbf{0}\\
\end{array}\right], \\
A_{cf} & =
\left[\begin{array}{@{}c c c c : c : cccc@{}}
	\mathbf{0} & ... && &&&&& \\ \hdashline
	\mathbf{0} & ... & \mathbf{0} & -\Phi &&&&& \\ \hdashline
	&&&&\ddots \\ \hdashline
	&&&&& \mathbf{0} & ... & \mathbf{0} & -\Phi 
\end{array}\right], 
	\hspace{12ex}
A_{cc} = \left[ \begin{array}{@{}c : c : c : c@{}}
	I & & & \\\hdashline
	& I & & \\ \hdashline
	& & \ddots & \\\hdashline
	& & & I 
\end{array}\right].
\end{align*}
Dotted lines are used to highlight the block nature, where each group of $k-1$ F-points are adjacent in the time domain, while all
C-points are disconnected. Next, further matrix forms that arise in residual and error propagation are derived:
{\small
\begin{align}
(A_{ff})^{-1} & = 
\left[\begin{array}{@{}cccc: c : cccc c@{}}
	I &&&&&&&& \\
		\Phi & I &&&&&&& \\
		\vdots & \ddots & \ddots &&&&& \\
		\Phi^{k-2} & ... & \Phi & I &&&& \\ \hdashline
	&&&& \ddots  & \text{ } \\ \hdashline
	&&&&& I &&&& \\
		&&&&& \Phi & I &&& \\
		&&&&& \vdots & \ddots & \ddots & \\
		&&&&& \Phi^{k-2} & ... & \Phi & I 
\end{array}\right],\label{eq:aff_inv}
\\
-A_{cf}(A_{ff})^{-1} & =
\left[\begin{array}{@{}c c c c : c : cccc@{}}
	\mathbf{0} & ... && &&&&& \\ \hdashline
	\Phi^{k-1} & \Phi^{k-2} &... & \Phi &&& \\ \hdashline
	& & & & \ddots \\ \hdashline
	&& &&& \Phi^{k-1} & \Phi^{k-2} &... & \Phi 
\end{array}\right],\label{eq:ideal_block}
\\
-(A_{ff})^{-1}A_{fc} & =
\left[\begin{array}{@{}c : c : c : c c@{}}
	\Phi & & &  \mathbf{0} \\
		\Phi^2 & & & \vdots \\ 
		~~\vdots & & &\vdots\\
		\Phi^{k-1} & & &\mathbf{0} \\ \hdashline 
	& \ddots & & \vdots \\ \hdashline
	&& \Phi & \mathbf{0} \\ 
		&& \Phi^2 & \vdots \\
		&& ~~\vdots &\vdots  \\ 
		&& \Phi^{k-1} & \mathbf{0}\\
\end{array}\right],
	\hspace{8ex}
A_{cf}(A_{ff})^{-1}A_{fc} = \left[ \begin{array}{@{}c : c : c : c@{}}
	\mathbf{0} & & & \\\hdashline
	\Phi^k & \mathbf{0} & & \\ \hdashline
	& \ddots & \ddots & \\\hdashline
	& & \Phi^k & \mathbf{0} 
\end{array}\right].\label{eq:relax_block}
\end{align}
}Note from \eqref{eq:relax_block} that the action of $XA_{cf}(A_{ff})^{-1}A_{fc}$ shifts all columns of $X$ to the left,
and scales all entries by $\Phi^k$. This will be useful in future derivations.

Recall that $RAP_{\textnormal{ideal}}$ is given by the Schur complement of $A$, independent of $R$. To be consistent with
\cite{DobrevKolevPeterssonSchroder2017,FalgoutFriedhoffKolevMaclachlanSchroder2014,Friedhoff:2015gt},
denote $A_\Delta := S_A = RAP_{\textnormal{ideal}}$. From above, it follows that
\begin{align}\label{eq:rap}
A_\Delta = A_{cc} - A_{cf}(A_{ff})^{-1}A_{fc} = \begin{bmatrix} I \\ -\Phi^k & I \\ & \ddots & \ddots \\ && -\Phi^k & I \end{bmatrix}.
\end{align}
Observe that the coarse-grid operator consists of taking $k$ time steps with the time-stepping function $\Phi$. Because this is
no cheaper to evaluate than $k$ individual steps of $\Phi$ -- that is, propagating $k$ steps on the fine grid -- a non-Galerkin coarse-grid
is used. Instead of taking $k$ time steps of size $\delta t$, corresponding to $\Phi^k$, $k$ steps are approximated by some operator $\Psi$,
\begin{align}
A_\Delta B_\Delta^{-1} & = \begin{bmatrix} I \\ -\Phi^k & I \\ & \ddots & \ddots \\ && -\Phi^k & I \end{bmatrix}
	\begin{bmatrix} I \\ \Psi & I \\ \Psi^2 & \Psi & I \\ \vdots & \vdots & \ddots & I \\ \Psi^{N_c-1} & \Psi^{N_c-2} & ... & \Psi & I \end{bmatrix} \label{eq:exact}\\
& = \begin{bmatrix} I \\ \Psi - \Phi^k & I \\ (\Psi-\Phi^k)\Psi & \Psi-\Phi^k & I \\ \vdots & \vdots & \ddots & \ddots\\ (\Psi-\Phi^k)\Psi^{N_c-2} & (\Psi-\Phi^k)\Psi^{N_c-3} & ... & \Psi-\Phi^k  & I\end{bmatrix}, \nonumber\\
I - A_\Delta B_\Delta^{-1} & = \textnormal{diag}(\Psi - \Phi^k)\begin{bmatrix} \mathbf{0} \\ I & \mathbf{0} \\ \Psi & I & \mathbf{0} \\ \vdots & \vdots & \ddots & \ddots \\ \Psi^{N_c-2} & \Psi^{N_c-3} & ... & I  & \mathbf{0} \end{bmatrix}, \label{eq:cgc_res}
\\
{I - B_\Delta^{-1} A_\Delta} & = {\begin{bmatrix} \mathbf{0} \\ I & \mathbf{0} \\ \Psi & I & \mathbf{0} \\ \vdots & \vdots & \ddots & \ddots \\ \Psi^{N_c-2} & \Psi^{N_c-3} & ... & I  & \mathbf{0} \end{bmatrix} } \textnormal{{diag}}{(\Psi - \Phi^k)}. \label{eq:cgc_err}
\end{align}
{Note that if $\Phi$ and $\Psi$ commute, then ${I - B_\Delta^{-1} A_\Delta} = I - A_\Delta B_\Delta^{-1}$.}

\section{The general case}\label{sec:gen}

This section derives a sequence of linear algebra lemmas, which are then used to present and prove a more precise version of
Theorem \ref{th:genF}. The underlying idea is that the $\ell^2$-norm of an operator is given by the
largest singular value, which is also equivalent to one divided by the smallest nonzero singular value of the respective
pseudoinverse. Here, we rely on block-Toeplitz matrix theory to place tight bounds on the maximum and minimum
singular values of operators related to error and residual propagation. 

From Section \ref{sec:conv:prop} and \eqref{eq:rf_pow}, \eqref{eq:ef_pow}, \eqref{eq:rfcf_pow}, and \eqref{eq:efcf_pow}, error- and
residual-propagation operators for $p$ iterations of two-level MGRiT, with F-relaxation and FCF-relaxation, take the following
forms: 
\begin{align}
\mathcal{E}_F^p & = P_{\textnormal{ideal}}({I - B_\Delta^{-1}A_\Delta} )^p,
	\hspace{5ex}\mathcal{E}_{FCF}^p = P_{\textnormal{ideal}}\left({(I - B_\Delta^{-1}A_\Delta })A_{cf}A_{ff}^{-1}A_{fc}\right)^p, \label{eq:err_pow}\\
\mathcal{R}_F^p & = (I - A_\Delta B_\Delta^{-1})^pR_{\textnormal{ideal}},
	\hspace{5ex}\mathcal{R}_{FCF}^p = \left((I - A_\Delta B_\Delta^{-1})A_{cf}A_{ff}^{-1}A_{fc}\right)^pR_{\textnormal{ideal}},\label{eq:res_pow}
\end{align}
where matrices are as in Section \ref{sec:conv:mat}. Notice that convergence over $p>1$ iterations in all cases is
determined by bounding either $\|(I - A_\Delta B_\Delta^{-1})^p\| < 1$ for F-relaxation or
$\|((I - A_\Delta B_\Delta^{-1})A_{cf}A_{ff}^{-1}A_{fc})^p\| < 1$ for FCF-relaxation. The leading (trailing) factor of $R_{\textnormal{ideal}}$
($P_{\textnormal{ideal}}$) accounts for the single iteration in Theorems\ref{th:genF} on which convergence
may not be observed. The following lemma proves that $\|R_{\textnormal{ideal}}\|,\|P_{\textnormal{ideal}}\| < \sqrt{k}$ if $\Phi$ is strongly stable.

\begin{lemma}[Bounds on $\|R_{\textnormal{ideal}}\|,\|P_{\textnormal{ideal}}\|$]\label{lem:ideal_bound}
Let $\|\Phi\| < 1$. Then,
\begin{align*}
\|R_{\textnormal{ideal}}\| = \|P_{\textnormal{ideal}}\| < \sqrt{k}.
\end{align*}
\end{lemma}
\begin{proof}
From \eqref{eq:ideal_block}, note that $\|R_{\textnormal{ideal}}\| = \sigma_{max}(R_{\textnormal{ideal}}) = \sqrt{\lambda_{max}(R_{\textnormal{ideal}}R_{\textnormal{ideal}}^*)}$, where 
$R_{\textnormal{ideal}}R_{\textnormal{ideal}}^*$ is block diagonal, with an identity in the first block, and the rest given by $\sum_{\ell=0}^{k-1} \Phi^\ell(\Phi^{\ell})^*$.
Then,
\begin{align*}
\|R_{\textnormal{ideal}}\| = \sqrt{\left\| \sum_{\ell=0}^{k-1} \Phi^\ell(\Phi^{\ell})^*\right\|} \leq \sqrt{\sum_{\ell=0}^{k-1} \left\|\Phi^\ell(\Phi^{\ell})^*\right\|}
	= \sqrt{\sum_{\ell=0}^{k-1} \left\|\Phi^\ell\right\|^2} < \sqrt{k}.
\end{align*}
An analogous derivation confirms that $\|P_{\textnormal{ideal}}\| < \sqrt{k}$. 
\end{proof}
Note that Lemma \ref{lem:ideal_bound} is not necessarily tight, but sufficient to prove that error cannot diverge
significantly in the $A^* A$- or $\ell^2$-norm in the first/last iteration.

\subsection{{Residual-propagation and $I - A_\Delta B_\Delta^{-1}$}}\label{sec:gen:res}

Now, we consider the maximum singular value of $I - A_\Delta B_\Delta^{-1}$ and
$(I-A_\Delta B_\Delta^{-1})A_{cf}A_{ff}^{-1}A_{fc}$, which arises in residual propagation.
From \eqref{eq:relax_block} and \eqref{eq:cgc_res}, it is clear that
both of these operators are block-Toeplitz matrices. Appealing to block-Toeplitz matrix theory, asymptotically (in $N_c$)
tight bounds can be placed on the maximum singular value by way of considering the operator's generator function. 
Let $\alpha_i$ denote the (potentially matrix-valued) Toeplitz coefficient for the $i$th diagonal of a (block) Toeplitz matrix,
where $\alpha_0$ is the diagonal, $\alpha_{-1}$ the first subdiagonal, and so on. Then the Toeplitz matrix corresponds
with a Fourier generator function,
\begin{align*}
{F}(x) & = \sum_{i=-\infty}^\infty \alpha_ie^{-\mathrm{i}x}.
\end{align*}
The following theorems introduce specific results from the field of block-Toeplitz operator theory, which prove important
in further derivations.

\begin{theorem}[Minimum eigenvalue of Hermitian block-Toeplitz operators \cite{Miranda:2000is,Serra:1998cl,Serra:1999ed}]
\label{th:toeplitz}
Let $T_N(F)$ be an $N\times N$ Hermitian block-Toeplitz matrix, with self-adjoint, continuous generating function
$F(x): [0,2\pi]\to\mathbb{C}^{m\times m}$, where $F(x) =F(x)^*$, and the minimum eigenvalue of $F(x)$ is not constant.
Then, the smallest eigenvalue of $T_N(F)$ is given by
\begin{align*}
\lambda_{\min}(T_N) = \min_{x\in[0,2\pi]} \lambda_{\min}(F(x)) + O(N^{-\alpha}),
\end{align*}
where $\alpha>0$ is the order of the highest-order zero in $x$ of 
\begin{align*}
\lambda_{\min}(F(x)) - \left[\min_{y\in[0,2\pi]} \lambda_{\min}(F(y)) \right].
\end{align*}
\end{theorem}

\begin{theorem}[Maximum singular value of block-Toeplitz operators \cite{Tilli:1998tl}]\label{th:toeplitz2}
Let $T_N(F)$ be an $N\times N$ block-Toeplitz matrix, with continuous generating function $F(x): [0,2\pi]\to\mathbb{C}^{m\times m}$.
Then, the maximum singular value is bounded above by
\begin{align*}
\sigma_{max}(T_N(F)) \leq \max_{x\in[0,2\pi]} \sigma_{max}(F(x)),
\end{align*}
 for all $N\in\mathbb{N}$.
\end{theorem}

Theorem \ref{th:toeplitz2} is now used in the following theorem to derive upper bounds on the maximum singular
values of interest. 

\begin{theorem}[Sufficient conditions]\label{th:suff}
Let $\Phi$ denote the fine-grid time-stepping operator and $\Psi$ denote the coarse-grid time-stepping operator,
with coarsening factor $k$, and $N_c$ coarse-grid time points. Assume that $\Phi$ satisfies an F-TAP$_1$ with 
respect to $\Psi$, with constant $\varphi_{F,1}$. Then,
\begin{align*}
\left\| I - A_\Delta B_\Delta^{-1}\right\| \leq \varphi_{F,1}\left(1 +\|\Psi^{N_c}\|\right).
\end{align*}
Similarly, assume that $\Phi$ satisfies an FCF-TAP$_1$ with respect to $\Psi$, with constant
$\varphi_{FCF,1}$. Then, 
\begin{align*}
\left\|(I-A_\Delta B_\Delta^{-1})A_{cf}A_{ff}^{-1}A_{fc}\right\| \leq \varphi_{FCF,1}\left(1 + \|\Phi^{-k}\Psi^{N_c}\Phi^k\|\right).
\end{align*}
\end{theorem}
\begin{proof}
Notice from \eqref{eq:cgc_res} that $I - A_\Delta B_\Delta^{-1}$ is a block-Toeplitz matrix with generating coefficients 
$\alpha_i = (\Psi-\Phi^k)\Psi^{-(1+i)}$ for $i=-1,...,-N_c$ and $\alpha_i = \mathbf{0}$ for $i\geq 0$. Let ${F}_F(x)$
denote this generating function. {First, note that if $\|\Psi^p\| < 1$ for some $p$, all eigenvalues of $\Psi$
must be less than one in magnitude. It then follows that $I - e^{\mathrm{i}x}\Psi$ is invertible for all $x$. If it were singular, then for some
$x$ and $\mathbf{v}$, we would have $e^{-\mathrm{i}x}\mathbf{v} = \Psi\mathbf{v}$, which contradicts that all eigenvalues of
$\Psi$ must be less than one in magnitude.} Then,
\begin{align*}
{F}_F(x) & = (\Psi-\Phi^k)\sum_{\ell=1}^{N_c} \Psi^{\ell-1}e^{\mathrm{i}\ell x} \\
	& = (\Psi-\Phi^k)e^{\mathrm{i}x}\sum_{\ell=0}^{N_c-1} \Psi^{\ell}e^{\mathrm{i}\ell x}, \\
& = e^{\mathrm{i}x} (\Psi-\Phi^k) (I - e^{\mathrm{i}N_cx}\Psi^{N_c})  (I - e^{\mathrm{i}x}\Psi)^{-1}. 
\end{align*}
Recall, under the assumption of an F-TAP$_1$, $\|(\Psi - \Phi^k)\mathbf{v}\| \leq \varphi_{F,1} \left[\min_{x\in[0,2\pi]}
\left\| (I - e^{\mathrm{i}x}\Psi )\mathbf{v}\right\| \right]$ for all $\mathbf{v}$. Theorem \ref{th:toeplitz2} then yields
\begin{align}
\begin{split}\label{eq:Fder}
\|I - A_\Delta B_\Delta^{-1}\| & \leq \max_{x\in[0,2\pi]} \sigma_{max}({F}_F(x)) \\
& = \max_{\substack{x\in[0,2\pi], \\ \mathbf{v}\neq \mathbf{0}}} \frac{\left\|(\Psi-\Phi^k)
	(I - e^{\mathrm{i}N_cx}\Psi^N_c)(I - e^{\mathrm{i}x}\Psi)^{-1}\mathbf{v}\right\|}{\|\mathbf{v}\|} \\
& \leq \max_{\mathbf{v}\neq \mathbf{0}} \frac{\left\|(\Psi-\Phi^k) \mathbf{v}\right\| + 
	\left\|(\Psi-\Phi^k)\Psi^{N_c}\mathbf{v}\right\|}{\min_{x\in[0,2\pi]}\left\|(I - e^{\mathrm{i}x}\Psi)\mathbf{v}\right\|} \\
& \leq \max_{\mathbf{v}\neq \mathbf{0}} \varphi_{F,1} + \varphi_{F,1}\frac{\min_{x\in[0,2\pi]}\left\|(I - e^{\mathrm{i}x}\Psi)
	\Psi^{N_c}\mathbf{v}\right\|}{\min_{x\in[0,2\pi]}\left\|(I - e^{\mathrm{i}x}\Psi)\mathbf{v}\right\|} \\
& \leq \max_{\mathbf{v}\neq \mathbf{0}} \varphi_{F,1} + \varphi_{F,1}\frac{\min_{x\in[0,2\pi]}\left\|\Psi^{N_c}(I - e^{\mathrm{i}x}\Psi)
	\mathbf{v}\right\|}{\min_{x\in[0,2\pi]}\left\|(I - e^{\mathrm{i}x}\Psi)\mathbf{v}\right\|} \\
& = \varphi_{F,1}(1 + \|\Psi^{N_c}\|).
\end{split}
\end{align}

A similar proof follows for the case of FCF-relaxation, where the generator function, ${F}_{FCF}$,
has coefficients $\alpha_i = (\Psi-\Phi^k)\Psi^{-(1+i)}\Phi^{k}$ for $i=-1,...,-N_c$ and $\alpha_i = \mathbf{0}$ for
$i\geq 0$. Then, by assumption of an FCF-TAP$_1$ with constant $\varphi_{FCF,1}$,
\begin{align*}
\left\|(I-A_\Delta B_\Delta^{-1})A_{cf}A_{ff}^{-1}A_{fc}\right\| & \leq \max_{x\in[0,2\pi]} \sigma_{max}({F}_{FCF}(x)) \\
& = \max_{\substack{x\in[0,2\pi], \\ \mathbf{v}\neq \mathbf{0}}} \frac{\left\|(\Psi-\Phi^k)
	(I - e^{\mathrm{i}Nx}\Psi^{N_c})(I - e^{\mathrm{i}x}\Psi)^{-1}\Phi^{k}\mathbf{v}\right\|}{\|\mathbf{v}\|} \\
& \leq \max_{\mathbf{v}\neq \mathbf{0}} \frac{\left\|(\Psi-\Phi^k)\mathbf{v}\right\| + 
	\left\|(\Psi-\Phi^k)\Psi^{N_c}\mathbf{v}\right\|}{\min_{x\in[0,2\pi]}\|\Phi^{-k}(I - e^{\mathrm{i}x}\Psi)\mathbf{v}\|} \\
& \leq \max_{\mathbf{v}\neq\mathbf{0}} \varphi_{FCF,1} + \varphi_{FCF,1} \frac{\min_{x\in[0,2\pi]}\left\|\Phi^{-k}(I - e^{\mathrm{i}x}\Psi)\Psi^{N_c}
	\mathbf{v}\right\|}{\min_{x\in[0,2\pi]}\|\Phi^{-k}(I - e^{\mathrm{i}x}\Psi)\mathbf{v}\|} \\
& = \max_{\mathbf{v}\neq\mathbf{0}} \varphi_{FCF,1} + \varphi_{FCF,1} \frac{\left\|\Phi^{-k}\Psi^{N_c}\Phi^k\mathbf{v}\right\|}{\|\mathbf{v}\|} \\
& = \varphi_{FCF,1}(1 + \|\Phi^{-k}\Psi^{N_c}\Phi^k\|).
\end{align*}
\end{proof}

Next, a more technical path is pursued, where the maximum singular values of $(I - A_\Delta B_\Delta^{-1})^p$ and
$((I-A_\Delta B_\Delta^{-1})A_{cf}A_{ff}^{-1}A_{fc})^p$ are analyzed by means of the minimum singular value of the
respective pseudoinverses. First, a pseudoinverse is derived for operators of the form $(I - A_\Delta B_\Delta^{-1})^p$ and
$((I-A_\Delta B_\Delta^{-1})A_{cf}A_{ff}^{-1}A_{fc})^p$, for $p\geq 1$. These pseuodinverses almost take the form of finite
block Toeplitz matrices, and we appeal again to Toeplitz matrix theory to bound the smallest nonzero singular value
from above.

First, some general pseudoinverses and their properties are derived. Let $f,g,$ and $h$ be invertible scalars or operators and define
\begin{align}
A_0 & = \begin{bmatrix} 
g \\ & g \\ & & g \\ & & & g \\ & & & &\ddots
\end{bmatrix}
\begin{bmatrix} 
\mathbf{0} \\ I & \mathbf{0} \\ 
f  & I & \mathbf{0} \\ 
f^2 & f & I & \mathbf{0} \\
\vdots & \ddots & \ddots & \ddots & \ddots
\end{bmatrix}
\begin{bmatrix} 
h \\ & h \\ & & h \\ & & & h \\ & & & &\ddots
\end{bmatrix} \label{eq:A0}\\
& = 
\begin{bmatrix} 
\mathbf{0} \\ gh & \mathbf{0} \\ 
gfh  & gh & \mathbf{0} \\ 
gf^2h & gfh & gh & \mathbf{0} \\
\vdots & \ddots & \ddots & \ddots & \ddots
\end{bmatrix}, 
\nonumber\\
A_1 & = \begin{bmatrix} 
g \\ & g \\ & & g \\ & & & g \\ & & & & \ddots
\end{bmatrix}
\begin{bmatrix} 
\mathbf{0} \\ 
\mathbf{0} & \mathbf{0} \\
I & \mathbf{0} & \mathbf{0} \\ 
f  & I & \mathbf{0} & \mathbf{0} \\ 
f^2 & f & I & \mathbf{0} & \mathbf{0} \\
\vdots & \ddots & \ddots & \ddots & \ddots & \ddots
\end{bmatrix}
\begin{bmatrix} 
h \\ & h \\ & & h \\ & & &h \\ & & & & \ddots
\end{bmatrix} \label{eq:A1}\\
 & =
\begin{bmatrix} 
\mathbf{0} \\
\mathbf{0} & \mathbf{0} \\
gh & \mathbf{0} & \mathbf{0} \\ 
gfh  & gh & \mathbf{0} & \mathbf{0} \\ 
gf^2h & gfh & gh & \mathbf{0} & \mathbf{0} \\
\vdots & \ddots & \ddots & \ddots & \ddots & \ddots
\end{bmatrix}.\nonumber
\end{align}

Note that \eqref{eq:A0} and \eqref{eq:A1} are general matrix forms, which encompass the coarse-grid correction
and diagonal blocks of two-grid residual and error propagation for MGRiT, with F- and FCF-relaxation; in particular,
$I - A_\Delta B_\Delta^{-1}$ \eqref{eq:cgc_res} takes the form of \eqref{eq:A0} and $(I - A_\Delta B_\Delta^{-1})A_{cf}A_{ff}^{-1}A_{fc}$
(\ref{eq:relax_block}, \ref{eq:cgc_res}) takes the form of \eqref{eq:A1}, with an additional zero row due to FCF-relaxation \eqref{eq:relax_block}. 
To construct pseudoinverses for operators
of these forms, recall the four properties that define a pseudoinverse: $AA^\dagger A = A$, $A^\dagger A A^\dagger = A^\dagger$,
$(AA^\dagger)^* = AA^\dagger$, and $(A^\dagger A)^* = A^\dagger A$. The subtle part of constructing a
pseudoinverse is preserving the (not full rank) image and kernel of $A$. However, matrices in \eqref{eq:A0} and \eqref{eq:A1}
have the nice property that they are full rank in the first $n-1$ or $n-2$ rows and columns, respectively. 
In the case of \eqref{eq:A0}, note that
\begin{align*}
\begin{bmatrix} 
\mathbf{0} \\ gh & \mathbf{0} \\ 
gfh  & gh & \mathbf{0} \\ 
gf^2h & gfh & gh & \mathbf{0} \\
\vdots & \ddots & \ddots & \ddots & \ddots
\end{bmatrix} 
& =
\begin{bmatrix} \mathbf{0} \\ & I \\  && I \\ &&& I \\ & & & & \ddots \end{bmatrix}
\begin{bmatrix} 
\mathbf{0} \\ gh & \mathbf{0} \\ 
gfh  & gh & \mathbf{0} \\ 
gf^2h & gfh & gh & \mathbf{0} \\
\vdots & \ddots & \ddots & \ddots & \ddots
\end{bmatrix} \\
&=
\begin{bmatrix} 
\mathbf{0} \\ gh & \mathbf{0} \\ 
gfh  & gh & \mathbf{0} \\ 
gf^2h & gfh & gh & \mathbf{0} \\
\vdots & \ddots & \ddots & \ddots & \ddots
\end{bmatrix}
\begin{bmatrix} I \\  & I \\ & &\ddots \\ &&& I \\ & & & & \mathbf{0}  \end{bmatrix}.
\end{align*}
Then, if we can build a matrix $\hat{A}_0^\dagger$ such that 
\begin{align}\label{eq:pinv_props}
A_0\hat{A}_0^\dagger = \begin{bmatrix} \mathbf{0} \\ & I \\  && I \\ &&& I \\ & & & & \ddots \end{bmatrix}, \hspace{3ex}
\hat{A}_0^\dagger A_0 = \begin{bmatrix} I \\  & I \\ & &\ddots \\ &&& I \\ & & & & \mathbf{0}  \end{bmatrix}, \hspace{3ex}
\hat{A}_0^\dagger  \begin{bmatrix} \mathbf{0} \\ & I \\  && I \\ &&& I \\ & & & & \ddots \end{bmatrix} = \hat{A}_0^\dagger,
\end{align}
it follows that all four properties of a pseudoinverse are satisfied. A similar result holds for $A_1$. We now have
all the tools needed to construct a pseudoinverse of $A_0$ and $A_1$, which is summarized in the following lemma.

\begin{lemma}\label{lem:pinv1}
Let $f$, $g$, and $h$ be invertible operators\footnote{{This is where Assumption 4, that $\Psi - \Phi^k$ is invertible,
comes in. If a pseudoinverse must be used for $g$ or $h$ instead of a formal inverse, the resulting pseudoinverses
of $A_0$ and $A_1$ do not take on such simple forms. For example, simply replacing $g^{-1}$ with $g^\dagger$ in
\eqref{eq:a0_dagger} fails the self-adjoint property $(A_0^\dagger A_0)^* = A_0^\dagger A_0$, although it does satisfy
the other three.}} and $A_0$ and $A_1$ be matrices defined as in \eqref{eq:A0} and \eqref{eq:A1},
respectively. Then, the unique pseudoinverses of $A_0$ and $A_1$ are given by
\begin{align}
A_0^\dagger
& =
\begin{bmatrix} 
\mathbf{0} \\ gh & \mathbf{0} \\ 
gfh  & gh & \mathbf{0} \\ 
gf^2h & gfh & gh & \mathbf{0} \\ \vdots & \ddots & \ddots & \ddots & \ddots
\end{bmatrix}^\dagger
=
\begin{bmatrix}
\mathbf{0} & h^{-1}g^{-1} \\
\mathbf{0} & -h^{-1}fg^{-1} & h^{-1}g^{-1} \\
\vdots & & \ddots & \ddots \\
& & & - h^{-1}fg^{-1} & h^{-1}g^{-1} \\ 
& & ... & \mathbf{0} & \mathbf{0} 
\end{bmatrix}, \label{eq:a0_dagger}
\\
A_1^\dagger & = 
\begin{bmatrix} 
\mathbf{0} \\
\mathbf{0} & \mathbf{0} \\
gh & \mathbf{0} & \mathbf{0} \\ 
gfh  & gh & \mathbf{0} & \mathbf{0} \\ 
gf^2h & gfh & gh & \mathbf{0} & \mathbf{0} \\
\vdots & \ddots & \ddots & \ddots & \ddots & \ddots
\end{bmatrix}^\dagger 
=
\begin{bmatrix}
\mathbf{0} & \mathbf{0} & h^{-1}g^{-1} \\
\mathbf{0} & \mathbf{0} & -h^{-1}fg^{-1} & h^{-1}g^{-1} \\
\vdots & \vdots & & \ddots & \ddots \\
& & & & - h^{-1}fg^{-1} & h^{-1}g^{-1} \\ 
& & & ... & \mathbf{0} & \mathbf{0}  \\
& & & ... & \mathbf{0} & \mathbf{0} 
\end{bmatrix}. \label{eq:a1_dagger}
\end{align}
\end{lemma}
\begin{proof}
The third relation in \eqref{eq:pinv_props} simply requires that the first column of $A_0^\dagger$ is zero. Working through
the system of equations established by the first two relations (and similar relations for $A_1$) yields the operators in
\eqref{eq:a0_dagger} and \eqref{eq:a1_dagger}.
\end{proof}

Notice that the pseudoinverse of $A_1$ is effectively equivalent to that of $A_0$, except with an additional zero row and
column. The only difference this leads to in the final results is an $O(1/\sqrt{N_c})$ perturbation vs. $O(1/\sqrt{N_c-1})$ perturbation. For
moderate to large $N_c$, this difference is arbitrary and, for simplicity's sake, F-relaxation and FCF-relaxation are both treated in the
form $A_0$ moving forward.

The following lemma generalizes this result, deriving the pseudoinverse for operators of
the form in \eqref{eq:A0} raised to powers. {In the context of MGRiT, this corresponds to powers of error and residual propagation,
which define how error and residual are propagated over multiple iterations. For non-normal operators $\Phi$ and $\Psi$, it is 
possible that, for example, $\|\mathcal{E}^p\| < \|\mathcal{E}\|^p$. In fact, it is possible that $\|\mathcal{E}\| > 1$ appears divergent,
but raising to powers results in a convergent method.}

\begin{lemma}[Pseudoinverse for matrix powers]\label{lem:pinv2}
Let $A_0$ be as in \eqref{eq:A0} and define the Toeplitz matrix
\begin{align*}
\mathcal{T}_0 =: \begin{bmatrix}
-h^{-1}fg^{-1} & h^{-1}g^{-1} \\
& -h^{-1}fg^{-1} & h^{-1}g^{-1} \\
 & & \ddots & \ddots \\
& & & - h^{-1}fg^{-1} & h^{-1}g^{-1} \\ 
& &  & & -h^{-1}fg^{-1}
\end{bmatrix}
\end{align*}
Then, for $p\geq 1$,
\begin{align*}
(A_0^p)^\dagger = \begin{bmatrix} I \\ & \mathbf{0}_{p\times p}\end{bmatrix} \mathcal{T}_0^p
	\begin{bmatrix} \mathbf{0}_{p\times p} \\ & I \end{bmatrix}.
\end{align*}
\end{lemma}
\begin{proof}
The case of $p=1$ was proven in Lemma \ref{lem:pinv1}, with pseudoinverse denoted by $A_0^\dagger$. Now let
$A_l^\dagger$ and $A_r^\dagger$ be tentative left and right pseudoinverses for $A_0^p$, $p>1$.
We start the proof by building $A_l^\dagger$ and $A_r^\dagger$ to satisfy certain properties of the
pseudoinverse, and conclude by merging them in a certain way to derive $(A_0^p)^\dagger$. 

First, note that $A_0^p$ is a strictly lower triangular matrix, with zeros on the diagonal and first $p-1$ subdiagonals.
Building on the proof of Lemma \ref{lem:pinv1} let us build $A_l^\dagger$ such that 
$A_l^\dagger A^p_0$ is diagonal, with zeros on the last $p$ entries, and the identity elsewhere (similar
to \eqref{eq:pinv_props}). This immediately satisfies two properties of a pseudoinverse,
$(A_l^\dagger A_0^p)^* = A_l^\dagger A_0^p$ and
$A_0^pA_l^\dagger A_0^p = A_0^p$. To do so, let us start by considering
$(A_0^\dagger)^p$ as a naive attempt at a pseudoinverse for $A_0^p$, and observe that
\begin{align*}
(A_0^\dagger)^p A_0^p & = { (A_0^\dagger)^{p-1} (A_0^\dagger A_0) A_0^{p-1}}
	= (A_0^\dagger)^{p-1} \begin{bmatrix} I \\ & 0\end{bmatrix}A_0A_0^{p-2}.
\end{align*}
{Here, $\begin{bmatrix} I \\ & 0\end{bmatrix}A_0$ simply eliminates the final row of $A_0$. Note from
\eqref{eq:a0_dagger} that when forming the product $A_0^\dagger M$, for some matrix $M$, only the
second-to-last row of $A_0^\dagger$ depends on the final row of $M$. Thus, if we consider
$A_0^\dagger \begin{bmatrix} I \\ & 0\end{bmatrix}A_0$, $A_0^\dagger$ will act as a (left) pseudoinverse
on all rows but the last one. Repeating a similar process by applying one
more power of $A_0^\dagger$ within the product $(A_0^\dagger)^pA_0^p$ yields}
\begin{align*}
(A_0^\dagger)^p A_0^p & = {(A_0^\dagger)^{p-2} \left(A_0^\dagger \begin{bmatrix} I \\ & 0\end{bmatrix}A_0\right) A_0^{p-2}}
	= (A_0^\dagger)^{p-2} \begin{bmatrix} I \\ \mathbf{e}_{0} & e_{1} \\ \mathbf{0} & 0 & 0\end{bmatrix}A_0A_0^{p-3},
\end{align*}
where $\mathbf{e}_0$ and $e_1$ are an error vector and scalar. Now, only the second-to-last and third-to-last rows of
$A_0^\dagger$ depend on the final two rows of $A_0$. Continuing this process to the power
of $p$, $(A_0^\dagger)^p A_0^p$ is given by an identity in the upper left block, zeros in the upper
right, and $p$ rows of error. To that end, define
\begin{align}\label{eq:tent_l}
A_l^\dagger A_0^p :=\left( \begin{bmatrix} I \\ & \mathbf{0}_{p\times p}\end{bmatrix}
	(A_0^\dagger)^p \right)A_0^p = \begin{bmatrix} I \\ & \mathbf{0}_{p\times p}\end{bmatrix}.
\end{align}
Note that we have defined $A_l^\dagger$ by forming $(A_0^\dagger)^p$ and eliminating the
last $p$ rows.

In an analogous process, define $A_r^\dagger$ as $(A_0^\dagger)^p$ with the \textit{first} $p$
\textit{columns} set to zero. Following similar steps as above, we arrive at
\begin{align}\label{eq:tent_r}
A_0^pA_r^\dagger := A_0^p \left( (A_0^\dagger)^p\begin{bmatrix} \mathbf{0}_{p\times p} \\ & I \end{bmatrix} \right) = \begin{bmatrix}\mathbf{0}_{p\times p} \\ & I\end{bmatrix}.
\end{align}
Now, recalling that $A_0^p$ is zero in the first $p$ rows and last $p$ columns, \eqref{eq:tent_l} and \eqref{eq:tent_r}
yield
\begin{align*}
A_l^\dagger A_0^p & = A_l^\dagger\begin{bmatrix}\mathbf{0}_{p\times p} \\ & I\end{bmatrix} A_0^p
	=\left( \begin{bmatrix} I \\ & \mathbf{0}_{p\times p}\end{bmatrix} (A_0^\dagger)^p \begin{bmatrix}\mathbf{0}_{p\times p} \\
	& I\end{bmatrix}\right)A_0^p = \begin{bmatrix} I \\ & \mathbf{0}_{p\times p}\end{bmatrix}, \\
A_0^pA_r^\dagger & = A_0^p\begin{bmatrix} I \\ & \mathbf{0}_{p\times p}\end{bmatrix}A_r^\dagger
	 = A_0^p \left( \begin{bmatrix} I \\ & \mathbf{0}_{p\times p}\end{bmatrix}(A_0^\dagger)^p\begin{bmatrix}
	 \mathbf{0}_{p\times p} \\ & I \end{bmatrix} \right) = \begin{bmatrix}\mathbf{0}_{p\times p} \\ & I\end{bmatrix}.
\end{align*}
Defining
\begin{align*}
(A_0^p)^\dagger := \begin{bmatrix} I \\ & \mathbf{0}_{p\times p}\end{bmatrix}(A_0^\dagger)^p\begin{bmatrix}
	 \mathbf{0}_{p\times p} \\ & I \end{bmatrix}
=  \begin{bmatrix} I \\ & \mathbf{0}_{p\times p}\end{bmatrix}\mathcal{T}_0^p\begin{bmatrix}
	 \mathbf{0}_{p\times p} \\ & I \end{bmatrix},
\end{align*}
it immediately follows that $(A_0^p)^\dagger$ satisfies the four properties of a pseudoinverse.
\end{proof}

Now, we introduce three lemmas on Toeplitz matrices, related to the pseudoinverse derived in Lemma \ref{lem:pinv2}.
{These lemmas derive the appropriate Toeplitz generating functions, and then provide a framework to bound the
smallest nonzero singular value from above.}

\begin{lemma}\label{lem:binom}
Consider a matrix of the form
\begin{align}\label{eq:T}
\mathcal{T}=\begin{bmatrix} -a & b \\ & - a & \ddots \\ & & \ddots & b \\ & & & -a \end{bmatrix},
\end{align}
where $\mathcal{T}$ is $n\times n$ and $a$ and $b$ some invertible coefficients or operators.  Then,
$\mathcal{T}^p$ is (block) Toeplitz, for $p\in\mathbb{N}$, $p < \lfloor n/2\rfloor$, {with Fourier generating
function given by
\begin{align}\label{eq:ab_coeff}
F_*(x) & = (-a+be^{\mathrm{i}x})^p.
\end{align}}
\end{lemma}
\begin{proof}
Because $\mathcal{T}$ is upper triangular, $\mathcal{T}^p$ is upper triangular for $p\geq 0$; furthermore, the stencil
expands one super-diagonal with each matrix multiplication, so $\mathcal{T}^p$ has exactly $p+1$ nonzero diagonals. 
Then, the defining Toeplitz coefficients of $\mathcal{T}^p$ are given by the $p+1$ nonzero entries in the first row of
$\mathcal{T}^p$, or, equivalently, the $p+1$ nonzero elements of $\mathbf{e}_0\mathcal{T}^p$, where $\mathbf{e}_0$
is the first canonical (block) basis vector, $(I,\mathbf{0},...,\mathbf{0})$. 

Now, consider a linear algebra framework to represent polynomials, where columns of $\mathcal{T}$ represent powers
of $x$. Then, given some coefficient vector $\mathbf{v}$, $\mathbf{v}\mathcal{T} = \mathbf{w}$, where $\mathbf{w}$
represents some polynomial $p(x)\sim\mathbf{w}$, and $\mathbf{w}_i$ corresponds to the polynomial coefficient of $x^i$. 
Then, for example, $\mathbf{e}_0\mathcal{T} \sim -a+bx \mapsto -a\mathbf{e}_0 + b\mathbf{e}_1$. Continuing,
\begin{align*}
\mathbf{e}_0\mathcal{T}^2 & = (-a\mathbf{e}_0 + b\mathbf{e}_1)\mathcal{T} \\
& \sim -a(-a+bx) + b(-ax+bx^2) \\
& = (-a+bx)^2 \\
& \mapsto a^2\mathbf{e}_0 - (ab+ba)\mathbf{e}_1 + b^2\mathbf{e}_2.
\end{align*}
Here, the $\ell$th polynomial coefficient in $(-a+bx)^2$, $x^\ell$, corresponds to the $\ell$th basis vector, $\mathbf{e}_\ell$,
for  $\ell=0,1,2$. By an inductive argument, this process continues, {with the $\ell$th element in the first row of $\mathcal{T}^p$
being given by the $\ell$th polynomial coefficient of $(-a+bx)^p$. Recalling that the Fourier generating function is given
by $F_*(x) = \sum_\ell \alpha_\ell e^{\mathrm{i}\ell x}$, where $\alpha_\ell$ is the $\ell$th Toepliz coefficient, we can simply replace
$x$ with $e^{\mathrm{i}x}$ to get
\begin{align*}
F_*(x) & = (-a+be^{\mathrm{i}x})^{p}.
\end{align*}}
\end{proof}

\begin{lemma}\label{lem:TsT}
Consider an $n\times n$ matrix $\mathcal{T}$ as in \eqref{eq:T}, and define
\begin{align}\label{eq:Tphat}
\widehat{\mathcal{T}}_p := \begin{bmatrix} I_{(n-p)\times(n-p)} \\ & \mathbf{0}_{p\times p} \end{bmatrix} \mathcal{T}^p,
\end{align}
for $p<\lfloor n/2\rfloor$. That is, $\widehat{\mathcal{T}}_p$ corresponds to the last $p$ rows eliminated from $\mathcal{T}^p$.
Then, $\widehat{\mathcal{T}}_p\widehat{\mathcal{T}}_p^*$ is Toeplitz in the upper $(n-p)\times (n-p)$ block and
zero elsewhere, {with real-valued Fourier generating function for the nonzero Toeplitz block given by}
\begin{align}\label{eq:Fgen}
F_p(x) & = (-a + be^{\mathrm{i}x})^p\left[(-a+be^{\mathrm{i}x})^p\right]^*.
\end{align}
\end{lemma}
\begin{proof}
{The proof proceeds by first deriving the Toeplitz coefficients for $\widehat{\mathcal{T}}_p\widehat{\mathcal{T}}_p^*$
based on those of $\mathcal{T}$, and showing that these coefficients correspond with the generating function 
$F_p(x) = (-a + be^{\mathrm{i}x})^p\left[(-a+be^{\mathrm{i}x})^p\right]^*$.}

From Lemma \ref{lem:binom}, $\mathcal{T}^p$ is Toeplitz and upper triangular with $p$ nonzero super-diagonals. Eliminating
the final $p$ rows of $\mathcal{T}^p$, it is straightforward to confirm that $\widehat{\mathcal{T}}_p\widehat{\mathcal{T}}_p^*$
is self-adjoint, Toeplitz in the upper left $(n-p)\times (n-p)$ block, and zero in the final $p$ rows and columns. 
By self-adjointness, the generating coefficients of $\widehat{\mathcal{T}}_p\widehat{\mathcal{T}}_p^*$ can be found
by considering the $p+1$ nonzero entries in the first row of $\widehat{\mathcal{T}}_p\widehat{\mathcal{T}}_p^*$ (and
their adjoints will be coefficients for the first $p+1$ rows). Let $\mathbf{e}_\ell$ denote the $\ell$th canonical basis
vector {and $\{\hat{\alpha}_\ell\}$ be the Fourier generating coefficients for $\mathcal{T}_p$. 
Then, for $\ell =0,...,p$, the Fourier generating coefficients for the nonzero Toeplitz block in
$\widehat{\mathcal{T}}_p\widehat{\mathcal{T}}_p^*$ are given by}
\begin{align}\label{eq:alphaF0}
\hat{\alpha}_\ell & = \left[\mathbf{e}_0\widehat{\mathcal{T}}_p\widehat{\mathcal{T}}_p^*\right]_\ell
= \sum_{j=0}^p \left[\widehat{\mathcal{T}}_p\right]_{0,j}\left[\widehat{\mathcal{T}}_p^*\right]_{j,\ell}
= \sum_{j=0}^p \left[\widehat{\mathcal{T}}_p\right]_{0,j}\left[\overline{\widehat{\mathcal{T}}_p}\right]_{\ell,j}
= \sum_{j=0}^{p-\ell} \left[\widehat{\mathcal{T}}_p\right]_{0,j+\ell}\left[\overline{\widehat{\mathcal{T}}_p}\right]_{0,j} 
= {\sum_{j=0}^{p-\ell} \alpha_{j+\ell}\alpha_j^*,}
\end{align}
where $\overline{\widehat{\mathcal{T}}_p}$ denotes the adjoint of operator entries, either the conjugate in the
case of a scalar matrix, or block adjoint in the case of a block matrix. The second-to-last equality follows from
the fact that in the $\ell$th column of $\widehat{\mathcal{T}}_p$, the first $\ell-1$ rows are zero. 

{Recall that the Fourier generating function for $\mathcal{T}_p$ is given by $F_*(x) =-a + be^{\mathrm{i}x}$ (Lemma 
\ref{lem:binom}), with coefficients $\{\alpha_\ell\}$. Then consider the block Toeplitz matrix associated with
generating function 
\begin{align*}
F_p(x) & = (-a + be^{\mathrm{i}x})^p\left[(-a+be^{\mathrm{i}x})^p\right]^* \\
& = \left(\sum_{j=0}^p \alpha_j e^{\mathrm{i}x}\right)\left(\sum_{j=0}^p \alpha_j^*e^{-\mathrm{i}x}\right) \\
& = \left( \alpha_0 + \alpha_1e^{\mathrm{i}x} + \alpha_2e^{2\mathrm{i}x} + ... \right)\left( \alpha_0^* + \alpha_1^*e^{-\mathrm{i}x} + \alpha_2^*e^{-2\mathrm{i}x} + ... \right).
\end{align*}
The corresponding coefficients can be obtained by gathering terms of $e^{\mathrm{i}jx}$, $j=-p,-(p-1),...,0,...,p$. This
leads to coefficients
\begin{align*}
\hat{\alpha}_\ell & = \sum_{\substack{j_0 - j_1 = \ell,\\j_0,j_1\leq p}} \alpha_{j_0}\alpha^*_{j_1}
	= \sum_{j=0}^{p-\ell} \alpha_{j+\ell}\alpha_j^*.
\end{align*}
Indeed, these are exactly the Toeplitz coefficients obtained by directly computing
$\widehat{\mathcal{T}}_p\widehat{\mathcal{T}}_p^*$ in \eqref{eq:alphaF0}, which completes the proof. 
}
\end{proof}

{\color{black}
\begin{remark}[Generating coefficients]
If $a$ and $b$ in \eqref{eq:T} commute, the Binomial Theorem gives a closed form for Toeplitz generating coefficients
$\{\alpha_\ell\}$ and $\{\hat{\alpha}_\ell\}$ by expanding $(a+be^{\mathrm{i}x})^p$. If $a$ and $b$ do not commute, there is a
generalization of the Binomial Theorem that takes the form
\begin{align*}
(a + be^{\mathrm{i}x})^p & = \sum_{\ell=0}^p \binom{p}{\ell}\left[\left(a + d_b\right)^\ell \mathbf{1}\right] (be^{\mathrm{i}x})^{p-\ell},
\end{align*}
where $\mathbf{1}$ denotes the identity on the underlying associative algebra, and $d_b$ is a derivation
defined by 
\begin{align*}
d_b(z) = be^{\mathrm{i}x}z - zbe^{\mathrm{i}x},
\end{align*}
for linear transformation $z$ \cite{wyss1980, binom}.
\end{remark}
}

\begin{lemma} \label{lem:sig}
Define $\widehat{\mathcal{T}}_p$ as in Lemma \ref{lem:TsT} \eqref{eq:Tphat} and define $\widehat{\mathcal{A}}_p$
similarly, in the form of the pseudoinverse from Lemma \ref{lem:pinv2},
\begin{align*}
\widehat{\mathcal{A}}_p & = \begin{bmatrix} I_{(n-p)\times(n-p)} \\ & \mathbf{0}_{p\times p} \end{bmatrix} \mathcal{T}^p
	\begin{bmatrix} \mathbf{0}_{p\times p} \\ & I_{(n-p)\times(n-p)} \end{bmatrix},
\end{align*}
that is, by setting the first $p$ columns and last $p$ rows of $\mathcal{T}^p$ equal to zero. Then,
\begin{align*}
\sigma_{\min}(\widehat{\mathcal{A}}_p) \leq \sigma_{\min}(\widehat{\mathcal{T}}_p),
\end{align*}
where $\sigma_{\min}$ denotes the minimum nonzero singular value.
\end{lemma}
\begin{proof}
\color{black}
Recall that $\widehat{\mathcal{T}}_p$ is upper triangular and zero in the last $p$ rows. Then, consider 
expressing $\widehat{\mathcal{T}}_p\widehat{\mathcal{T}}_p^*$ and $\widehat{\mathcal{A}}_p\widehat{\mathcal{A}}_p^*$ in
block form, with $\varepsilon$, $M_0$, and $M_1$ chosen to denote the nonzero blocks in $\widehat{\mathcal{T}}_p$:
\begin{align*}
\widehat{\mathcal{T}}_p\widehat{\mathcal{T}}_p^* & = 
	\begin{bmatrix}\varepsilon & M_0 \\ \mathbf{0}_{(n-2p)\times p} & M_1 \\ \mathbf{0}_{p\times p} & \mathbf{0}_{p\times(n-p)}\end{bmatrix}
	\begin{bmatrix}\varepsilon^* & \mathbf{0}_{p\times(n-2p)} & \mathbf{0}_{p\times p} \\ M_0^* & M_1^* & \mathbf{0}_{(n-p)\times p}\end{bmatrix} \\
& = \begin{bmatrix} \varepsilon\varepsilon^* + M_0M_0^* & M_0M_1^* & \mathbf{0}_{p\times p} \\
	M_1M_0^* & M_1M_1^* & \mathbf{0}_{(n-2p)\times p} \\
	\mathbf{0}_{p\times p} & \mathbf{0}_{p\times(n-2p)} & \mathbf{0}_{p\times p} \end{bmatrix}, \\
\widehat{\mathcal{A}}_p\widehat{\mathcal{A}}_p^* & = 
	 \begin{bmatrix} M_0M_0^* & M_0M_1^* & \mathbf{0}_{p\times p} \\
	M_1M_0^* & M_1M_1^* & \mathbf{0}_{(n-2p)\times p} \\
	\mathbf{0}_{p\times p} & \mathbf{0}_{p\times(n-2p)} & \mathbf{0}_{p\times p} \end{bmatrix}.
\end{align*}
Because we are interested in the minimum \textit{nonzero} singular value of $\widehat{\mathcal{T}}_p$, consider
the nonzero block in $\widehat{\mathcal{T}}_p\widehat{\mathcal{T}}_p^* - \widehat{\mathcal{A}}_p\widehat{\mathcal{A}}_p^*$,
given by
\begin{align}\label{eq:nnzblock}
\begin{bmatrix} \varepsilon\varepsilon^* + M_0M_0^* & M_0M_1^* \\ M_1M_0^* & M_1M_1^*\end{bmatrix} - 
	\begin{bmatrix} M_0M_0^* & M_0M_1^* \\ M_1M_0^* & M_1M_1^*\end{bmatrix}
	& = \begin{bmatrix} \varepsilon\varepsilon^* & \mathbf{0}_{p\times (n-2p)} \\ \mathbf{0}_{(n-2p)\times p} & \mathbf{0}_{(n-p)\times(n-p)}\end{bmatrix}
	\geq 0,
\end{align}
in a positive semi-definite sense. 

The proof then follows from a generalization of the monotonicity theorem or Weyl's inequality
\cite[Theorem 10.4.11]{bernstein2018scalar}. In particular, let $A$ and $B$ be Hermitian matrices. Then,
\begin{align}\label{eq:weyl}
\lambda_{\min}(A) + \lambda_{\min}(B) \leq \lambda_{\min}(A+B) \leq \lambda_{\min}(A)+\lambda_{max}(B).
\end{align}
Applying \eqref{eq:weyl} to \eqref{eq:nnzblock} yields
\begin{align*}
\lambda_{\min}\left(\begin{bmatrix} M_0M_0^* & M_0M_1^* \\ M_1M_0^* & M_1M_1^*\end{bmatrix}\right) &
	\leq \lambda_{\min}\left(\begin{bmatrix} \varepsilon\varepsilon^* + M_0M_0^* & M_0M_1^* \\ M_1M_0^* & M_1M_1^*\end{bmatrix}\right) - 
	\lambda_{\min}\left(\begin{bmatrix} \varepsilon\varepsilon^* & \mathbf{0}_{p\times (n-2p)} \\
		\mathbf{0}_{(n-2p)\times p} & \mathbf{0}_{(n-p)\times(n-p)}\end{bmatrix}\right) \\
& \leq \lambda_{\min}\left(\begin{bmatrix} \varepsilon\varepsilon^* + M_0M_0^* & M_0M_1^* \\ M_1M_0^* & M_1M_1^*\end{bmatrix}\right).
\end{align*}
To that end, the minimum nonzero eigenvalue of $\widehat{\mathcal{T}}_p\widehat{\mathcal{T}}_p^*$ provides an upper bound on
the minimum nonzero eigenvalue of $\widehat{\mathcal{A}}_p\widehat{\mathcal{A}}_p^*$, and the result follows because the singular
values of a matrix $M$ are given by the square root of the eigenvalues of $MM^*$. 

\end{proof}

We now have all the necessary tools to prove necessary conditions for convergence of MGRiT and Parareal. 

\begin{theorem}[Necessary conditions]\label{th:gen}
Let $\Phi$ denote the fine-grid time-stepping operator and $\Psi$ denote the coarse-grid time-stepping operator,
with coarsening factor $k$, and $N_c$ coarse-grid time points. Assume that $(\Psi-\Phi^k)$ is invertible and that
$\Phi$ satisfies an F-TAP$_1$ with respect to $\Psi$, with minimum constant $\varphi_{F,1}$. Then,
\begin{align*}
\left\|I - A_\Delta B_\Delta^{-1}\right\| & \geq \frac{\varphi_{F,1}}{1+O(1/\sqrt{N_c})}.
\end{align*}
If we further assume that $\Phi$ and $\Psi$ commute, that is, $\Phi\Psi = \Psi\Phi$, and that $\Phi$ satisfies an
F-TAP$_p$ with respect to $\Psi$, with minimum constant $\varphi_{F,p}$, then 
\begin{align*}
\left\|(I - A_\Delta B_\Delta^{-1})^p\right\| & \geq \frac{\varphi_{F,p}}{1+O(1/\sqrt{N_c})}.
\end{align*}

Similarly, assume that $(\Psi-\Phi^k)$ is invertible and that $\Phi$ satisfies an FCF-TAP$_1$ with respect to
$\Psi$, with constant minimum $\varphi_{FCF,1}$. Then
\begin{align*}
\left\|(I - A_\Delta B_\Delta^{-1})A_{cf}A_{ff}^{-1}A_{fc}\right\| \geq \frac{\varphi_{FCF,1}}{1+O(1/\sqrt{N_c})}.
\end{align*}
If we further assume that $\Phi$ and $\Psi$ commute, that is, $\Phi\Psi = \Psi\Phi$, and that $\Phi$ satisfies an
FCF-TAP$_p$ with respect to $\Psi$, with minimum constant $\varphi_{FCF,p}$, then 
\begin{align*}
\left\|\left((I - A_\Delta B_\Delta^{-1})A_{cf}A_{ff}^{-1}A_{fc}\right)^p\right\| & \geq \frac{\varphi_{FCF,p}}{1+O(1/\sqrt{N_c})}.
\end{align*}
\end{theorem}
\begin{proof}
To bound $(I - A_\Delta B_\Delta^{-1})^p$ and $\left((I - A_\Delta B_\Delta^{-1})A_{cf}A_{ff}^{-1}A_{fc}\right)^p$ in norm,
we note that the $\ell^2$-norm of an operator is given by its largest singular value, which is equal to one over the
smallest nonzero singular value of the operator's inverse or pseudoinverse.

Now, notice that these operators exactly take the form of $A_0$ \eqref{eq:A0} and $A_1$ \eqref{eq:A1}.
For F-relaxation, $f = \Psi$, $g = (\Psi-\Phi^k)$, and $h=I$, and for FCF relaxation, $f = \Psi$, $g = (\Psi-\Phi^k)$,
and $h=\Phi^k$. Lemma \ref{lem:pinv2} gives an exact pseudoinverse for powers of such operators, and Lemma
\ref{lem:sig} proves the minimum nonzero singular value of this pseudoinverse is bounded above by the minimum
singular value of the Toeplitz operator $\widehat{\mathcal{T}}_p$ \eqref{eq:Tphat}, with $a=h^{-1}fg^{-1}$ and $b=h^{-1}g^{-1}$
\eqref{eq:T}. Equivalently, we can consider the minimum nonzero eigenvalue of the corresponding normal equations.
Appealing to Lemma \ref{lem:TsT}, the Fourier generating functions for the nonzero Toeplitz block in these operators
are given by
\begin{align}
F_{F}(x,p) & = \Big(e^{\mathrm{i}x}\Psi(\Psi-\Phi^k)^{-1} - (\Psi-\Phi^k)^{-1}\Big)^p\Big(e^{\mathrm{i}x}\Psi(\Psi-\Phi^k)^{-1} - (\Psi-\Phi^k)^{-1}\Big)^{*^p}, \label{eq:gen_f} \\
F_{FCF}(x,p) & = \Big(e^{\mathrm{i}x}\Phi^{-k}\Psi(\Psi-\Phi^k)^{-1} - \Phi^{-k}(\Psi-\Phi^k)\Big)^p
	 \Big(e^{\mathrm{i}x}\Phi^{-k}\Psi(\Psi-\Phi^k)^{-1} - \Phi^{-k}(\Psi-\Phi^k)^{-1}\Big)^{*^p} , \label{eq:gen_fcf}
\end{align}
where $F_F(x,p)$ will lead to a bound on $(I - A_\Delta B_\Delta^{-1})^p$ (F-relaxation) and $F_{FCF}(x,p)$
will lead to a bound on $((I - A_\Delta B_\Delta^{-1})A_{cf}A_{ff}^{-1}A_{fc})^p$ (FCF relaxation). 

By Theorem \ref{th:toeplitz} we seek the infimum over $x$ of the minimum nonzero eigenvalue of $F_{F}(x,p)$ and $F_{FCF}(x,p)$.
Let $\lambda_k(A)$ and $\sigma_k(A)$ denote the $k$th eigenvalue and singular value of some operator $A$ and
consider the case of $p=1$ for $F_F(x,1)$:
\begin{align*}
\min_{\substack{x\in[0,2\pi],\\k}} \lambda_k(F_{F}(x,1)) & = \min_{\substack{x\in[0,2\pi],\\k}} \sigma_k\Big((e^{\mathrm{i}x}\Psi-I)(\Psi - \Phi^k)^{-1}\Big)^2 \\
& = \min_{\substack{x\in[0,2\pi], \\ \mathbf{v}\neq \mathbf{0}}} \frac{\left\| (e^{\mathrm{i}x}\Psi-I)(\Psi - \Phi^k)^{-1}\mathbf{v}\right\|^2}{\|\mathbf{v}\|^2} \\
& =  \min_{\substack{x\in[0,2\pi], \\ \mathbf{v}\neq \mathbf{0}}} \frac{\left\| (e^{\mathrm{i}x}\Psi - I)\mathbf{v}\right\|^2}{\|(\Psi - \Phi^k)\mathbf{v}\|^2}.
\end{align*}

Appealing to Theorem \ref{th:toeplitz},
\begin{align}
\left\|I - A_\Delta B_\Delta^{-1}\right\| & \geq \frac{1}{\sqrt{\min_{\substack{x\in[0,2\pi], \\ \mathbf{v}\neq \mathbf{0}}}
	\frac{\left\| (e^{\mathrm{i}x}\Psi - I)\mathbf{v}\right\|^2}{\|(\Psi - \Phi^k)\mathbf{v}\|^2} + O(1/N_c)}}
\geq \frac{1}{\min_{\substack{x\in[0,2\pi], \\ \mathbf{v}\neq \mathbf{0}}}
	\frac{\left\| (e^{\mathrm{i}x}\Psi - I)\mathbf{v}\right\|}{\|(\Psi - \Phi^k)\mathbf{v}\|} + O(1/\sqrt{N_c})} \nonumber\\ & \hspace{10ex}
= \max_{\mathbf{v}\neq \mathbf{0}} \frac{\|(\Psi - \Phi^k)\mathbf{v}\|}{\min_{x\in[0,2\pi]} \left\| (I - e^{\mathrm{i}x}\Psi )\mathbf{v}\right\|+O(1/\sqrt{N_c})}.\label{eq:lowbound}
\end{align}
By assumption of an F-TAP$_1$ with constant $\varphi_{F,1}$\footnote{Note that in \eqref{eq:tap1}, the leading constant in
the $O(1/\sqrt{N_c})$ terms changes; however, the change in constant is marginal for any moderate $N_c \gg O(1)$.}
\begin{align}\label{eq:tap1}
\|(\Psi - \Phi^k)\mathbf{v}\| \leq \varphi_{F,1}\left[\min_{x\in[0,2\pi]} \left\| (I - e^{\mathrm{i}x}\Psi )\mathbf{v}\right\| \right] =
	\frac{\varphi_{F,1}}{1+O(1/\sqrt{N_c})}\left[\min_{x\in[0,2\pi]} \left\| (I - e^{\mathrm{i}x}\Psi )\mathbf{v}\right\| + O(1/\sqrt{N_c}) \right].
\end{align}
Assuming that $\varphi_{F,1}$ is a tight bound, there exists some $\mathbf{v}$ such that \eqref{eq:tap1} holds with
equality. Then, plugging \eqref{eq:tap1} into \eqref{eq:lowbound},
\begin{align*}
\left\|I - A_\Delta B_\Delta^{-1}\right\| & \geq \frac{\varphi_{F,1}}{1+O(1/\sqrt{N_c})}.
\end{align*}
A similar derivation based on the assumption of an FCF-TAP$_1$ with constant $\varphi_{FCF,1}$ follows to bound
\begin{align*}
\left\|(I - A_\Delta B_\Delta^{-1})A_{cf}A_{ff}^{-1}A_{fc}\right\| \geq \frac{\varphi_{FCF,1}}{1+O(1/\sqrt{N_c})}.
\end{align*}

Finally, if $\Phi$ and $\Psi$ commute, then, for example, 
\begin{align*}
\Big(e^{\mathrm{i}x}\Psi(\Psi-\Phi^k)^{-1} - (\Psi-\Phi^k)^{-1}\Big)^p = (e^{\mathrm{i}x}\Psi - I)^p(\Psi-\Phi^k)^{-p}.
\end{align*}
Under the assumption of an F-TAP$_p$ and FCF-TAP$_p$ with constants $\varphi_{F,p}$ and $\varphi_{FCF,P}$,
respectively, for $p\geq 1$, an analogous derivation as used for $p=1$ yields the bounds 
\begin{align*}
\left\|(I - A_\Delta B_\Delta^{-1})^p\right\| & \geq \frac{\varphi_{F,p}}{1+O(1/\sqrt{N_c})}, \\
\left\|\left((I - A_\Delta B_\Delta^{-1})A_{cf}A_{ff}^{-1}A_{fc}\right)^p\right\| & \geq \frac{\varphi_{FCF,p}}{1+O(1/\sqrt{N_c})}.
\end{align*}
and similarly for the case of FCF-relaxation. 
\end{proof}

{Coupling Theorems \ref{th:suff} and \ref{th:gen} with the operator form of residual propagation for $p$
iterations \eqref{eq:res_pow} and the equivalence of $\|\mathcal{R}\| = \|\mathcal{E}\|_{A^*A}$ \eqref{eq:AsA}
completes the proof of Theorem \ref{th:genF}.
Recall from \eqref{eq:cgc_res} and \eqref{eq:cgc_err} that if $\Phi$ and $\Psi$ commute,
then $I - B_\Delta^{-1} A_\Delta = I - A_\Delta B_\Delta^{-1}$, which proves Corollary \ref{cor:err_res}. }

{\color{black}
\subsection{Error-propagation and $I -B_\Delta^{-1} A_\Delta $}\label{sec:gen:err}

This section provides proofs of Theorem \ref{th:genF2}. The framework developed in previous sections
allows for a streamlined presentation. First, sufficient conditions for convergence of error in the $\ell^2$-norm (based
on $B_\Delta^{-1}$ as a left approximate inverse) are presented. 

\begin{theorem}[Sufficient conditions ($\ell^2$-error)]\label{th:suff2}
Let $\Phi$ denote the fine-grid time-stepping operator and $\Psi$ denote the coarse-grid time-stepping operator,
with coarsening factor $k$, and $N_c$ coarse-grid time points. Assume that $\Phi$ satisfies an F-ITAP with 
respect to $\Psi$, with constant $\tilde{\varphi}_{F}$. Then,
\begin{align*}
\left\| I - B_\Delta^{-1}A_\Delta\right\| \leq \tilde{\varphi}_{F}\left(1 +\|\Psi^{N_c}\|\right).
\end{align*}
Similarly, assume that $\Phi$ satisfies an FCF-ITAP with respect to $\Psi$, with constant
$\tilde{\varphi}_{FCF}$. Then, 
\begin{align*}
\left\|(I- B_\Delta^{-1}A_\Delta)A_{cf}A_{ff}^{-1}A_{fc}\right\| \leq \tilde{\varphi}_{FCF}\left(1 + \|\Psi^{N_c}\|\right).
\end{align*}
\end{theorem}
\begin{proof}
The proof is analogous to that of Theorem \ref{th:suff}. 
Notice from \eqref{eq:cgc_err} that $I - B_\Delta^{-1}A_\Delta $ is a block-Toeplitz matrix with generating coefficients 
$\alpha_i = \Psi^{-(1+i)}(\Psi-\Phi^k)$ for $i=-1,...,-N_c$ and $\alpha_i = \mathbf{0}$ for $i\geq 0$. Following Theorem \ref{th:suff},
the generating function is given by
\begin{align*}
{F}_F(x) & = e^{\mathrm{i}x}(I - e^{\mathrm{i}Nx}\Psi^N) (I - e^{\mathrm{i}x}\Psi)^{-1} (\Psi-\Phi^k).
\end{align*}
Recall, under the assumption of an F-ITAP, $\tilde{\varphi}_{F} \|\mathbf{v}\| \leq \left[\max_{x\in[0,2\pi]} \left\| (I - e^{\mathrm{i}x}\Psi )^{-1}(\Psi - \Phi^k)\mathbf{v}\right\| \right]$ for all $\mathbf{v}$, with equality for some $\mathbf{v}$. Theorem \ref{th:toeplitz2} then yields
\begin{align*}
\|I - B_\Delta^{-1}A_\Delta\| & \leq \max_{x\in[0,2\pi]} \sigma_{max}({F}_F(x)) \\
& = \max_{\substack{x\in[0,2\pi], \\ \mathbf{v}\neq \mathbf{0}}} \frac{\left\|
	(I - e^{\mathrm{i}Nx}\Psi^N)(I - e^{\mathrm{i}x}\Psi)^{-1}(\Psi-\Phi^k)\mathbf{v}\right\|}{\|\mathbf{v}\|} \\
& \leq \max_{\substack{x\in[0,2\pi], \\ \mathbf{v}\neq \mathbf{0}}} \frac{\left(1 + \|\Psi^{N_c}\|\right)
	\left\|(I - e^{\mathrm{i}x}\Psi)^{-1}(\Psi-\Phi^k)\mathbf{v}\right\|}{\|\mathbf{v}\|} \\
& = \tilde{\varphi}_{F}\left(1 + \|\Psi^{N_c}\|\right).
\end{align*}

A similar proof follows for the case of FCF-relaxation, where the generator function, ${F}_{FCF}$,
has coefficients $\alpha_i = \Psi^{-(1+i)}(\Psi-\Phi^k)\Phi^{k}$ for $i=-1,...,-N_c$ and $\alpha_i = \mathbf{0}$ for
$i\geq 0$. By assumption of an FCF-ITAP with constant $\tilde{\varphi}_{FCF}$, the result follows.
\end{proof}

Now, we present a similar result to Theorem \ref{th:gen}, which provides necessary conditions for convergence
of error in the $\ell^2$-norm. 

\begin{theorem}[Necessary conditions ($\ell^2$-error)]\label{th:gen2}
Let $\Phi$ denote the fine-grid time-stepping operator and $\Psi$ denote the coarse-grid time-stepping operator,
with coarsening factor $k$, and $N_c$ coarse-grid time points. Assume that $(\Psi-\Phi^k)$ is invertible and that
$\Phi$ satisfies an F-ITAP with respect to $\Psi$, with constant $\tilde{\varphi}_{F}$. Then,
\begin{align*}
\left\|I - B_\Delta^{-1}A_\Delta\right\| & \geq \frac{\tilde{\varphi}_{F}}{1+O(1/\sqrt{N_c})}.
\end{align*}
Similarly, assume that $(\Psi-\Phi^k)$ is invertible and that $\Phi$ satisfies an FCF-ITAP with respect to
$\Psi$, with constant $\tilde{\varphi}_{FCF}$. Then
\begin{align*}
\left\|(I - B_\Delta^{-1}A_\Delta)A_{cf}A_{ff}^{-1}A_{fc}\right\| \geq \frac{\tilde{\varphi}_{FCF}}{1+O(1/\sqrt{N_c})}.
\end{align*}
\end{theorem}
\begin{proof}
This proof is analogous to that of Theorem \ref{th:gen}, this time with $g = I, f = \Psi$, and $h = \Psi - \Phi^k$ for
F-relaxation, and $h = (\Psi - \Phi^k)\Phi^k$ for FCF-relaxation. Similar to \eqref{eq:gen_f} and \eqref{eq:gen_fcf},
the Fourier generating functions of interest are then given by 
\begin{align*}
F_{F}(x) & = \Big(e^{\mathrm{i}x}(\Psi-\Phi^k)^{-1}\Psi - (\Psi-\Phi^k)^{-1}\Big)\Big(e^{\mathrm{i}x}(\Psi-\Phi^k)^{-1}\Psi - (\Psi-\Phi^k)^{-1}\Big)^{*}, \\
F_{FCF}(x) & = \Big(e^{\mathrm{i}x}\Psi(\Psi-\Phi^k)^{-1}\Phi^{-k} - \Phi^{-k}(\Psi-\Phi^k)\Big)
	 \Big(e^{\mathrm{i}x}\Psi(\Psi-\Phi^k)^{-1}\Phi^{-k} - \Phi^{-k}(\Psi-\Phi^k)^{-1}\Big)^{*}.
\end{align*}
Following the algebraic steps in the proof of Theorem \ref{th:gen} completes the proof.
\end{proof}

}
\section{The diagonalizable case}\label{sec:diag}

So far results have been derived in terms of the time-stepping operators $\Phi$ and $\Psi$. In this section, we assume that
$\Phi$ and $\Psi$ commute and are diagonalizable. In general this corresponds to the spatial operator being diagonalizable,
which holds for many parabolic-type problems, among others. The purposes of this section are:
\begin{enumerate}
\item Derive exact bounds on convergence for diagonalizable operators.
\item Show that theory developed in this paper is, in some sense, a strengthening and generalization of previous results in
\cite{DobrevKolevPeterssonSchroder2017}.
\end{enumerate}

If $\Phi$ and $\Psi$ commute and are diagonalizable, this means that they are diagonalizable with the same eigenvectors,
or simultaneously diagonalizable. Under the assumption of simultaneous diagonalizability, certain bounds can be derived
on functions of $\Phi$ and $\Psi$. Let $\Phi = U\Lambda U^{-1}$, where $\Lambda_{ii} = \lambda_i$, for $i=1,...,n$, is a
diagonal matrix consisting of the eigenvalues of $\Phi$, and columns of $U$ are the corresponding eigenvectors. Similarly, let
$\Psi = U\Xi U^{-1}$, for diagonal matrix $\Xi$, where $\Xi_{ii} = \mu_i$ are the eigenvalues of $\Psi$. {Note, in this
section subscript $i$ corresponds to eigenvalue index, as opposed to iteration number as used previously.} Now let $\mathcal{A}$
be some matrix operator, where each entry is a rational function of $\Phi$ and $\Psi$,
\begin{align}
\mathcal{A}(\Phi,\Psi)&  = 
\begin{bmatrix}
a_{00}(\Phi,\Psi) & a_{01}(\Phi,\Psi) & ...\\
a_{10}(\Phi,\Psi) & a_{11}(\Phi,\Psi) & ... \\
\vdots & \vdots  & \ddots
\end{bmatrix}
=
\begin{bmatrix}
U \\ & U \\ & & \ddots
\end{bmatrix}
\begin{bmatrix}
a_{00}(\Lambda,\Xi) & a_{01}(\Lambda,\Xi) & ...\\
a_{10}(\Lambda,\Xi) & a_{11}(\Lambda,\Xi) & ... \\
\vdots & \vdots  & \ddots
\end{bmatrix}
\begin{bmatrix}
U^{-1} \\ & U^{-1} \\ & & \ddots
\end{bmatrix}.\label{eq:genA}
\end{align}
Denote $\mathbf{U}$ as the block diagonal matrix of eigenvectors, $U$, in \eqref{eq:genA}, define $\mathcal{P}$
as the orthogonal permutation matrix such that $\mathcal{P}\mathcal{A}(\Lambda,\Xi)\mathcal{P}^T$ is block diagonal, with blocks
given by $\mathcal{A}(\lambda_i,\mu_i)$, and let $\widetilde{U} = \mathbf{U}\mathcal{P}^T$. Then, 
\begin{align}\label{eq:u*u}
\|\mathcal{A}(\Phi,\Psi)\|_{(\widetilde{U}\widetilde{U}^*)^{-1}} & = \sup_{\mathbf{x} \neq \mathbf{0}}
	\frac{\left\| \widetilde{U}^{-1}\mathbf{U}\mathcal{P}^T\mathcal{P}\mathcal{A}(\Lambda,\Xi)\mathcal{P}^T\mathcal{P}
	\mathbf{U}^{-1}\mathbf{x}\right\|}{\|\widetilde{U}^{-1}\mathbf{x}\|^2} 
= \sup_{\mathbf{x} \neq \mathbf{0}} \frac{\left\| \mathcal{P}\mathcal{A}(\Lambda,\Xi)\mathcal{P}^T\mathbf{x}\right\|}{\|\mathbf{x}\|^2} 
= \sup_i \|\mathcal{A}(\lambda_i,\mu_i)\|.
\end{align}
Thus, the $(\widetilde{U}\widetilde{U}^*)^{-1}$-norm of $\mathcal{A}(\Phi,\Psi)$ can be computed by maximizing the
norm of $\mathcal{A}$ over eigenvalue indices of $\Phi$ and $\Psi$. In the case that $\Phi$ and $\Psi$ are a normal matrices, $U$ is unitary
and the $(\widetilde{U}\widetilde{U}^*)^{-1}$-norm reduces to the standard Euclidean 2-norm. More generally, we have
the relation 
\begin{align}\label{eq:norml2}
\frac{1}{\kappa(U)} \left(\sup_i \|\mathcal{A}(\lambda_i,\mu_i)^k\| \right) \leq \|\mathcal{A}(\Phi,\Psi)^k\| \leq
	\kappa(U) \left( \sup_i \|\mathcal{A}(\lambda_i,\mu_i)^k\| \right),
\end{align}
where $\kappa(U)$ denotes the matrix condition number of $U$.\footnote{A similar modified norm also occurs in
the case of integrating in time with a mass matrix \cite{DobrevKolevPeterssonSchroder2017}.}

Here, we are interested in $\mathcal{A}$ corresponding to the error- and residual-propagation operators of MGRiT, 
$\mathcal{R}$ and $\mathcal{E}$ (see \eqref{eq:r_pref}, \eqref{eq:e_pref}, \eqref{eq:r_fcf}, and \eqref{eq:e_fcf}).
For notation, let, for example, $\mathcal{R}_F(\lambda_i,\mu_i)$ denote residual
propagation for F-relaxation \eqref{eq:r_pref} operating on eigenvalues $\lambda_i$ and $\mu_i$ as opposed to
operators $\Phi$ and $\Psi$. Convergence of MGRiT requires $\|\mathcal{R}^p\|,\|\mathcal{E}^p\| \to 0$ with iteration
$p$; to that end, bounding $\sup_i \|\mathcal{R}(\lambda_i,\mu_i)^p\| < 1$ for all $i$ provides necessary and sufficient
conditions for $\|\mathcal{R}(\Phi,\Psi)^p\|\to 0$ with $p$, and similarly for $\mathcal{E}(\lambda_i,\mu_i)$.

\subsection{Necessary conditions}\label{sec:diag:nec}

First, let us extend the necessary conditions for $p$ iterations (Theorems \ref{th:necF}, \ref{th:necFCF}, and \ref{th:gen})
to the diagonalizable case. For notation, let, for example, $\left[I - A_\Delta B_\Delta^{-1}\right]_i$ denote the $N_x\times N_x$ matrix of
$I - A_\Delta B_\Delta^{-1}$ evaluated at the $i$th eigenmode of $\Phi$ and $\Psi$, where $\Phi$ and $\Psi$ are
$N_x\times N_x$. Then, by assumption of simultaneous diagonalizability {and the fact that
$I - A_\Delta B_\Delta^{-1} = I - B_\Delta^{-1}A_\Delta$ when $\Phi$ and $\Psi$ commute,}
\begin{align}\label{eq:i_norm}
\begin{split}
\|(I - B_\Delta^{-1} A_\Delta)^p\|_{(UU^*)^{-1}} & = \|(I - A_\Delta B_\Delta^{-1})^p\|_{(UU^*)^{-1}} \\
	& = \sup_i \left\|\left[I - A_\Delta B_\Delta^{-1}\right]_i^p\right\|, \\
\left\|\left((I -  B_\Delta^{-1}A_\Delta)A_{cf}A_{ff}^{-1}A_{fc}\right)^p\right\|_{(UU^*)^{-1}} & =
	\left\|\left((I - A_\Delta B_\Delta^{-1})A_{cf}A_{ff}^{-1}A_{fc}\right)^p\right\|_{(UU^*)^{-1}} \\ &=
	\sup_i \left\|\left[(I - A_\Delta B_\Delta^{-1})A_{cf}A_{ff}^{-1}A_{fc}\right]_i^p\right\|.
\end{split}
\end{align}
Now, we can follow the derivation in Section \ref{sec:gen}. For a pseduoinverse (Lemmas \ref{lem:pinv1} and \ref{lem:pinv2}),
we have $f = \mu_i$, $g = \mu_i - \lambda_i^k$, and, for F-relaxation, $h = 1$, while for FCF-relaxation, $h = \lambda_i^k$.
Because $f,g$, and $h$ now commute, we can remove a leading factor of $g^{-1}h^{-1}$, and we are interested in the
smallest nonzero singular value of (Lemma \ref{lem:pinv2})
\begin{align}\label{eq:pinv_mu}
(\mathcal{A}_0^p)^\dagger = \frac{1}{\mu_i(\mu_i-\lambda_i^k)} \begin{bmatrix} I \\ & \mathbf{0}_{p\times p}\end{bmatrix}
	\mathcal{T}_0^p \begin{bmatrix} \mathbf{0}_{p\times p} \\ & I \end{bmatrix}, \hspace{3ex} \textnormal{where}\hspace{3ex}
\mathcal{T}_0 = \begin{bmatrix} - \mu_i & 1 \\ & -\mu_i & \ddots \\ && \ddots & 1 \\ & & & -\mu_i\end{bmatrix}.
\end{align}
Following the further derivations in Section \ref{sec:gen}, the minimum nonzero singular value of $(\mathcal{A}_0^p)^\dagger$
is bounded above by the minimum eigenvalue of $\widehat{\mathcal{T}}_p\widehat{\mathcal{T}}_p^*$, where $a = \mu_i$
and $b = 1$ (see \eqref{eq:T}, \eqref{eq:Tphat}, and Lemmas \ref{lem:binom}, \ref{lem:TsT}, and \ref{lem:sig}).
The Fourier generating function for $\widehat{\mathcal{T}}_p\widehat{\mathcal{T}}_p^*$ from Lemma \ref{lem:TsT}
\eqref{eq:Fgen} now takes the form
\begin{align*}
F_{p}(x) & = \left(1+|\mu_i|^2-(\mu_ie^{-\mathrm{i}x}+\overline{\mu}_ie^{\mathrm{i}x})\right)^p \\
& = \Big(1+\Rea(\mu_i)^2 + \Ima(\mu_i)^2- 2\Rea(\mu_i)\cos(x) - 2\Ima(\mu_i)\sin(x)\Big)^p.
\end{align*}
To obtain the minimum of $F_p(x)$, note that
\begin{align}
F_{p}'(x) & = p\Big(2\Rea(\mu_i)\sin(x)-2\Ima(\mu_i)\cos(x)\Big)\Big(1+|\mu_i|^2-(\mu_ie^{-\mathrm{i}x}+\overline{\mu}_ie^{\mathrm{i}x})\Big)^{p-1} \nonumber\\
& = p\Big(2\Rea(\mu_i)\sin(x)-2\Ima(\mu_i)\cos(x)\Big)\Big( (\sin(x)-\Ima(\mu_i))^2 + (\cos(x)-\Rea(\mu_i))^2\Big).\label{eq:roots}
\end{align}
The first term in \eqref{eq:roots} has real roots given by $n\pi + \arctan(\Ima(\mu_i)/\Rea(\mu_i))$, for $n\in\mathbb{Z}$.
Any other roots of $F_{p}'(x)$ must satisfy
\begin{align*}
 (\sin(x)-\Ima(\mu_i))^2 + (\cos(x)-\Rea(\mu_i))^2 = 0,
\end{align*}
or, equivalently, $\sin(x)=\Ima(\mu_i)$ and $\cos(x)=\Rea(\mu_i)$. Suppose this holds. Solving for $\tilde{x} = \arcsin(\Ima(\mu_i))$,
we must also have $\cos(\tilde{x}) -\Rea(\mu_i) = 0$, which implies $\Rea(\mu_i) = \sqrt{1-\Ima(\mu_i)^2}$. However, then
$|\mu_i| = 1$, which violates the assumption that $|\mu_i| < 1$. It follows that the only real roots of $F'_{p}(x)$ are given by
$\hat{x} := n\pi + \arctan(\Ima(\mu_i)/\Rea(\mu_i))$, for $n\in\mathbb{Z}$.

Given $F_p(x)$ is continuous and $2\pi$-periodic, the infimum of $F_p(x)$ over $x\in\mathbb{R}$ is attained at one of
the roots of $F'_p(x)$. It is easily confirmed that the infimum is achieved for even $n$, $\hat{x}_{\min} := 2\hat{n}\pi +
\arctan(\Ima(\mu_i)/\Rea(\mu_i))$, where $\hat{n}\in\mathbb{Z}$, which yields
\begin{align*}
\inf_{x\in\mathbb{R}} F_{p}(x) & = \Big(1+\Rea(\mu_i)^2 + \Ima(\mu_i)^2- 2\Rea(\mu_i)\cos(\hat{x}) - 2\Ima(\mu_i)\sin(\hat{x})\Big)^p \\
& = \left( 1+\Rea(\mu_i)^2 + \Ima(\mu_i)^2 - 2\frac{\Rea(\mu_i)}{\sqrt{1+\frac{\Ima(\mu_i)}{\Rea(\mu_i)}}} 
	- 2\frac{\Ima(\mu_i)^2/\Rea(\mu_i)}{\sqrt{1+\frac{\Ima(\mu_i)}{\Rea(\mu_i)}}} \right)^p \\
& = \left( 1+\Rea(\mu_i)^2 + \Ima(\mu_i)^2 - 2\frac{\Ima(\mu_i)^2 + \Rea(\mu_i)^2}{\sqrt{\Ima(\mu_i)^2 + \Rea(\mu_i)^2}} \right)^p \\
& = \Big( 1 + |\mu_i|^2 - 2|\mu_i|\Big)^p \\
& = (1 - |\mu_i|)^{2p}.
\end{align*}
Noting that $F_p(x) - (1 - |\mu_i|)^{2p}$ has a zero of order two at $\hat{x}_{\min}$, it follows from Theorem \ref{th:toeplitz}
that $\lambda_{\min}(\widehat{\mathcal{T}}_p\widehat{\mathcal{T}}_p^*) = (1 - |\mu_i|)^{2p} + O(1/N_c^2)$. Note the faster
convergence in $N_c$ of eigenvalues to the infimum over $F(x)$ in the diagonalizable case, $O(1/N_c^2)$, compared
with the general case in Section \ref{sec:gen}, where the first-order root in $x$ led to convergence $O(1/N_c)$. 

This discussion is summarized in the following theorem on necessary conditions for convergence.

\begin{theorem}[Necessary conditions -- the diagonalizable case]\label{th:nec_diag}
Let $\Phi$ denote the fine-grid time-stepping operator and $\Psi$ denote the coarse-grid time-stepping operator,
with coarsening factor $k$, and $N_c$ coarse-grid time points. Assume that $\Phi$ and $\Psi$ commute and are
diagonalizable, with eigenvectors as columns of $U$. Then, for number of iterations $p\geq 1$
\begin{align*}
\|(I - A_\Delta B_\Delta^{-1})^p\|_{(UU^*)^{-1}} & \geq \sup_i \frac{|\mu_i - \lambda_i^k|^p}{\sqrt{(1 - |\mu_i|)^{2p} + O(1/N_c^2)}}
	> \sup_i \frac{|\mu_i - \lambda_i^k|^p}{(1 - |\mu_i|)^{p} + O(1/N_c)}, \\
\left\|\left((I - A_\Delta B_\Delta^{-1})A_{cf}A_{ff}^{-1}A_{fc}\right)^p\right\|_{(UU^*)^{-1}} & \geq \sup_i
	\frac{|\lambda_i^{kp}||\mu_i - \lambda_i^k|^p}{\sqrt{(1 - |\mu_i|)^{2p} + O(1/N_c^2)}} 
	> \sup_i \frac{|\lambda_i^{kp}||\mu_i - \lambda_i^k|^p}{(1 - |\mu_i|)^{p} + O(1/N_c)}.
\end{align*}
\end{theorem}
\begin{proof}
The proof follows from \eqref{eq:i_norm}, Theorem \ref{th:toeplitz}, the minimum derived for $F_p(x)$,
and the fact that $x+y > \sqrt{x^2+y^2}$ for $x,y >0$.
\end{proof}

\subsection{Sufficient conditions}\label{sec:diag:suff}

Now consider sufficient conditions for convergence under the assumption that $\Phi$ and $\Psi$ commute
and are diagonalizable. To do so, we consider the minimum nonzero singular value of $(I - A_\Delta B_\Delta^{-1})^\dagger$.
As in Section \ref{sec:diag:nec} and \eqref{eq:pinv_mu}, we can pull out leading constants, form the normal 
equations with the remaining block, and reduce the problem to finding the minimum nonzero singular value of
the following symmetric positive semi-definite matrix
\begin{align}\label{eq:almost_toe}
\begin{bmatrix}
{0} & {0} & {0} & ... \\
{0} & 1 + |\mu_i|^2 & -\overline{\mu}_i \\
{0} & -\mu_i & \ddots & \ddots \\
\vdots && \ddots  & 1 + |\mu_i|^2  & -\overline{\mu}_i \\
&& & -\mu_i & 1  
\end{bmatrix}.
\end{align}
The nonzero block is a single-entry perturbation to a symmetric tridiagonal Toeplitz matrix, for which we can place tight
bounds on the minimum nonzero eigenvalue (see Appendix, Lemma \ref{lem:tridiag} and \eqref{eq:s_bounds}). Using the bounds derived
in \eqref{eq:s_bounds} leads to the following theorem.

\begin{theorem}[Tight bounds -- the diagonalizable case]\label{th:diag_tight}
Let $\Phi$ denote the fine-grid time-stepping operator and $\Psi$ denote the coarse-grid time-stepping operator,
with coarsening factor $k$, and $N_c$ coarse-grid time points. Assume that $\Phi$ and $\Psi$ commute and are
diagonalizable, with eigenvectors as columns of $U$. Then,
\begin{align}\label{eq:v_bounds}
\begin{split}
\sup_i \frac{|\mu_i - \lambda_i^k|}{\sqrt{ (1 - |\mu_i|)^2 + \frac{\pi^2|\mu_i|}{N_c^2}}}
	& \leq \|I - A_\Delta B_\Delta^{-1}\|_{(UU^*)^{-1}}
	\leq \sup_i  \frac{|\mu_i - \lambda_i^k|}{\sqrt{(1 - |\mu_i|)^2 + \frac{\pi^2|\mu_i|}{6N_c^2}} }, \\
\sup_i \frac{|\lambda_i^k||\mu_i - \lambda_i^k|}{\sqrt{ (1 - |\mu_i|)^2 + \frac{\pi^2|\mu_i|}{N_c^2}}}
	& \leq \left\|(I - A_\Delta B_\Delta^{-1})A_{cf}A_{ff}^{-1}A_{fc}\right\|_{(UU^*)^{-1}}
	\leq \sup_i  \frac{|\lambda_i^k||\mu_i - \lambda_i^k|}{\sqrt{(1 - |\mu_i|)^2 + \frac{\pi^2|\mu_i|}{6N_c^2}} }.
\end{split}
\end{align}
Furthermore, for $p\geq 1$
\begin{align}\label{eq:p_bounds}
\begin{split}
\left\|(I - A_\Delta B_\Delta^{-1})^p\right\|_{(UU^*)^{-1}} & = \sup_i  \frac{|\mu_i - \lambda_i^k|^p}{\sqrt{(1 - |\mu_i|)^{2p} + O(1/N_c^2)} } , \\
\left\|\left((I - A_\Delta B_\Delta^{-1})A_{cf}A_{ff}^{-1}A_{fc}\right)^p\right\|_{(UU^*)^{-1}} & =
	\sup_i  \frac{|\lambda_i^k|^p|\mu_i - \lambda_i^k|^p}{\sqrt{(1 - |\mu_i|)^{2p} + O(1/N_c^2)} }.
\end{split}
\end{align}
\end{theorem}
\begin{proof}
The single-iteration bounds follow immediately from \eqref{eq:i_norm} and Lemma \ref{lem:tridiag} \eqref{eq:s_bounds}. 
Applying the sub-multiplicative norm property to \eqref{eq:v_bounds} yields an upper bound on $p$ iterations, and Theorem
\ref{th:nec_diag} yields lower bounds, each of which take the form, for example, with F-relaxation,
\begin{align*}
\sup_i  \frac{|\mu_i - \lambda_i^k|^p}{\sqrt{(1 - |\mu_i|)^{2p} + O(1/N_c^2)} }.
\end{align*}
This completes the proof.
\end{proof}

\subsection{Relation to the TAP}\label{sec:diag:tight}

Returning to the general theoretical framework, suppose that $\Phi$ and $\Psi$ commute and are diagonalizable,
$\Phi = U\Lambda U^{-1}$ and $\Psi = U\Xi U^{-1}$. Further suppose that $\Phi$ satisfies an F-TAP$_p$ with respect to
$\Psi$, with constant $\varphi_{F,p}$, \textit{in the $(UU^*)^{-1}$-norm}. This is equivalent to saying that there
exists a constant $\varphi_{F,p}$ such that for all $\mathbf{v}$, 
\begin{align}
\|(\Psi - \Phi^k)^p\mathbf{v}\|_{(UU^*)^{-1}} & \leq \varphi_{F,p} \left[\min_{x\in[0,2\pi]} \| (I - e^{\mathrm{i}x}\Psi)^p\mathbf{v}\|_{(UU^*)^{-1}} \right], \nonumber\\
\Longleftrightarrow\hspace{6ex}
\|(\Xi - \Lambda^k)^pU^{-1}\mathbf{v}\| & \leq \varphi_{F,p} \left[\min_{x\in[0,2\pi]} \| (I - e^{\mathrm{i}x}\Xi)^pU^{-1}\mathbf{v}\| \right].\label{eq:almost}
\end{align}
Now note that if $\Phi$ and $\Psi$ are diagonalizable, the eigenvectors form a basis, and any vector $\mathbf{v}$ can be
written as a linear combination of eigenvectors of $\Phi,\Psi$, where $\mathbf{v} = \sum_{\ell=1}^{N_x} \alpha_\ell \mathbf{u}_\ell$.
Then, because $U^{-1}\mathbf{u}_i = \mathbf{e}_i$, where $\mathbf{e}_i$ is the $i$th canonical basis vector, \eqref{eq:almost}
reduces to 
\begin{align*}
\sum_{\ell=0}^{N_x-1} \alpha_\ell |\mu_\ell - \lambda_\ell^k|^p \leq \varphi_{F,1} \sum_{\ell=0}^{N_x-1} (1-|\mu_\ell|)^p.
\end{align*}
Note that this is only satisfied for all $\mathbf{v}$ if, for every eigenvalue index $i$,
\begin{align*}
|\mu_i - \lambda_i^k|^p \leq \varphi_{F,p} (1 - |\mu_i|)^p
	\hspace{5ex}\Longleftrightarrow\hspace{5ex}
|\mu_i - \lambda_i^k| \leq \varphi_{F,p} (1 - |\mu_i|).
\end{align*}
Indeed, this is exactly the F-TEAP introduced in Section \ref{sec:res:comm}. A similar property holds for FCF-relaxation, which
is summarized in the following proposition.

\begin{proposition}[Equivalent approximation properties]
The F-TEAP is the same as the F-TAP$_p$, in the $(UU^*)^{-1}$-norm, for arbitrary $p$, and the FCF-TEAP is the same as the FCF-TAP$_p$ in the $(UU^*)^{-1}$-norm, for arbitrary $p$. If $\Phi$ and $\Psi$ are normal, the two types of approximation property are identical.
\end{proposition}
\begin{proof}
The proof follows from the above discussion.
\end{proof}

We are now ready to present the final result.

\begin{theorem}[Tight bounds -- multiple iterations]\label{th:tight}
Let $\Phi$ denote the fine-grid time-stepping operator and $\Psi$ denote the coarse-grid time-stepping operator,
with coarsening factor $k$, and $N_c$ coarse-grid time points. Assume that $\Phi$ and $\Psi$ commute and are
diagonalizable, with eigenvectors given as columns of $U$. Suppose that $\Phi$ satisfies an F-TEAP with
respect to $\Psi$, with constant $\varphi_{F}$. Then, for $p\geq 1$,
\begin{align*}
\left\|(I - A_\Delta B_\Delta^{-1})^p\right\|_{(UU^*)^{-1}}^2 & = \varphi_{F}^{2p} - O(1/N_c^2).
\end{align*}
Similarly, suppose that $\Phi$ satisfies an FCF-TEAP with
respect to $\Psi$, with constant $\varphi_{FCF}$. Then, for $p\geq 1$,
\begin{align*}
\left\|\left((I - A_\Delta B_\Delta^{-1})A_{cf}A_{ff}^{-1}A_{fc}\right)^p\right\|_{(UU^*)^{-1}} ^2 & = \varphi_{FCF}^{2p} - O(1/N_c^2).
\end{align*}

\end{theorem}
\begin{proof}
By assumption of the T-FEAP and Theorem \ref{th:diag_tight}, to order $O(1/N_c^2)$,
\begin{align*}
\left\|(I - A_\Delta B_\Delta^{-1})^p\right\|_{(UU^*)^{-1}}^2 & = \sup_i  \frac{\left(|\mu_i - \lambda_i^k|^p\right)^2}{(1 - |\mu_i|)^{2p} + O(1/N_c^2) }
	 = \frac{\varphi_{F}^{2p} (1 - |\mu_{i_{max}}|)^{2p} }{(1 - |\mu_{i_{max}}|)^{2p} + O(1/N_c^2)}.
\end{align*}
Note, equality holds in the second relation because there exists some $i={i_{max}}$ such that $\varphi_{F}$ is tight.
A simple Taylor/Laurent series argument about $N_c=\infty$ confirms that for $N_c > 1/(1 - |\mu_i|)^{p}$,
\begin{align*}
\frac{1}{(1 - |\mu_i|)^{2p} + O(1/N_c^2)} = \frac{1}{(1 - |\mu_i|)^{2p}} - \frac{1}{N_c^2(1 - |\mu_i|)^{4p}} + \frac{1}{N_c^4(1 - |\mu_i|)^{6p}} -...
\end{align*}
This yields $\left\|(I - A_\Delta B_\Delta^{-1})^p\right\|_{(UU^*)^{-1}}^2 = \varphi_{F}^{2p} - O(1/N_c^2).$
An analogous proof confirms the result for FCF-relaxation. 
\end{proof}

Interestingly, despite having $1 - |\mu_i|$ in the denominator, it is typically \textit{not} eigenvalues $|\mu_i| \approx 1$ for which
the maximum $\varphi_F$ is obtained \cite{DobrevKolevPeterssonSchroder2017}. To that end, the $O(1/N_c^2)$ in Theorem
\ref{th:tight} will generally be fairly small, except for potentially in the case of very large $p$. How tight the bounds
are clearly depends on the size of $p$ and $N_c$; in practice, however, Theorem \ref{th:tight} indicates that
the upper bound on convergence in the $(UU^*)^{-1}$-norm will generally not improve in later iterations, that is,
$\|\mathcal{E}^p\|_{(UU^*)^{-1}} \approx \|\mathcal{E}\|_{(UU^*)^{-1}}^p$. 

These results also lead to a corollary, which proves that, in some cases, the bounds derived in
\cite{DobrevKolevPeterssonSchroder2017} are asymptotically exact in $N_c$, in a single-iteration sense. 

\begin{corollary}[Sharp matrix inequalities]
The matrix norm inequality $\|M\|^2_2\leq\|M\|_1\|M\|_\infty$ \cite[Fact 11.9.27]{bernstein2018scalar} is
asymptotically exact, for $\left[I - A_\Delta B_\Delta^{-1}\right]_i$ and
$\left[(I - A_\Delta B_\Delta^{-1})A_{cf}A_{ff}^{-1}A_{fc}\right]_i$; that is,
\begin{align}\label{eq:holder}
\lim_{N_c\to\infty} \Big\|\left[I - A_\Delta B_\Delta^{-1}\right]_i\Big\|_2^2 =
	\Big\|\left[I - A_\Delta B_\Delta^{-1}\right]_i\Big\|_1\Big\|\left[I - A_\Delta B_\Delta^{-1}\right]_i\Big\|_\infty,
\end{align}
and likewise for $\left\|\left[(I - A_\Delta B_\Delta^{-1})A_{cf}A_{ff}^{-1}A_{fc}\right]_i\right\|_2^2$.

Moreover, 
\begin{align*}
\lim_{N_c\to\infty}\left\|(I - A_\Delta B_\Delta^{-1})^p\right\|_{(UU^*)^{-1}} & =
	\lim_{N_c\to\infty}\left\|I - A_\Delta B_\Delta^{-1}\right\|_{(UU^*)^{-1}}^p \\
	& =  \left(\sup_i  \frac{|\mu_i - \lambda_i^k|}{1 - |\mu_i|}\right)^p , \\
\lim_{N_c\to\infty} \left\|\left((I - A_\Delta B_\Delta^{-1})A_{cf}A_{ff}^{-1}A_{fc}\right)^p\right\|_{(UU^*)^{-1}} & =
	\lim_{N_c\to\infty} \left\|(I - A_\Delta B_\Delta^{-1})A_{cf}A_{ff}^{-1}A_{fc}\right\|_{(UU^*)^{-1}}^p \\
	& =\left(\sup_i \frac{|\lambda_i^k||\mu_i - \lambda_i^k|}{1 - |\mu_i| } \right)^p.
\end{align*}
\end{corollary}
\begin{proof}
The inequality in \eqref{eq:holder} was used in \cite{DobrevKolevPeterssonSchroder2017} to establish bounds
\begin{align*}
\left\|\left[I - A_\Delta B_\Delta^{-1}\right]_i\right\| & \leq \frac{ |\mu_i-\lambda_i^k|(1-|\mu_i|^{N_c-1})}{1 - |\mu_i|}, \\
\left\|\left[(I - A_\Delta B_\Delta^{-1})A_{cf}A_{ff}^{-1}A_{fc}\right]_i\right\| & \leq\frac{ |\lambda_i^k||\mu_i-\lambda_i^k|(1-|\mu_i|^{N_c-2})}{1 - |\mu_i|}.
\end{align*}
It is easily verified that as $N_c\to\infty$, these bounds asymptote as in \eqref{eq:v_bounds}.
The second result follows from a limiting argument applied to Theorem \ref{th:tight}.
\end{proof}

\section{Time-dependent operators}\label{sec:time}

\subsection{The general case}\label{sec:time:gen}

The previous section focused on the specific case of commuting, diagonalizable time-stepping operators. This section
moves to the more general setting of (almost) arbitrary, linear time-stepping operators. In particular, we drop the
assumption that $\Phi$ and $\Psi$ are fixed for all time points, allowing for $\Phi$ and $\Psi$ to be time-dependent
operators. Much of the theory so far has, on some level, been based on Toeplitz matrix theory. Allowing for time-dependent
operators leads to non-Toeplitz block matrices, and such theory does not apply. Indeed, without some further assumptions
or knowledge of $\Phi$ and $\Psi$, in general results cannot be extended to the time-dependent setting. However, this section
shows that the pseudoinverse of $I-A_\Delta B_\Delta^{-1}$ derived in Lemma \ref{lem:pinv1} can indeed be extended to
the time-dependent setting. Although bounds for its minimum singular value are not clear, the resulting bi-diagonal matrix
is still more amenable to analysis than the dense lower-triangular matrix of $I-A_\Delta B_\Delta^{-1}$ \eqref{eq:cgc_res}.

{Note, in the time-dependent setting, $I - B_\Delta^{-1}A_\Delta \neq I - A_\Delta B_\Delta^{-1}$, even if $\Phi$ and $\Psi$
commute. To that end, this section considers $I-A_\Delta B_\Delta^{-1}$, corresponding to residual propagation in the
$\ell^2$-norm and error propagation in the $A^*A$-norm. However, similar results can be derived for error in the
$\ell^2$-norm based on analogous derivations applied to $I - B_\Delta^{-1}A_\Delta$. }

For preliminary notation, assume we are considering $N$ time points and a coarse grid of $N_c$ time points. Then
the linear system corresponding to time integration \eqref{eq:system} takes the generalized form of
\begin{align}\label{eq:system2}
A\mathbf{u} = \begin{bmatrix} I \\ -\Phi_1 & I \\ & -\Phi_2 & I \\ & & \ddots & \ddots \\
	& & & -\Phi_{N-1} & I \end{bmatrix}
	\begin{bmatrix} \mathbf{u}_0 \\ \mathbf{u}_1\\ \mathbf{u}_2 \\ \vdots \\ \mathbf{u}_{N-1}\end{bmatrix} = \mathbf{f},
\end{align}
where $\Phi_i$ denotes $\Phi$ evaluated at time point $t_i$. As in the time-independent case, there is a closed
form for inverses with the form of \eqref{eq:system2}, which will prove useful for further derivations:
\begin{align}\label{eq:geninv}
\begin{bmatrix} I \\ -\Phi_1 & I \\ & -\Phi_2 & I \\ & & \ddots & \ddots \\ & & & -\Phi_{N-1} & I \end{bmatrix}^{-1}
	& =
\begin{bmatrix} I \\ \Phi_1 & I \\
	\Phi_2\Phi_1 & \Phi_2 & I \\
	\Phi_3\Phi_2\Phi_1 & \Phi_3\Phi_2 & \Phi_3 & I \\ 
	\vdots & \vdots  & & \ddots  & \ddots \\
	\Phi_{N-1}...\Phi_1 & \Phi_{N-1}...\Phi_2 & ... & ... & \Phi_{N-1} & I
\end{bmatrix}.
\end{align}
Excusing the slight abuse of notation, define $\Phi_{i}^{j} := \Phi_i\Phi_{i-1}...\Phi_j$. Then, using the inverse in
\eqref{eq:geninv} and analogous matrix derivations as in Section \ref{sec:conv:mat}, leads to a Schur complement
coarse grid given by 
\begin{align*}
A_\Delta = \begin{bmatrix} I \\ -\Phi_k^1 & I \\ & -\Phi_{2k}^{k+1} & I \\ & & \ddots & \ddots \\
	& & & -\Phi_{(N_c-1)k}^{(N_c-2)k+1} & I \end{bmatrix}.
\end{align*}
Notice that, despite the more complicated notation, the Schur-complement coarse-grid operator in the time-dependent
case does exactly what it does in the time-independent case \eqref{eq:rap}: it takes exactly $k$ steps on the fine grid,
in this case using the appropriate sequence of time-dependent operators. Let $\Psi_i$ denote the non-Galerkin
approximation to $\Phi_{ik}^{(i-1)k+1}$. Then, the operator we are primarily interested in for error and residual
propagation, $I-A_\Delta B_\Delta ^{-1}$, is given by
{\footnotesize
\begin{align}\label{eq:genI}
I - A_\Delta B_\Delta^{-1} & =
\begin{bmatrix} 
	\mathbf{0} \\
	\Phi_k^1-\Psi_1 & \mathbf{0} \\
	(\Phi_{2k}^{k+1}-\Psi_2)\Psi_1 & (\Phi_{2k}^{k+1}-\Psi_2) & \mathbf{0} \\
	(\Phi_{3k}^{2k+1}-\Psi_3)\Psi_2^1 & (\Phi_{3k}^{2k+1}-\Psi_3)\Psi_2 & \Phi_{3k}^{2k+1}-\Psi_3& \mathbf{0} \\
	\vdots & \vdots & & \ddots & \ddots \\
	\left(\Phi_{(N_c-1)k}^{(N_c-2)k+1}-\Psi_{N_c-1}\right)\Psi_{N_c-2}^1 & 
		 & ... &  & \Phi_{(N_c-1)k}^{(N_c-2)k+1}-\Psi_{N_c-1} & \mathbf{0}
\end{bmatrix}.
\end{align}
}

Moreover, the pseudoinverses derived in Lemma \ref{lem:pinv1} can be extended to the time-dependent case.

\begin{lemma}[Time-dependent pseudoinverse]\label{lem:pinv3}
Let $\{\Phi_i\}_{i=1}^{N-1}$, and $\{\Psi_i\}_{i=1}^{N_c-1}$ denote two sets of operators and, for notation, define
$\Phi_{i}^{j} := \Phi_i\Phi_{i-1}...\Phi_j$. Assume that $\left(\Phi_{ik}^{(i-1)k+1} - \Psi_i\right)$ is invertible, for $i=1,...,N_c-1$,
and define $I - A_\Delta B_\Delta^{-1}$ as in \eqref{eq:genI}. Then,
{\footnotesize
\begin{align}
\begin{split}\label{eq:gen_pinv}
(I - &A_\Delta B_\Delta^{-1})^\dagger = \\ &
\begin{bmatrix}
	\mathbf{0} & (\Phi_k^1-\Psi_1)^{-1} \\
	\mathbf{0} & -\Psi_1(\Phi_k^1-\Psi_1)^{-1} & (\Phi_{2k}^{k+1}-\Psi_2)^{-1} \\
	 & & \ddots & \ddots \\
	 && & -\Psi_{N_c-2}\left(\Phi_{(N_c-2)k}^{(N_c-3)k+1}-\Psi_{N_c-2}\right)^{-1} &  \left(\Phi_{(N_c-1)k}^{(N_c-2)k+1}-\Psi_{N_c-1}\right)^{-1} \\
	 & & & \mathbf{0} & \mathbf{0} 
\end{bmatrix}.
\end{split}
\end{align}
}

\end{lemma}
\begin{proof}
Following from the proof of Lemma \ref{lem:pinv1}, it is easy to confirm from \eqref{eq:genI} and \eqref{eq:gen_pinv} that
$(I - A_\Delta B_\Delta^{-1})^\dagger$ satisfies the four properties of a pseuduinverse. 
\end{proof}

As in the time-dependent case, we seek the minimum nonzero singular value of $(I - A_\Delta B_\Delta^{-1})^\dagger$,
which is equivalent to the maximum singular value (and $\ell^2$-norm) of $I - A_\Delta B_\Delta^{-1}$. This can be
expressed as a minimization over a linear combination of operators as follows:
\begin{align}
\sigma_{\min}&\left((I-A_\Delta B_\Delta^{-1})^\dagger\right)^2 = \min_{\mathbf{v}\not\in\textnormal{ker}\left((I-A_\Delta B_\Delta^{-1})^\dagger\right)}
	\frac{\left\|(I-A_\Delta B_\Delta^{-1})^\dagger\mathbf{v}\right\|^2}{\|\mathbf{v}\|^2} \nonumber\\
& = \min_{\mathbf{v}_i,i=1,...,(N_c-1)} \frac{\left\| (\Phi_k^1-\Psi_1)^{-1}\mathbf{v}_1\right\|^2 + \sum_{i=1}^{N_c-2} \left\| 
	(\Phi_{(i+1)k}^{ik+1}-\Psi_{i+1})^{-1}\mathbf{v}_{i+1} - \Psi_i (\Phi_{ik}^{(i-1)k+1}-\Psi_i)^{-1}\mathbf{v}_i\right\|^2}
	{\sum_{i=1}^{N_c-1}\|\mathbf{v}_i\|^2} \nonumber\\
& = \min_{\mathbf{v}_i,i=1,...,(N_c-1)} \frac{\left\| \mathbf{v}_1\right\|^2 + \sum_{i=1}^{N_c-2} \left\| 
	\mathbf{v}_{i+1} - \Psi_i\mathbf{v}_i\right\|^2}
	{\sum_{i=1}^{N_c-1}\left\|(\Phi_{ik}^{(i-1)k+1}-\Psi_i)\mathbf{v}_i\right\|^2}. \label{eq:lowbound2}
\end{align}
Further analysis likely requires some knowledge on $\{\Phi_i\}$ and $\{\Psi_i\}$. In particular, now we are letting
$\{\Phi_i\}$ and $\{\Psi_i\}$ be completely arbitrary operators. In practice, there is typically some continuity
in how operators change between time steps, that is, $\Phi_i$ and $\Phi_{i+1}$ are similar in some sense. 

\subsection{The diagonalizable case}\label{sec:time:diag}

Finally, suppose that $\Phi$ and $\Psi$ are time-dependent, but simultaneously diagonalizable for all times, $t_i$.
In particular this allows for time-dependent reaction terms in the spatial operator, $\mathcal{L}$, and for variable 
time-step size, as occurs in, for example, adaptive time-stepping. To that end, 
\begin{align*}
\|I - A_\Delta B_\Delta^{-1}\|_{(UU^*)^{-1}} = \max_i &\left\|\left[I - A_\Delta B_\Delta^{-1}\right]_i\right\| = \frac{1}{\min_i \left\|\left[(I - A_\Delta B_\Delta^{-1})^\dagger\right]_i\right\|} \\ &= \frac{1}{\min_i \sigma_{max}\left(\left[(I - A_\Delta B_\Delta^{-1})^\dagger\right]_i\right)},
\end{align*}
where, recall, $\left[I - A_\Delta B_\Delta^{-1}\right]_i$ denotes \eqref{eq:gen_pinv} evaluated at eigenvalues of $\Phi$ and $\Psi$
as opposed to the actual operators. As previously, $\sigma_{max}\left(\left[(I - A_\Delta B_\Delta^{-1})^\dagger\right]_i\right)$
is given by the square root of the minimum nonzero eigenvalue of the corresponding normal residual equations (that is,
$MM^*$ as opposed to $M^*M$). Eliminating the final zero-row and zero-column (corresponding to the zero eigenvalue),
this is equivalent to the minimum eigenvalue of the tridiagonal matrix
\begin{align}\label{eq:tri_mat}
\begin{bmatrix}
\frac{1}{|\lambda_k^1-\mu_1|^2} & \frac{-\overline{\mu}_1}{|\lambda_k^1-\mu_1|^2} \\
\frac{-{\mu}_1}{|\lambda_k^1-\mu_1|^2} & \frac{|{\mu}_1|^2}{|\lambda_k^1-\mu_1|^2} + \frac{1}{|\lambda_{2k}^{k+1}-\mu_2|^2} & 
	\frac{-\overline{\mu}_2}{|\lambda_{2k}^{k+1}-\mu_2|^2} \\
& \frac{-{\mu}_2}{|\lambda_{2k}^{k+1}-\mu_2|^2} &  \frac{|{\mu}_2|^2}{|\lambda_{2k}^{k+1}-\mu_2|^2} + \frac{1}{|\lambda_{3k}^{2k+1}-\mu_3|^2}
	& \frac{-\overline{\mu}_3}{|\lambda_{3k}^{2k+1}-\mu_3|^2}\\ 
& & \ddots & \ddots & \ddots
\end{bmatrix}.
\end{align}
\textit{Note the change in notation} -- here, for example, $\lambda_k^1$ denotes the product of the $i$th eigenvalue of $\Phi$
at times $t_1,...,t_k$, that is, $\lambda_k^1 = \lambda_i(t_k)\lambda_i(t_{k-1})...\lambda_i(t_1)$, and $\mu_1$ denotes the $i$th
eigenvalue of $\Psi$ evaluated at time $t_1$. The $i$ is dropped from eigenvalues to limit subscript/superscript notation. 

This leads to the final result on convergence in the time-dependent case:

\begin{theorem}\label{th:nontoeplitz}
Let $\{\lambda_j\}^{(i)}$ and $\{\mu_j\}^{(i)}$ denote the sets of the $i$th eigenvalue of $\Phi$ and $\Psi$, respectively,
evaluated at time indices, $j=1,...,N_c-1$. Let $\hat{\sigma}_{\min}$ denote the minimum nonzero eigenvalue of \eqref{eq:tri_mat}.
Then,
\begin{align*}
\|I - A_\Delta B_\Delta^{-1}\|_{(UU^*)^{-1}} = \frac{1}{\sqrt{\hat{\sigma}_{\min}}}.
\end{align*}
Now, assume that for each eigenvalue index $i$, a TEAP-like approximation property holds, where,
for all $j=1,...,N_c-1$,
\begin{align*}
|\lambda_{jk}^{(j-1)k+1}-\mu_j|^2 \leq \widehat{\varphi}_j^{(i)} |1 - \mu_j|.
\end{align*}
Then,
\begin{align*}
\|I - A_\Delta B_\Delta^{-1}\|_{(UU^*)^{-1}}^2 & \leq \max_i \max\left\{ \widehat{\varphi}_1^{(i)}, \max_{j=1,...,N_c-2} \frac{\widehat{\varphi}_j^{(i)}\widehat{\varphi}_{j+1}^{(i)}}{ \widehat{\varphi}_j^{(i)} + \mu_j\widehat{\varphi}_{j+1}^{(i)}} \right\}.
\end{align*}
Moreover, sufficient conditions for convergence are that for all $i$ and for all $j$, 
\begin{align*}
\frac{\widehat{\varphi}_j^{(i)}\widehat{\varphi}_{j+1}^{(i)}}{ \widehat{\varphi}_j^{(i)} + \mu_j\widehat{\varphi}_{j+1}^{(i)}}  < 1,
\end{align*}
and, in addition, that for all $i$, $\widehat{\varphi}_1^{(i)} < 1$. 
\end{theorem}
\begin{proof}
The proof follows by applying the Gershgorin circles theorem to \eqref{eq:tri_mat} to bound the minimum eigenvalue
from below, and using this to bound the maximum singular value of $\left[I - A_\Delta B_\Delta^{-1}\right]_i$ from
above. For example, forming the Gershgorin disc for row one of \eqref{eq:tri_mat} yields a lower bound 
\begin{align*}
\frac{1 - |\mu_2|}{|\lambda_{2k}^{k+1}-\mu_2|^2} + \frac{|\mu_1|(1 - |\mu_1|)}{|\lambda_k^1-\mu_1|^2} & \geq
	 \frac{1}{\widehat{\varphi}_1^{(i)}} +  \frac{|\mu_1|}{\widehat{\varphi}_2^{(i)}}.
\end{align*}
Repeating for all rows and inverting yields the result. 
\end{proof}

Note from Section \ref{sec:diag} that in the diagonalizable case, Gershgorin is indeed asymptotically tight on the non-boundary
rows. The Gershgorin disc for row $i>1$ of \eqref{eq:almost_toe} bounds the minimum eigenvalue below by
$1+|\mu_i|^2-2|\mu_i| = (1 - |\mu_i|)^2$, which is exactly the minimum eigenvalue of \eqref{eq:almost_toe} to $O(1/N_c^2)$.
In practice, it is likely that the bound in Theorem \ref{th:nontoeplitz} is not tight, in particular the boundary term
$\widehat{\varphi}_1^{(i)}$, and that convergence will actually resemble $\widehat{\varphi}_j^{(i)}\widehat{\varphi}_{j+1}^{(i)}$.

Fortunately, although we cannot derive a closed form for eigenvalues of \eqref{eq:tri_mat}, a simple estimate of the convex
hull allows us to compute exact convergence bounds on two-level Parareal and MGRiT by solving a tridiagonal eigenvalue
problem, which is a computationally tractable task. This result allows for rigorous, problem-specific analysis
in some time-dependent cases, such as adaptive time-stepping.

\section{Conclusion}\label{sec:conc}

This paper derives necessary and sufficient conditions for the convergence of Parareal and MGRiT, assuming that time-stepping
operators $\Phi$ and $\Psi$ are linear and not time-dependent, and sufficient conditions for a subset of the general linear (time-dependent) case. This is
accomplished by introducing a temporal approximation
property (TAP), which gives a measure of how accurately $\Phi^k$ approximates the action of $\Psi$, for any vector $\mathbf{v}$. 
How accurately the TAP is satisfied then defines the $\ell^2$- and $A^*A$-norm of error reduction over successive iterates.
With further assumptions on the diagonalizability of $\Phi$ and $\Psi$, these results are strengthened to give tight bounds
on an arbitrary number of iterations. 

For space-time PDEs with a symmetric positive semi-definite (or symmetric negative semi-definite) spatial component,
the real eigenvalues of $\Phi$ and $\Psi$ can be explicitly computed as a function of time-step size, $\delta t$,
and eigenvalues of the spatial operator. With a simple estimate of the minimum and maximum eigenvalue of the
spatial operator, exact bounds on the convergence of Parareal and MGRiT can be easily computed by evaluating
the TEAP over the range of eigenvalues of $\Phi$ and $\Psi$. In the general case, for example, that arises in hyperbolic
PDEs, the eigenvectors no longer form an orthogonal basis and the TAP does not reduce to just considering eigenvalues.
However, for most time-stepping schemes applied to some operator $\mathcal{L}$, it is straightforward to expand
$\Phi$ and $\Psi$ in terms of $\mathcal{L}$. This permits a robust method to derive the expected convergence of
Parareal and MGRiT applied to a problem of the form $\mathbf{u}_t = \mathcal{L}\mathbf{u} + \mathbf{g}$, for
arbitrary $\mathcal{L}$. 

Further research regarding the optimal $\Psi$ with respect to $\Phi$, the difficulties in solving hyperbolic problems,
and the more general time-dependent and nonlinear cases, are ongoing work.

\section*{Acknowledgments} 
The author would like to thank Andreas Hessenthaler for bringing up convergence of MGRiT as a research
topic and providing useful numerical comparisons, Professor Tom Manteuffel for his helpful comments and discussions,
and Professor Stefano Serra-Capizzano for generously sharing his insight and expertise on block-Toeplitz matrix
theory. 

\bibliographystyle{siamplain}
\bibliography{main.bib}

\section*{Appendix A}

\begin{lemma}[Minimum eigenvalue of tridiagonal Toeplitz perturbation]\label{lem:tridiag}
Define the $n\times n$ tridiagonal matrix
\begin{align}\label{eq:rank1}
\mathcal{D}_i & = 
\begin{bmatrix}
1 + |\mu_i|^2 & -\overline{\mu}_i \\
-\mu_i & \ddots & \ddots \\
& \ddots  & 1 + |\mu_i|^2  & -\overline{\mu}_i \\
& & -\mu_i & 1  
\end{bmatrix}.
\end{align}
The minimum eigenvalue of $\mathcal{D}_i$, denoted $\lambda_n$, is bounded by
\begin{align}\label{eq:s_bounds}
\begin{split}
(1 - |\mu_i|)^2 + \frac{\pi^2|\mu_i|}{6n^2} & \leq 1 + |\mu_i|^2 + 2|\mu_i|\cos\left(\frac{n\pi}{n+1/2}\right)\\
	& \leq \lambda_n \\
	& \leq 1 + |\mu_i|^2 + 2|\mu_i|\cos\left(\frac{n\pi}{n+1}\right) \\
	& \leq (1 - |\mu_i|)^2 + \frac{\pi^2|\mu_i|}{n^2} 
\end{split}
\end{align}
\end{lemma}
\begin{proof}
Let $\widehat{\mathcal{D}}_i$ denote the self-adjoint, tridiagonal, Toeplitz matrix for which $\mathcal{D}_i$ is a rank-one
perturbation. In the scalar setting, there is a closed form for eigenvalues of a tridiagonal Toeplitz matrix of size $n\times n$,
given by
\begin{align}
\lambda\left(\widehat{\mathcal{D}}_i\right) & = \left\{ 1 + |\mu_i|^2 + 2|\mu_i|\cos\left(\frac{\ell\pi}{n+1}\right) \text{ $|$ }
	\ell = 1,...,n\right\}.\label{eq:eig}
\end{align}

Returning to \eqref{eq:rank1}, consider the rank-one perturbation in $\mathcal{D}_i$. Following from \cite{Yueh:2005ux}, the
spectrum of a tridiagonal Toeplitz matrix, and single-entry perturbations, is derived by building and solving a three-term recursion
relation. One of the general results in \cite{Yueh:2005ux}, Eq. (7), states that all eigenvalues of the matrices in \eqref{eq:rank1} take
the form
\begin{align}
\lambda = 1 + |\mu_i|^2 + 2 |\mu_i|\cos(\theta),\label{eq:geneig}
\end{align}
for a given $\theta \neq m\pi$, $m \in\mathbb{Z}$. In the case of a Toeplitz matrix, the necessary conditions on $\theta$ are
$\sin\left((n+1)\theta\right) = 0$, which is satisfied for $\widehat{\theta}_\ell = \frac{\ell\pi}{n+1}$, $\ell=1,...,n$, yielding the result in
\eqref{eq:eig}.  For the perturbation in $\mathcal{D}_i$, necessary conditions on $\theta$ are that (Eq. (6), \cite{Yueh:2005ux})
\begin{align}\label{eq:nec_theta}
T(\theta) := \sin\left((n+1)\theta\right) + |\mu_i|\sin(n\theta) = 0.
\end{align}
Unfortunately, \eqref{eq:nec_theta} does not have a closed-form solution as found in the Toeplitz case and several other
perturbations with closed form spectrum, introduced in \cite{Yueh:2005ux,Cheng:2008fm}.
However, each eigenvalue of $\mathcal{D}_i$
can be shown to be a small perturbation to eigenvalues of $\widehat{\mathcal{D}}_i$. Denote
$\left\{\widehat{\lambda}_\ell\right\}_{\ell=1}^n$ as the eigenvalues of $\widehat{\mathcal{D}}_i$, with corresponding
$\theta$-values $\left\{\widehat{\theta}_\ell\right\}_{\ell=1}^n$, and $\left\{{\lambda}_\ell\right\}_{\ell=1}^n$ the eigenvalues of
$\mathcal{D}_i$, with corresponding $\theta$-values $\left\{{\theta}_\ell\right\}_{\ell=1}^n$.

Consider $\widehat{\theta}_\ell = \frac{\ell\pi}{n+1}$, $\ell=1,...,n$, which yields all $n$ eigenvalues of $\widehat{\mathcal{D}}_i$,
in the context of necessary conditions for $\mathcal{D}_i$ \eqref{eq:nec_theta}:
\begin{align}
T\left(\widehat{\theta}_\ell\right) = |\mu_i|\sin\left(\frac{n}{n+1}\ell\pi\right) \mapsto
	\begin{cases} < 0 & 2|\ell \\ > 0 & 2\not | \ell\end{cases}.\label{eq:tk}
\end{align}
Now, define $\widetilde{\theta}_\ell = \frac{\ell\pi}{n+\frac{1}{2}}$ for $\ell=1,...,n$. Then, under the assumption that $|\mu_i| < 1$,
\begin{align}
T(\widetilde{\theta}_\ell) & = \sin\left(\frac{n+1}{n+\frac{1}{2}} \ell\pi\right) + |\mu_i|\sin\left(\frac{n}{n+\frac{1}{2}} \ell \pi\right)\nonumber \\
	& = - \sin\left(\frac{n}{n+\frac{1}{2}} \ell \pi\right) + |\mu_i|\sin\left(\frac{n}{n+\frac{1}{2}} \ell \pi\right)\nonumber \\
	& =  -(1 - |\mu_i|)\sin\left(\frac{n}{n+\frac{1}{2}} \ell \pi\right) \nonumber\\
	&  \mapsto \begin{cases} > 0 & 2|\ell \\ < 0 & 2\not | \ell\end{cases}.\label{eq:thk}
\end{align}
From \eqref{eq:tk}, \eqref{eq:thk}, and the continuity of $T(\theta)$, it follows that there exists
${\theta}_\ell\in\left(\frac{\ell\pi}{n+1}, \frac{\ell\pi}{n+\frac{1}{2}}\right)$, ${\theta}_\ell\neq m\pi, m\in\mathbb{Z}$,
such that $T({\theta}_\ell) = 0$, for $\ell=1,...,n$. Following from \eqref{eq:geneig}, eigenvalues of
$\mathcal{D}_i$ take the form
\begin{align*}
{\lambda}_\ell = 1+|\mu_i|^2 + 2|\mu_i|\cos\left({\theta}_\ell\right).
\end{align*}
The smallest nonzero eigenvalue of $\mathcal{D}_i$ is given by $\lambda_n = 1 + |\mu_i|^2 + 2|\mu_i|\cos(\theta_n)$, where
$\cos(\theta_\ell) \to -1$ as $\ell\to \infty$. Given $\frac{n\pi}{n+1} \leq \theta_n \leq \frac{n\pi}{n+1/2}$, $\lambda_n$ can then
be bounded by
\begin{align}\label{eq:ln_bound}
1 + |\mu_i|^2 + 2|\mu_i|\cos\left(\frac{n\pi}{n+1/2}\right) \leq \lambda_n & \leq 1 + |\mu_i|^2 + 2|\mu_i|\cos\left(\frac{n\pi}{n+1}\right).
\end{align}
With a little extra work, we can show that $\lambda_n = 1+|\mu_i|^2 + O(1/n^2)$, which leads to necessary and sufficient
conditions for convergence.

Consider the term $\cos\left(\frac{n\pi}{n+\frac{1}{2}}\right)$. As $n\to\infty$, $\cos\left(\frac{n\pi}{n+1}\right) \to^+ -1$,
that is, from above. Consider a series expansion of $f(n) = \cos\left(\frac{n\pi}{n+\frac{1}{2}}\right)$ about $n = \infty$. To 
accomplish this, apply the change of variable $n = \frac{1}{w}$ to get $f(w) = \cos\left(\frac{\frac{1}{w}\pi}{\frac{1}{w}+\frac{1}{2}}\right)$,
and expand about $w=0$. Formally, this would be expanded as a Laurent series, but recognizing that
$w^kf(w) = w^k\cos\left(\frac{\frac{1}{w}\pi}{\frac{1}{w}+\frac{1}{2}}\right) = w^k\cos\left(\frac{\pi}{1+\frac{w}{2}}\right)$
is holomorphic about $w=0$ for $k=0,1,...$, negative coefficients in the Laurent series are zero by the Cauchy Integral Theorem.
Our expansion reduces to a Taylor expansion for $\cos\left(\frac{\pi}{1+\frac{w}{2}}\right)$ about $w=0$,
\begin{align*}
\cos\left(\frac{\pi}{1+\frac{w}{2}}\right) = -1 + \frac{\pi^2w^2}{8} - \frac{\pi^3w^3}{8} + O(w^4).
\end{align*}
In fact, truncating the Taylor expansion at $k = 1$ leads to a remainder term
\begin{align}
\cos\left(\frac{\pi}{1+\frac{w}{2}}\right) & = -1 + \int_{0}^w (w-t)\frac{\partial}{\partial_{tt}}\left[\cos\left(\frac{\pi}{1+\frac{t}{2}}\right) \right]dt \nonumber\\
& = -1 + \int_{0}^w (w-t) \frac{-4\pi\left[\pi\cos\left(\frac{\pi}{1+\frac{t}{2}}\right) + (2+t)\sin\left(\frac{\pi}{1+\frac{t}{2}}\right)\right]}{(2+t)^4} dt.\label{eq:rem}
\end{align}
Note that 
\begin{align*}
0 < \frac{\pi^2}{6} < \frac{-4\pi\left[\pi\cos\left(\frac{\pi}{1+\frac{t}{2}}\right) +
	(2+t)\sin\left(\frac{\pi}{1+\frac{t}{2}}\right)\right]}{(2+t)^4} \leq \frac{\pi^2}{4},
\end{align*}
for all $t \in[0,1/10]$ (this range is not tight, just sufficient for our purposes). By positivity of the two terms being integrated in
the remainder \eqref{eq:rem}, the remainder can be bounded above and below. For $w\in[0,1/10]$,
substituting in for $n\geq 10$ yields
\begin{align}
-1 + \frac{\pi^2}{12n^2} \leq \cos\left(\frac{n\pi}{n+\frac{1}{2}}\right) \leq -1 + \frac{\pi^2}{8n^2}.\label{eq:np12}
\end{align}

A similar expansion on $\cos\left(\frac{n\pi}{n+1}\right)$ yields a truncated Taylor expansion and remainder given by
\begin{align*}
\cos\left(\frac{\pi}{w+1}\right) & = -1 + \int_{0}^w (w-t) \frac{-\pi\left[\pi\cos\left(\frac{\pi}{t+1}\right) + 2(1+t)\sin\left(\frac{\pi}{t+1}\right)\right]}{(1+t)^4} dt.
\end{align*}
Here, note that
\begin{align}
0 < \frac{\pi^2}{2} < \frac{-\pi\left[\pi\cos\left(\frac{\pi}{t+1}\right) + 2(1+t)\sin\left(\frac{\pi}{t+1}\right)\right]}{(1+t)^4} \leq \pi^2,\label{eq:rem_bound}
\end{align}
for $t \in[0,1/10]$. Integrating the product of two positive functions, the bounds in \eqref{eq:rem_bound} can be pulled out and, 
for $w\in[0,1/10]$,
substituting in for $n \geq 10$ yields
\begin{align*}
-1 + \frac{\pi^2}{4n^2} \leq \cos\left(\frac{n\pi}{n+1}\right) \leq -1 + \frac{\pi^2}{2n^2}.
\end{align*}

Returning to the upper and lower bounds on $\lambda_n$ in \eqref{eq:ln_bound} completes the proof.
\end{proof}

\end{document}